\documentclass[aihp,preprint]{imsart}

\RequirePackage[OT1]{fontenc}
\RequirePackage[colorlinks,citecolor=blue,urlcolor=blue]{hyperref}

\usepackage{amsmath,amsfonts,amsbsy,amsgen,amscd,mathrsfs,amssymb,amsthm,mathtools}

\usepackage{accents}

\usepackage{pgfplots}
\usepackage{tikz-cd}
\usetikzlibrary{arrows,automata}

\startlocaldefs

\numberwithin{equation}{section}
\newtheorem{theorem}{Theorem}[section]
\newtheorem{corollary}[theorem]{Corollary}
\newtheorem{conjecture}[theorem]{Conjecture}
\newtheorem{lemma}[theorem]{Lemma}
\newtheorem{proposition}[theorem]{Proposition}
\theoremstyle{definition}
\newtheorem{assumption}[theorem]{Assumption}
\newtheorem{definition}[theorem]{Definition}

\newtheorem{notation}[theorem]{Notation}
\newtheorem{question}[theorem]{Question}
\newtheorem{remark}[theorem]{Remark}

\newcommand\al{\alpha}
\newcommand\be{\beta}
\newcommand\dd{\mathrm d}
\newcommand\De{\Delta}
\newcommand\de{\delta}
\newcommand\deq{\stackrel{\mathrm d}{=}}
\newcommand\eps{\varepsilon}
\newcommand\Ga{\Gamma}
\newcommand\ga{\gamma}
\newcommand\ka{\kappa}
\newcommand\La{\Lambda}
\newcommand\la{\lambda}

\newcommand\Si{\Sigma}
\newcommand\si{\sigma}
\newcommand\ubar[1]{%
\underaccent{\bar}{#1}}
\newcommand\ze{\zeta}
\renewcommand\d{~\mathrm d}
\renewcommand\phi{\varphi}
\renewcommand\rho{\varrho}
\renewcommand\th{\vartheta}

\newcommand\bs{\boldsymbol}
\newcommand\mbb{\mathbb}
\newcommand\mbf{\mathbf}
\newcommand\mc{\mathcal}
\newcommand\mf{\mathfrak}
\newcommand\mr{\mathrm}

\endlocaldefs

\begin{document}

\begin{frontmatter}

\title{On the Convergence of Random Tridiagonal
Matrices to Stochastic Semigroups}
\runtitle{Convergence of Tridiagonal Matrices to Stochastic Semigroups}

\begin{aug}
\author{\fnms{Pierre Yves} \snm{Gaudreau Lamarre}\corref{}\thanksref{a}\ead[label=e1]{plamarre@princeton.edu}}
\address[a]{Princeton University, Princeton, NJ\space\space08544, USA.\\
\printead{e1}}

\runauthor{Pierre Yves Gaudreau Lamarre}

\affiliation{Princeton University}

\end{aug}

\begin{abstract}
We develop an improved version of the stochastic semigroup approach to study the edge
of $\be$-ensembles pioneered by Gorin and Shkolnikov \cite{GorinShkolnikov}, and later extended to rank-one additive
perturbations by the author and Shkolnikov \cite{GaudreauLamarreShkolnikov}.
Our method is applicable to
a significantly more general class of random tridiagonal matrices than that considered in \cite{GaudreauLamarreShkolnikov,GorinShkolnikov},
including some non-symmetric cases that are not covered by the stochastic operator formalism of Bloemendal, Ram\'irez,
Rider, and Vir\'ag \cite{BloemendalVirag,RamirezRiderVirag}.

We present two applications of our main results: Firstly, we prove the convergence of $\be$-Laguerre-type
(i.e., sample covariance) random tridiagonal matrices to the stochastic Airy semigroup and its rank-one
spiked version. Secondly, we prove the convergence of the eigenvalues
of a certain class of non-symmetric random tridiagonal matrices to the spectrum of a
continuum Schr\"odinger operator with Gaussian white noise potential.
\end{abstract}

\begin{abstract}[language=french]
Nous d\'eveloppons une version am\'elior\'ee de l'approche de {\it stochastic semigroup} pour \'etudier
l'extr\'emit\'e des ensembles b\^eta introduite par Gorin et Shkolnikov \cite{GorinShkolnikov},
ensuite \'etendue
aux ensembles b\^eta gaussiens avec perturbation de rang un par
l'auteur et Shkolnikov \cite{GaudreauLamarreShkolnikov}.
Notre m\'ethode est applicable \`a
une classe nettement plus g\'en\'erale de matrices tridiagonales al\'eatoires que celles dans \cite{GaudreauLamarreShkolnikov,GorinShkolnikov},
y compris certains cas non sym\'etriques qui ne sont pas couverts par la m\'ethode de {\it stochastic operators}
introduite par Bloemendal, Ram\'irez,
Rider et Vir\'ag \cite{BloemendalVirag, RamirezRiderVirag}.

Nous pr\'esentons deux applications de nos principaux r\'esultats :
Premi\`erement, nous prouvons la convergence
de matrices tridiagonales al\'eatoires
de type $\be$-Laguerre (c.-\`a-d., matrices de covariances empiriques)
vers le semi-groupe du {\it stochastic Airy operator} et sa
perturbation de rang un.
Deuxi\`emement, nous prouvons la convergence des valeurs propres
d'une certaine classe de matrices tridiagonales al\'eatoires non sym\'etriques vers
le spectre
d'op\'erateurs de Schr\"odinger avec bruit blanc gaussien.
\end{abstract}




\end{frontmatter}

\section{Introduction}

\subsection{Operator Limits of Random Matrices}

This paper, which is a direct sequel of \cite{GaudreauLamarreShkolnikov,GorinShkolnikov}, is concerned with
operator limits of random matrices. The theory of operator limits was initiated in \cite{DumitriuEdelman,
EdelmanSutton,RamirezRiderVirag} and eventually gave rise to a vast literature on the subject. We refer to the
survey article \cite{Virag} for a recent historical account of these early developments.

A fundamental object in this theory is the {\bf stochastic Airy operator}, formally defined as
\[\mr{SAO}_\be f(x):=-f''(x)+xf(x)+W_\be'(x)f(x),\qquad f:\mbb R_+\to\mbb R,\]
where $\be>0$ is fixed parameter,
$W_\be$ is a Brownian motion with variance $4/\be$,
$\mbb R_+:=[0,\infty)$,
and $f$ obeys a Dirichlet or Robin boundary condition at the origin.
We refer to \cite[Section 2.3]{BloemendalVirag}, \cite[Section 2]{Minami}, and \cite[Section 2]{RamirezRiderVirag}
for a rigorous definition.

The interest of studying $\mr{SAO}_\be$ comes from the fact that its eigenvalue
point process captures the asymptotic edge fluctuations of a large class of random matrices
and interacting particle systems. In \cite{BloemendalVirag,RamirezRiderVirag},
this was proved for the $\be$-Hermite ensemble, the $\be$-Laguerre ensemble (for the right edge),
as well as rank-one perturbations of the $\be$-Hermite and $\be$-Laguerre ensembles
(the spiked models). Then, \cite{KrishnapurRiderVirag}
established operator limits as a means of
proving edge universality for general $\be$-ensembles (c.f., \cite{BourgadeErdosYau}).
More generally, \cite{BloemendalVirag,RamirezRiderVirag} proved the eigenvalue and eigenvector convergence of
a wide class of symmetric random tridiagonal matrices to the spectrum of Schr\"odinger operators of the form $-\De+Y'$,
where $Y$ is a random function.

\subsection{Stochastic Semigroups}

More recently,
Gorin and Shkolnikov
introduced in \cite{GorinShkolnikov} a new method of studying edge fluctuations of $\be$-ensembles.
Their main result was that high powers of a generalized version of the $\be$-Hermite ensemble converge to
a random Feynman-Kac-type semigroup that was dubbed the {\bf stochastic Airy semigroup}
(\cite[Theorem 2.1]{GorinShkolnikov}), which we denote
by $\mr{SAS}_\be(t)$ for $\be,t>0$ (see Definition \ref{Definition: Continuum Limit}
and Notation \ref{Notation: SAS}).

Combining their result with the fact that the 
edge-rescaled $\be$-Hermite ensemble converges to $\mr{SAO}_\be$, Gorin and Shkolnikov concluded that
$\mr{SAS}_\be(t)=\mr e^{-t\,\mr{SAO}_\be/2}$ for all $t>0$
(\cite[Corollary 2.2]{GorinShkolnikov}),
thus providing a new tool with which
$\mr{SAO}_\be$'s spectrum can be studied. As a demonstration of this, it was shown in
\cite[Corollary 2.3 and Proposition 2.6]{GorinShkolnikov} that certain statistics
of $\mr{SAS}_\be(t)$ admit an especially simple form when $\be=2$. Among other things, this provided the first
manifestation of the special integrable structure in the $\be$-ensembles when $\be\in\{1,2,4\}$
at the level of the operator limits describing edge fluctuations. These results were extended to
rank-one spiked $\be$-Hermite models in
\cite{GaudreauLamarreShkolnikov}.
Feynman-Kac formulas for general one-dimensional Schr\"odinger operators with
multiplicative Gaussian noise were obtained more recently in \cite{GaudreauLamarre}.

\subsection{Overview of Main Results}

In this paper, we introduce a modification of the formalism
developed in \cite{GaudreauLamarreShkolnikov,GorinShkolnikov}. Our main results
(Theorems \ref{Theorem: Main Dirichlet} and \ref{Theorem: Main Robin}) establish the convergence
of high powers of a large class of random tridiagonal matrices to the semigroups of continuum Schr\"odinger
operators with Gaussian white noise. Our results improve on \cite{GaudreauLamarreShkolnikov,GorinShkolnikov}
 and \cite{BloemendalVirag,RamirezRiderVirag} in two significant ways.

Firstly, a main technical achievement of \cite{GorinShkolnikov} was to show that the
moment method can be used to study edge fluctuations of $\be$-ensembles for $\be\not\in\{1,2,4\}$.
The key to achieving this is to relate the combinatorics of traces of high powers of tridiagonal matrices to strong invariance
principles for random walks and their occupation measures (\cite[Section 3]{GorinShkolnikov} and 
\cite[Section 3.1]{GaudreauLamarreShkolnikov}). A notable feature of the combinatorial
analysis in \cite{GorinShkolnikov} is that it requires the tridiagonal matrices under consideration to have
diagonal entries of smaller order than their super/sub-diagonal entries (see Section \ref{Subsection: Brief Comparison}
for details). In particular, this argument is not directly applicable to the $\be$-Laguerre
ensemble.
In this context, one contribution of this paper is to
develop an improved version of the stochastic semigroup formalism that does not have restrictions on
the relative size of  diagonal/off-diagonal entries.
As a demonstration of this, we prove that our main results apply to
every matrix model considered in \cite{GaudreauLamarreShkolnikov,GorinShkolnikov}, as well as generalized
$\be$-Laguerre ensembles (Section \ref{Section: Classical Beta}).

Secondly, a notable feature of our results is that they appear to be the first to apply to non-symmetric matrices.
As a consequence, we prove new limit laws for the eigenvalues of certain non-symmetric random
tridiagonal matrices (Propositions \ref{Proposition: Convergence of Eigenvalues} and \ref{Proposition: Non-Symmetric Eigenvalues}).
In particular, we identify a new matrix model whose edge fluctuations are in the
Tracy-Widom universality class (Corollary \ref{Corollary: Non Symmetric Example}).
These results complement previous investigations on the spectrum of
non-symmetric random tridiagonal matrices, such as \cite{GoldsheidSodin,GoldsheidKhoruzhenko2,
GoldsheidKhoruzhenko3,GoldsheidKhoruzhenko}.

Several features of the strategy of proof in \cite{GaudreauLamarreShkolnikov,GorinShkolnikov} for analyzing
the combinatorics of large powers of tridiagonal matrices carry over to this paper. For instance, strong invariance principles
for occupation measures of random walks also play a fundamental role in our proofs. That being said, the differences are
significant enough that many nontrivial modifications and new ideas need to be introduced. Most notably, 
several results in the literature concerning strong approximations of Brownian local time that are used without
modification in \cite{GaudreauLamarreShkolnikov,GorinShkolnikov} require significant work to be applicable to our setting
(Sections \ref{Section: Dirichlet Local Time Coupling} and \ref{Section: Spiked Local Time Coupling}).

\subsection{Organization}

In Section \ref{Section: Main Results}, we introduce our random
matrix models, their continuum limits,
and we state our main results.
In Section \ref{Section: Applications}, we discuss applications of our main
results to random matrices.
In Section \ref{Section: Method}, we explain the main idea in our strategy of proof,
and we make a brief comparison with the method of \cite{GaudreauLamarreShkolnikov,GorinShkolnikov}.
In Sections \ref{Section: Dirichlet Local Time Coupling} and \ref{Section: Spiked Local Time Coupling},
we prove two local time strong invariance results that lie at the heart of our proof.
Finally, in Sections \ref{Section: Main Dirichlet 1}, \ref{Section: Main Dirichlet 2},
and \ref{Section: Main Robin}, we complete the proofs of our main results.

\subsection*{Acknowledgments}

The author thanks Mykhaylo Shkolnikov for his continuous guidance and support,
and for multiple discussions regarding this paper. The author thanks
Vadim Gorin, Diane Holcomb and anonymous referees for multiple valuable comments that
helped greatly improve the presentation of the paper.

\section{Setup and Main Results}\label{Section: Main Results}

\subsection{Random Matrix Models}

We begin by introducing our random matrix models.
Let $(m_n)_{n\in\mbb N}$
be a sequence of positive numbers
and $(w_n)_{n\in\mbb N}$ be a sequence of real numbers such that
the following holds:

\begin{assumption}\label{Assumption: mn}
There exists $0<C\leq 1$ and $1/13<\mf d<1/2$ such that
\begin{align}\label{Equation: mn}
Cn^{\mf d}\leq m_n\leq C^{-1}n^{\mf d},\qquad n\in\mbb N.
\end{align}
\end{assumption}

\begin{assumption}\label{Assumption: wn}
There exists some $w\in\mbb R$ such that
\begin{align}\label{Equation: wn}
\lim_{n\to\infty}m_n(1-w_n)=w.
\end{align}
\end{assumption}

For every $n\in\mbb N$, let us define the $(n+1)\times (n+1)$ tridiagonal matrices $\De_n,\De^w_n$, and $Q_n$ as
\[\De_n:=m_n^2\left[\begin{array}{ccccc}
-2&1\\
1&\ddots&\ddots\\
&\ddots&\ddots&1\\
&&1&-2
\end{array}\right],
\qquad
 Q_n:=\left[\begin{array}{ccccc}
D_n(0)&U_n(0)\\
L_n(0)&\ddots&\ddots\\
&\ddots&\ddots&U_n(n-1)\\
&&L_n(n-1)&D_n(n)
\end{array}\right],
\]
\[\De_n^w:=\De_n+\mr{diag}_n(m_n^2w_n,0,\ldots,0)\]
where $D_n(a),U_n(a),L_n(a)$ are real-valued random variables
for every $n\in\mbb N$ and $0\leq a\leq n$ (or $0\leq a\leq n-1$).

\begin{notation}
Throughout, we index the entries of a $(n+1)\times(n+1)$ matrix $ M$ as
$ M(a,b)$ for $0\leq a,b\leq n$. Similarly, $v\in\mbb R^{n+1}$ is indexed as
$v(a)$ for $0\leq a\leq n$. We use $\mr{diag}_n(d_0,\ldots,d_n)$
to denote the $(n+1)\times(n+1)$ diagonal matrix $M$ with entries $M(a,a)=d_a$
for $0\leq a\leq n$.
\end{notation}
\begin{notation}
For simplicity, we often state
properties of $D_n(a),U_n(a),L_n(a)$ for $0\leq a\leq n$, with the understanding
that $a\leq n-1$ for $U_n(a)$ and $L_n(a)$.
\end{notation}

We assume that the entries of $ Q_n$ satisfy the following decomposition:
For $E\in\{D,U,L\}$,
\begin{align}
\label{Equation: Potential+Noise Decomposition}
E_n(a)&=V^E_n(a)+\xi^E_n(a),&0\leq a\leq n,
\end{align}
where the $V^E_n(a)$ are deterministic and the $\xi^E_n(a)$ are random.
We call $V^E_n$ the {\bf potential terms} and $\xi^E_n$ the {\bf noise
terms}. The random matrix models studied in this paper are as follows.

\begin{definition}[Random Matrix Models]\label{Definition: Random Matrix Models}
For every $n\in\mbb N$ and $t>0$, we define
\begin{align}
\label{Equation: Matrix Models}
\hat K_n(t):=\left(I_n-\frac{-\De_n+Q_n}{3m_n^2}\right)^{\lfloor m_n^2(3t/2)\rfloor},\qquad
\hat K^w_n(t):=\left(I_n-\frac{-\De^w_n+Q_n}{3m_n^2}\right)^{\lfloor m_n^2(3t/2)\rfloor}.
\end{align}
\end{definition}

\subsection{Continuum Limit}

We now describe the continuum limits of \eqref{Equation: Matrix Models}.
In order to describe these objects, we need some notations:

\begin{notation}
We use $B$ to denote a standard Brownian motion on $\mbb R$,
and $X$ to denote a standard reflected Brownian motion on $\mbb R_+$.

Let $Z=B$ or $X$. For every $t>0$ and $x,y\geq0$,
we denote
\[Z^x:=\big(Z|Z(0)=x\big)
\qquad\text{and}\qquad
Z^{x,y}_t:=\big(Z|Z(0)=x\text{ and }Z(t) = y\big),\]
and we use $\mbf E^x$ and $\mbf E^{x,y}_t$ to denote the expected value with respect to the law of $Z^x$
and $Z^{x,y}_t$ respectively.

For any $t>0$, we use $x\mapsto L^x_t(Z)$ to denote the continuous version
of the local time process of $Z$ on $[0,t]$, which we characterize by the requirement that
for every measurable function $f$, one has
\begin{align}\label{Equation: Interior Local Time}
\int_0^tf\big(Z(s)\big)\d s=\int_{\mbb R}L^x_t(Z)\, f(x)\d x.
\end{align}
As a matter of convention, in the case where $Z=X$, we distinguish the boundary local time $\mf L^0_t(Z)$ from the above, which we define as
\begin{align}\label{Equation: Boundary Local Time}
\mf L^0_t(Z):=\lim_{\eps\to0}\frac{1}{2\eps}\int_0^t\mbf 1_{\{0\leq Z(s)<\eps\}}\d s.
\end{align}
Finally, we let $\tau_0(B)$ denote the first hitting time of zero by $B$.
\end{notation}

\begin{definition}[Continuum Limits]\label{Definition: Continuum Limit}
Let $Q$ be the diffusion process
\[\dd Q(x)=V(x)\dd x+\dd W(x),\qquad x\geq0,\]
where $V\geq0$ is a deterministic locally integrable function on $\mbb R_+$,
and $W$ is a Brownian motion with variance $\si^2>0$.
For every $t>0$, we let
$\hat K(t)$ and $\hat K^w(t)$ be the integral operators on $L^2(\mbb R_+)$
with random kernels
\begin{align}
\label{Equation: Dirichlet Semigroup}
\hat K(t;x,y)&:=\frac{\mr e^{-(x-y)^2/2t}}{\sqrt{2\pi t}}\mbf E^{x,y}_t\left[\mbf 1_{\{\tau_0(B)>t\}}\mr e^{-\langle L_t(B),Q'\rangle}\right]\\
\label{Equation: Robin Semigroup}
\hat K^w(t;x,y)&:=\left(\frac{\mr e^{-(x-y)^2/2t}}{\sqrt{2\pi t}}+\frac{\mr e^{-(x+y)^2/2t}}{\sqrt{2\pi t}}\right)\mbf E^{x,y}_t\left[\mr e^{-\langle L_t(X),Q'\rangle-w\mf L^0_t(X)}\right]
\end{align}
for $x,y\geq0$, where
\begin{enumerate}
\item we assume that $B$ and $X$ are independent of $W$,
and that $\mbf E^{x,y}_t$ is the conditional expected value of $B^{x,y}_t$ or $X^{x,y}_t$
given $W$; and
\item for any piecewise continuous and compactly supported function $f$,
\[\langle f,Q'\rangle:=\int_\mbb Rf(x)\d Q(x)\]
denotes $\dd Q$ pathwise stochastic integration (see \cite[Remark 2.18]{GaudreauLamarre}).
\end{enumerate}
\end{definition}

\begin{remark}\label{Remark: Strong Continuity}
Consider the operator $\hat H:=-\tfrac12\De+V+W'$ acting on $\mbb R_+$
with Dirichlet boundary condition at zero, and let $\hat H^w$
be the same operator but with Robin boundary condition $f'(0)=wf(0)$.
If the function $V$ satisfies
\begin{align}\label{Equation: Log Potential}
\lim_{x\to\infty}V(x)/\log x=\infty,
\end{align}
then $\hat H$ and $\hat H^w$ can be rigorously defined as self-adjoint
operators with compact resolvent (and thus discrete spectrum)
using quadratic forms (\cite[Proposition 2.9 and Corollary 2.12]{GaudreauLamarre};
see also \cite{BloemendalVirag,Minami,RamirezRiderVirag}).
According to \cite[Theorem 2.23]{GaudreauLamarre}, for every
$t>0$, it holds with probability one that $\hat K(t)$ and $\hat K^w(t)$ are self-adjoint
Hilbert-Schmidt operators on $L^2(\mbb R_+)$,
and $\hat K(t)=\mr e^{-t\hat H}$ and $\hat K^w(t)=\mr e^{-t\hat H^w}$.
We also have the trace formula $\mr{Tr}[\hat K(t)]=\int_0^\infty\hat K(t;x,x)\d x<\infty.$
\end{remark}

\begin{notation}\label{Notation: SAS}
Let $\be>0$. If $V(x)=x/2$ and $\si^2=1/\be$ in Definition \ref{Definition: Continuum Limit},
then we use the notation $\mr{SAS}_\be(t):=\hat K(t)$ and $\mr{SAS}^w_\be(t):=\hat K^w(t)$,
since in this case we recover the stochastic Airy semigroup defined in
\cite{GaudreauLamarreShkolnikov,GorinShkolnikov}, which is the semigroup of
the stochastic Airy operator.
\end{notation}

\subsection{Technical Assumptions}\label{Subsection: Technical Assumptions}

We are now finally in a position to state our main results and the assumptions under which they apply.
We begin with the assumptions on the random entries of $Q_n$
in \eqref{Equation: Potential+Noise Decomposition}; our theorems are stated in Section \ref{Subsection: Theorems}.

\subsubsection{Assumptions on the Potential Terms $V_n^E$}

\begin{assumption}[Potential Convergence]\label{Assumption: Potential Convergence}
There exists nonnegative continuous functions
$V^D,V^U,V^L:\mbb R_+\to\mbb R$ such that
\[\lim_{n\to\infty}V^E_n(\lfloor m_nx\rfloor)=V^E(x),\qquad x\geq0\]
uniformly on compact sets for every $E\in\{D,U,L\}$.
Moreover, the function
\begin{align}\label{Equation: Combination of Potentials}
V:=\tfrac12(V^D+V^U+V^L),\qquad x\geq0
\end{align}
satisfies \eqref{Equation: Log Potential}.
\end{assumption}

\begin{assumption}[Growth Upper Bounds]\label{Assumption: Growth Upper Bounds}
For every $E\in\{D,U,L\}$ we have the following:
For large enough $n$,
\begin{align}\label{Equation: Potential Absolute Bounds}
0\leq V^E_n(a)\leq 2 m_n^2,\qquad 0\leq a\leq n,
\end{align}
and if $C_n=o(n)$ as $n\to\infty$, then
\begin{align}\label{Equation: Potential Growth Upper Bound}
\max_{a\leq C_n}V^E_n(a)=o(m_n^2),\qquad n\to\infty.
\end{align}
\end{assumption}

\begin{assumption}[Growth Lower Bounds]\label{Assumption: Growth Lower Bounds}
 At least one of $E\in\{D,U,L\}$ satisfies the following:
 For every $\theta>0$, there exists $c=c(\theta)>0$ and $N=N(\theta)\in\mbb N$ such that
 for every $n\geq N$,
 \begin{align}\label{Equation: Log Potential Lower Bound}
\theta\log(1+a/m_n)-c\leq V^E_n(a)\leq m_n^2,\qquad 0\leq a\leq n.
\end{align}
 Moreover, at least one of $E\in\{D,U,L\}$ (not necessarily the same as
 \eqref{Equation: Log Potential Lower Bound}) satisfies the following:
With $\mf d$ as in \eqref{Equation: mn},
there exists $\mf d/2(1-\mf d)<\al\leq2\mf d/(1-\mf d)$, $\eps>0$, and positive constants
$\ka$ and $C>0$ such that
\begin{align}\label{Equation: Polynomial Potential Lower Bound}
\ka(a/m_n)^\al\leq V^E_n(a)\leq m_n^2,\qquad Cn^{1-\eps}\leq a\leq n
\end{align}
for $n$ large enough.
\end{assumption}

\subsubsection{Assumptions on the Noise Terms $\xi_n^E$}

\begin{assumption}[Independence]\label{Assumption: Independence}
For every $n\in\mbb N$, the variables $\xi^D_n(0),\ldots,\xi^D_n(n)$
are independent, and likewise for $\xi^U_n(0),\ldots,\xi^U_n(n-1)$
and $\xi^L_n(0),\ldots,\xi^L_n(n-1)$. We emphasize, however, that
the random vectors
$\xi_n^D$, $\xi_n^U$, and $\xi_n^L$ need not be independent of each other
(for instance, if $Q_n$ is symmetric, then $\xi^U_n=\xi^L_n$).
\end{assumption}

\begin{assumption}[Moment Asymptotics]\label{Assumption: Moments}
For every $E\in\{D,U,L\}$, we have:
\begin{align}\label{Equation: Noise Expected Decay}
|\mbf E[\xi_n^E(a)]|=o\big(m_{n-a}^{-1/2}\big)\qquad\text{as }(n-a)\to\infty,
\end{align}
and there exists constants $C>0$ and $0<\ga<2/3$ such that
\begin{align}\label{Equation: Noise Moment Bound}
\mbf E\big[|\xi^E_n(a)|^q\big]\leq m_n^{q/2}C^qq^{\ga q}
\end{align}
for every $0\leq a\leq n$, integer $q\in\mbb N$,
and $n$ large enough.
\end{assumption}

\begin{assumption}[Noise Convergence]\label{Assumption: Noise Convergence}
There exists Brownian motions $W^D$, $W^U$, and $W^L$ such that
\begin{align}\label{Equation: Skorokhod Noise Terms}
\lim_{n\to\infty}\bigg(\frac1{m_n}\sum_{a=0}^{\lfloor m_n x\rfloor}\xi^E_n(a)\bigg)_{E=D,U,L}=
\Big(W^E(x)\Big)_{E=D,U,L},\qquad x\geq0
\end{align}
in joint distribution with respect to the Skorokhod topology. We assume that
\begin{align}\label{Equation: Combination of BM}
W:=\tfrac12\big(W^D+W^U+W^L\big)
\end{align}
is also a Brownian motion with some variance $\si^2>0$.
Furthermore, if $\phi_1,\ldots,\phi_k$ are continuous and compactly supported functions
and $(\phi^{(n)}_1)_{n\in\mbb N},\ldots,(\phi^{(n)}_k)_{n\in\mbb N}$ are such that $\phi^{(n)}_i\to\phi_i$
uniformly for every $1\leq i\leq k$, then
\begin{align}\label{Equation: Noise Stochastic Integral Convergence}
\lim_{n\to\infty}\bigg(\sum_{a\in\mbb N_0}\phi^{(n)}_i(a/m_n)\frac{\xi^E_n(a)}{m_n}\bigg)_{E=D,U,L;~1\leq i\leq k}
=\bigg(\int_{\mbb R_+}\phi_i(a)\d W^E(a)\bigg)_{E=D,U,L;~1\leq i\leq k}
\end{align}
in joint distribution, and also jointly with \eqref{Equation: Skorokhod Noise Terms}.
\end{assumption}

\subsubsection{Assumptions for the Robin Boundary Condition}

The following assumption will only be made when considering $\hat K^w_n(t)$:

\begin{definition}
We say that a sequence $(X_n)_{n\in\mbb N}$ is uniformly sub-Gaussian
if there exists $C,c>0$ independent of $n$ such that
\begin{align}
\label{Equation: Uniform SubGaussian}
\sup_{n\in\mbb N}\mbf E\left[\mr e^{y\,|X_n|}\right]\leq C\mr e^{cy^2},\qquad y\geq0.
\end{align}
\end{definition}

\begin{assumption}\label{Assumption: Robin}
$\big(D_n(0)/m_n^{1/2}\big)_{n\in\mbb N}$ is uniformly sub-Gaussian.
\end{assumption}

\begin{remark}
If $\ga<1/2$ in \eqref{Equation: Noise Moment Bound},
then Assumption \ref{Assumption: Robin} is satisfied.
\end{remark}

\subsection{Main Theorems}\label{Subsection: Theorems}

\begin{notation}
In order to make sense of the claim that $\hat K_n(t)\to\hat K(t)$
and $\hat K_n^w(t)\to\hat K^w(t)$, we need to ensure that the discrete
and continuous objects act on the same space. For this purpose,
we note that the action of the matrices \eqref{Equation: Matrix Models} on $\mbb R^{n+1}$
can naturally be extended to step functions on $\mbb R_+$ of the form
\[\textstyle\sum_{a=0}^nv(a)\,\mbf 1_{[a/m_n,(a+1)/m_n)}\qquad\text{for some $v\in\mbb R^{n+1}$.}\]
This can then be further extended to any locally integrable $f:\mbb R_+\to\mbb R$
via
\begin{align}
\label{Notation: Embedding}
\pi_nf:=m_n^{1/2}\sum_{a=0}^n\int_{a/m_n}^{(a+1)/m_n}f(x)\d x\,\mbf 1_{[a/m_n,(a+1)/m_n)}.
\end{align}
Thus, for any $(n+1)\times(n+1)$ matrix $ M$ and locally integrable functions $f,g$, we define
$Mf$ as the vector/step function $M(\pi_nf)$, and we define
\[\langle f, Mg\rangle:=m_n\sum_{0\leq a,b\leq n}\left(\int_{a/m_n}^{(a+1)/m_n}f(x)\d x\right)M(a,b)\left(\int_{b/m_n}^{(b+1)/m_n}g(x)\d x\right).\]
\end{notation}

Our limit results are as follows.

\begin{theorem}\label{Theorem: Main Dirichlet}
Suppose that Assumptions \ref{Assumption: mn}
and \ref{Assumption: Potential Convergence}--\ref{Assumption: Noise Convergence} hold.
Let $ \hat K(t)$ be defined as in \eqref{Equation: Dirichlet Semigroup},
where $V$ is given by \eqref{Equation: Combination of Potentials} and $W$ is given by \eqref{Equation: Combination of BM}.
Then, $ \hat K_n(t)\to \hat K(t)$ as $n\to\infty$ in the following two senses:
\begin{enumerate}
\item For every $t_1,\ldots,t_k>0$ and $f_1,g_1,\ldots,f_k,g_k:\mbb R_+\to\mbb R$ uniformly continuous and bounded,
\[\lim_{n\to\infty}\big(\langle f_i, \hat K_n(t_i)g_i\rangle\big)_{1\leq i\leq k}=\big(\langle f_i, \hat K(t_i)g_i\rangle\big)_{1\leq i\leq k}\]
in joint distribution and mixed moments.
\item For every $t_1,\ldots,t_k>0$,
\[\lim_{n\to\infty}\big(\mr{Tr}[ \hat K_n(t_i)]\big)_{1\leq i\leq k}=\big(\mr{Tr}[ \hat K(t_i)]\big)_{1\leq i\leq k}\]
in joint distribution and mixed moments.
\end{enumerate}
\end{theorem}

\begin{theorem}\label{Theorem: Main Robin}
Suppose that Assumptions \ref{Assumption: mn}, \ref{Assumption: wn},
and \ref{Assumption: Potential Convergence}--\ref{Assumption: Robin} hold.
Let $\hat K^w(t)$ be defined as in \eqref{Equation: Robin Semigroup},
where $V$ is given by \eqref{Equation: Combination of Potentials} and $W$ is given by \eqref{Equation: Combination of BM}.
Then, $ \hat K^w_n(t)\to \hat K^w(t)$ as $n\to\infty$ in the following sense:
For every $t_1,\ldots,t_k>0$ and $f_1,g_1,\ldots,f_k,g_k:\mbb R_+\to\mbb R$ uniformly continuous and bounded,
\[\lim_{n\to\infty}\big(\langle f_i, \hat K^w_n(t_i)g_i\rangle\big)_{1\leq i\leq k}=\big(\langle f_i, \hat K^w(t_i)g_i\rangle\big)_{1\leq i\leq k}\]
in joint distribution and mixed moments.
\end{theorem}

\begin{remark}
Unlike Theorem \ref{Theorem: Main Dirichlet},
Theorem \ref{Theorem: Main Robin} contains no statement
on the convergence of traces. Similarly to the lack of trace convergence in \cite{GaudreauLamarreShkolnikov},
this is due to the fact that we were unable to construct a strong coupling of a certain Markov chain and
its occupation measures with the reflected Brownian bridge $X^{x,x}_t$
and its local time process. Throughout this paper, we make several remarks and
conjectures concerning this trace convergence, its consequences,
and the related strong invariance result (see Conjectures
\ref{Conjecture: Trace Convergence} and \ref{Conjecture: Trace Convergence 2}, and
Remark \ref{Remark: Trace Convergence 1}).
\end{remark}

\begin{conjecture}
\label{Conjecture: Trace Convergence}
In the setting of Theorem \ref{Theorem: Main Robin}, for every $t_1,\ldots,t_k>0$,
\[\lim_{n\to\infty}\big(\mr{Tr}[ \hat K^w_n(t_i)]\big)_{1\leq i\leq k}=\big(\mr{Tr}[ \hat K^w(t_i)]\big)_{1\leq i\leq k}\]
in joint distribution and mixed moments.
\end{conjecture}

\begin{remark}\label{Remark: Slightly Different Powers}
The conclusions of Theorems \ref{Theorem: Main Dirichlet}
and \ref{Theorem: Main Robin} remain valid if we define
\[\hat K_n(t)=\left(I_n-\frac{-\De_n+Q_n}{3m_n^2}\right)^{\th(n,t)},
\qquad
\hat K^w_n(t)=\left(I_n-\frac{-\De^w_n+Q_n}{3m_n^2}\right)^{\th(n,t)}\]
for $\th(n,t):=\lfloor m_n^2(3t/2)\rfloor\pm1$, instead of \eqref{Equation: Matrix Models}. Thus, up to making this minor change,
there is no loss of generality in assuming that $\lfloor m_n^2(3t/2)\rfloor$ is always even or odd if that
is more convenient (this distinction comes in handy in the proof of Proposition \ref{Proposition: Convergence of Eigenvalues}
below). We refer to Remark \ref{Remark Slightly Different Powers 2} for more details.
\end{remark}

\section{Applications to Random Matrices}\label{Section: Applications}

In this section, we provide applications of our main results to the study of random matrices and
$\be$-ensembles. We begin by stating our results in Sections \ref{Section: Convergence of Eigenvalues}--\ref{Section: Non-Symmetric Example},
and then provide their proofs in Sections \ref{Subsection: Proof of Eigenvalue Convergence 1}--\ref{Subsection: Non-Symmetric Proof}.

\subsection{Application 1. Convergence of Eigenvalues}
\label{Section: Convergence of Eigenvalues}

Throughout Section \ref{Section: Convergence of Eigenvalues}, we assume that 
$-\De_n+Q_n$ satisfies
the hypotheses of Theorem \ref{Theorem: Main Dirichlet},
and we denote by $-\infty<\la_1(\hat H)\leq\la_2(\hat H)\leq\cdots$ the
eigenvalues of the operator $\hat H=-\tfrac12\De+V+W'$ (as per
Remark \ref{Remark: Strong Continuity}), where $W$ is given by \eqref{Equation: Combination of BM},
and $V$ by \eqref{Equation: Combination of Potentials}.
The main result of Section \ref{Section: Convergence of Eigenvalues} is the following:

\begin{proposition}\label{Proposition: Convergence of Eigenvalues}
Suppose that $-\De_n+Q_n$
is diagonalizable with real eigenvalues $\la_{n;1}\leq\la_{n;2}\leq\cdots\leq\la_{n;n+1}$
for large enough $n$,
and that there exists $\de>0$ such that
\begin{align}\label{Equation: Eigenvalue Upper Bound}
\mbf P[\la_{n;n+1}\geq(6-\de)m_n^2\textnormal{ for infinitely many }n]=0.
\end{align}
Then for every $k\in\mbb N$,
\begin{align}\label{Equation: Convergence of Eigenvalues}
\lim_{n\to\infty}\tfrac12\big(\la_{n;1},\ldots,\la_{n;k}\big)=\big(\la_1(\hat H),\ldots,\la_k(\hat H)\big)\qquad
\text{in joint distribution.}
\end{align}
\end{proposition}

\begin{remark}
\label{Remark: Trace Convergence 1}
Proposition \ref{Proposition: Convergence of Eigenvalues} is only stated
for the Dirichlet boundary condition since it depends on the trace convergence
of Theorem \ref{Theorem: Main Dirichlet}-(2). If Conjecture \ref{Conjecture: Trace Convergence}
holds, then the same argument used to prove Proposition \ref{Proposition: Convergence of Eigenvalues}
would imply that the
eigenvalues of $\frac12(-\De^w_n+Q_n)$ converge to that of $\hat H^w$.
\end{remark}

\begin{question}
It would be interesting to see if some analog of Proposition \ref{Proposition: Convergence of Eigenvalues}
can be proved in the case where $-\De_n+Q_n$ is diagonalizable with complex eigenvalues.
We leave this as an open question.
\end{question}

We have the following convenient sufficient condition for \eqref{Equation: Eigenvalue Upper Bound},
which is easily seen to be satisfied for every example considered in Sections \ref{Section: Classical Beta} and \ref{Section: Non-Symmetric Example} below.

\begin{proposition}\label{Proposition: Convergence of Eigenvalues Condition}
Suppose that there exists $\bar\de>0$ and $N\in\mbb N$ such that
\begin{align}\label{Equation: Eigenvalue Upper Bound Condition}
\max_{0\leq a\leq n}\left(2+\frac{V^D_n(a)}{m_n^2}
+\left|\frac{V^U_n(a)}{m_n^2}-1\right|
+\left|\frac{V^L_n(a-1)}{m_n^2}-1\right|\right)\leq 6-\bar\de
\end{align}
for every $n\geq N$.
Then, \eqref{Equation: Eigenvalue Upper Bound} holds.
\end{proposition}

Finally, the following result provides a simple sufficient condition that allows to apply
Proposition \ref{Proposition: Convergence of Eigenvalues} to a very general
class of non-symmetric matrices.

\begin{proposition}\label{Proposition: Non-Symmetric Eigenvalues}
Suppose that there exists $N\in\mbb N$ large enough so that $Q_n$'s off-diagonal entries satisfy
\begin{align}\label{Equation: Product Condition}
(U_n(a)-m_n^2)(L_n(a)-m_n^2)>0,\qquad 0\leq a\leq n-1,~n\geq N.
\end{align}
Then, $-\De_n+Q_n$ is diagonalizable with real eigenvalues for $n\geq N$.
\end{proposition}

Propositions \ref{Proposition: Convergence of Eigenvalues}, \ref{Proposition: Convergence of Eigenvalues Condition}, and \ref{Proposition: Non-Symmetric Eigenvalues}
are proved in Sections \ref{Subsection: Proof of Eigenvalue Convergence 1}--\ref{Subsection: Proof of Eigenvalue Convergence 3}.
See Section \ref{Section: Non-Symmetric Example} for an example
of how these three results can be combined
to prove new eigenvalue limit laws for non-symmetric tridiagonal matrices.

\subsection{Application 2. Classical $\be$-Ensembles}
\label{Section: Classical Beta}

In Section \ref{Section: Classical Beta} we show that 
our main results apply to the edge-rescaled $\be$-Hermite ensemble,
the right-edge-rescaled $\be$-Laguerre ensemble,
as well as their rank-one spiked versions.
In all cases, the limits we obtain
are the stochastic Airy semigroups $\mr{SAS}_\be(t)$ and $\mr{SAS}^w_\be(t)$
respectively, thus extending the results of \cite{GaudreauLamarreShkolnikov,GorinShkolnikov}.

\subsubsection{Generalized $\be$-Hermite Ensembles}\label{Subsubsection: General Hermite Statement}

\begin{definition}
\label{Definition: Hermite}
Let $\xi^D_n\in\mbb R^{n+1}$ and $\xi^U_n=\xi^L_n\in\mbb R^n$ be random vectors that satisfy Assumptions
\ref{Assumption: Independence}--\ref{Assumption: Noise Convergence}
with $m_n=n^{1/3}$.
Let $\be>0$ be such that the Brownian motion $W$ in \eqref{Equation: Combination of BM}
has variance $1/\be$.
Let us denote $\chi_n(a):=\sqrt{n-a}-\xi_n^U(a)/n^{1/6}$ for all $0\leq a\leq n$. We define the
{\bf generalized $\bs\be$-Hermite ensemble} as
\[ H_n:=\left[\begin{array}{ccccc}
-\xi_n^D(0)/n^{1/6}&\chi_n(0)\\
\chi_n(0)&\ddots&\ddots\\
&\ddots&\ddots&\chi_n(n-1)\\
&&\chi_n(n-1)&-\xi_n^D(n)/n^{1/6}
\end{array}\right].
\]
\end{definition}
\begin{definition}
\label{Definition: Spiked Hermite}
Let $(\mu_n)_{n\in\mbb N}$ be a sequence of real numbers such that
\begin{align}\label{Equation: Hermite Spike}
\lim_{n\to\infty}n^{-1/6}(\sqrt{n}-\mu_n)=w\in\mbb R.
\end{align}
Let $\xi^E_n$ and $H_n$ be as in Definition \ref{Definition: Hermite},
assuming further that $(\xi^D_n(0)/n^{1/6})_{n\in\mbb N}$ is uniformly sub-Gaussian.
The {\bf generalized spiked $\bs\be$-Hermite ensemble} is defined as
$H^w_n:=H_n+\mr{diag}_n(\mu_n,0,\ldots,0).$
\end{definition}

$H_n$ and $H_n^w$ are slight generalizations of the random matrix models studied
in \cite{GaudreauLamarreShkolnikov,GorinShkolnikov}.
As shown in \cite[Lemma 2.1]{GorinShkolnikov}, the
$\be$-Hermite ensemble studied in \cite{DumitriuEdelman,EdelmanSutton,RamirezRiderVirag}
is a special case of $H_n$. Similarly,
as noted in \cite[Remarks 1.3 and 1.8]{GaudreauLamarreShkolnikov},
$H^w_n$ generalizes
the spiked $\be$-Hermite ensemble with a critical (i.e., of size $\sqrt{n}$) rank-one additive perturbation
introduced in \cite[(1.5)]{BloemendalVirag} (see also \cite{Peche}).
As per classical theory, the edge fluctuations of $H_n$ and $H_n^w$ are captured by
the rescalings
\begin{align}\label{Equation: Hermite Rescaled}
R_n:=n^{1/6}(2\sqrt{n}\,I_n- H_n)
\qquad\text{and}\qquad
R_n^w:=n^{1/6}(2\sqrt{n}\,I_n- H^w_n).
\end{align}
We have the following result regarding \eqref{Equation: Hermite Rescaled},
which we prove in Section \ref{Section: Hermite Proof}.

\begin{corollary}
\label{Corollary: Hermite}
We can define $Q_n$ so that
$R_n=-\De_n+Q_n$
and $R_n^w=-\De^w_n+Q_n$ satisfy the hypotheses of
Theorems \ref{Theorem: Main Dirichlet} and \ref{Theorem: Main Robin}
respectively,
where $m_n=n^{1/3}$,
$w_n=\mu_n/\sqrt{n}$,
$W$ in \eqref{Equation: Combination of BM} has variance $1/\be$,
and $V(x)$ in \eqref{Equation: Combination of Potentials} equals $x/2$.
\end{corollary}

\subsubsection{Generalized $\be$-Laguerre Ensembles}\label{Subsubsection: General Laguerre Statement}

\begin{definition}
\label{Definition: Laguerre}
Suppose that $\tilde\xi^D_n$ and $\tilde\xi^U_n=\tilde\xi^L_n$ satisfy Assumptions
\ref{Assumption: Independence} and \ref{Assumption: Moments}
with $m_n=n^{1/3}$, and that $\tilde\xi^E_n$ satisfy
\eqref{Equation: Skorokhod Noise Terms} and \eqref{Equation: Noise Stochastic Integral Convergence}
with $m_n=n^{1/3}$. Denoting the limits in distribution
\[\tilde W^E(x):=\lim_{n\to\infty}\frac1{n^{1/3}}\sum_{a=0}^{\lfloor n^{1/3} x\rfloor}\tilde\xi^E_n(a),\qquad E\in\{D,U,L\},\]
we further assume that $\tilde W^D+\tilde W^U=\tilde W^D+\tilde W^L$ is a Brownian motion
with variance $1/\be$ for some $\be>0$.
Let $p=p(n)>n$ be an increasing sequence such that $n/p\to\nu\in[0,1]$ as $n\to\infty$.
Denote
$\chi_n(a):=\sqrt{n-a}-\tilde\xi^U_n(a)/n^{1/6}$
and $\chi_{n;p}(a):=\sqrt{p-a}-\tilde\xi^D_n(a)/n^{1/6}$.
We define the {\bf generalized $\bs\be$-Laguerre ensemble} as
$L_n:=(L^*_n)^\top L^*_n$, where
\[L^*_n:=\left[\begin{array}{ccccc}
\chi_{n;p}(0)\\
\chi_n(0)&\chi_{n;p}(1)\\
&\ddots&\ddots\\
&&\chi_n(n-1)&\chi_{n;p}(n)\\
\end{array}\right].\]
\end{definition}
\begin{definition}
Let $\tilde\xi^E_n,p(n)$, and $L^*_n$
be as in Definition \ref{Definition: Laguerre},
with the additional assumption that
$(\tilde\xi^D_n(0)/n^{1/6})_{n\in\mbb N}$ and $(\tilde\xi^L(0)/n^{1/6})_{n\in\mbb N}$ are uniformly sub-Gaussian.
Let $(\ell_n)_{n\in\mbb N}$ be a sequence of real numbers such that
\begin{align}\label{Equation: Laguerre Spike}
\lim_{n\to\infty}
{\textstyle \left(\frac{\sqrt{np}}{\sqrt{n}+\sqrt{p}}\right)^{2/3}}\left(1-\sqrt{p/n}\,(\ell_n-1)\right)=w\in\mbb R.
\end{align}
The {\bf generalized spiked $\bs\be$-Laguerre ensemble} is defined as
 $L^w_n:=(L^*_n)^\top\,\mr{diag}_n(\ell_n,1,\ldots,1)\, L^*_n$.
\end{definition}

$L_n$ is a generalization of the $\be$-Laguerre ensemble studied in
\cite{DumitriuEdelman,EdelmanSutton,RamirezRiderVirag};
$L^w_n$ is a generalization of the critical (i.e., of size $1+\sqrt{\nu}$) rank-one spiked model of the
$\be$-Laguerre ensemble (c.f., \cite{BaikBenArousPeche} and \cite[(1.2)]{BloemendalVirag}).
The right-edge (i.e., largest eigenvalues) fluctuations of these matrices are captured by the rescalings
\begin{align}\label{Equation: Laguerre Rescaling}
\Si_n:=\frac{m_n^2}{\sqrt{np}}\big((\sqrt{n}+\sqrt{p})^2  I_n- L_n\big),
\quad\text{where}\quad
m_n:= \left(\frac{\sqrt{np}}{\sqrt{n}+\sqrt{p}}\right)^{2/3},
\end{align}
and $\Si^w_n:=\big(m_n^2/\sqrt{np}\big)\big((\sqrt{n}+\sqrt{p})^2  I_n- L^w_n\big)$
with the same $m_n$.
The following is proved in Section \ref{Section: Proof of Laguerre}:

\begin{corollary}\label{Corollary: Laguerre}
We can define $Q_n$ so that
$\Si_n=-\De_n+Q_n$
and $\Si_n^w=-\De^w_n+Q_n$ satisfy the hypotheses of
Theorems \ref{Theorem: Main Dirichlet} and \ref{Theorem: Main Robin}
respectively,
where $m_n$ is as in \eqref{Equation: Laguerre Rescaling},
$w_n=\sqrt{p/n}(\ell_n-1)$,
$W$ in \eqref{Equation: Combination of BM} has variance $1/\be$,
and $V(x)$ in \eqref{Equation: Combination of Potentials} equals $x/2$.
\end{corollary}

\subsection{Application 3. Non-Symmetric Ensemble}
\label{Section: Non-Symmetric Example}

We now provide an example of a non-symmetric matrix model
for which we can prove a new limit law. The following model
is inspired by the $\be$-Hermite ensemble:

\begin{definition}
\label{Definition: non-Symmetric Hermite-type}
Suppose that $\xi^D_n$ and $\xi^U_n\neq\xi^L_n$ satisfy Assumptions
\ref{Assumption: Independence}--\ref{Assumption: Noise Convergence} with $m_n=n^{1/3}$.
Let us denote
$\chi^U_n(a):=\sqrt{n-a}-\xi_n^U(a)/n^{1/6}$ and $\chi^L_n(a):=\sqrt{n-a}-\xi_n^L(a)/n^{1/6}$,
and assume  that $\chi^U_n(a),\chi^L_n(a)>0$ (or, equivalently,
$\xi^U_n(a),\xi^L_n(a)<n^{1/6}\sqrt{n-a}$) for every $0\leq a\leq n-1$.
Define the random matrix
\begin{align}\label{Equation: non-Symmetric Hermite-type}
 \tilde H_n:=\left[\begin{array}{ccccc}
-\xi_n^D(0)/n^{1/6}&\chi^U_n(0)\\
\chi^L_n(0)&\ddots&\ddots\\
&\ddots&\ddots&\chi^U_n(n-1)\\
&&\chi^L_n(n-1)&-\xi_n^D(n)/n^{1/6}
\end{array}\right].
\end{align}
\end{definition}

In order to capture the edge fluctuations of $\tilde H_n$,
we consider the rescaled version
\[\tilde R_n:=n^{1/6}(2\sqrt{n}\,I_n-\tilde H_n).\]
The following result is proved in Section \ref{Subsection: Non-Symmetric Proof}:

\begin{corollary}
\label{Corollary: Non Symmetric Example}
For every $k\in\mbb N$, the $k$ smallest eigenvalues of
$\tilde R_n$ converge in joint distribution to the $k$ smallest eigenvalues of
$\mr{SAO}_\be$ with Dirichlet boundary condition.
\end{corollary}

\subsection{Proof of Proposition \ref{Proposition: Convergence of Eigenvalues}}
\label{Subsection: Proof of Eigenvalue Convergence 1}

As argued in \cite[Section 6]{GorinShkolnikov} and \cite[Section 5]{Sodin}, it suffices
to prove the convergence of Laplace transforms
\[\lim_{n\to\infty}\left(\sum_{j=1}^{n+1}\mr e^{-t_i\la_{n;j}/2}\right)_{0\leq i\leq k}
=\left(\sum_{j=1}^\infty\mr e^{-t_i\la_j(\hat H)}\right)_{0\leq i\leq k},\qquad t_1,\ldots,t_k>0\]
in joint distribution.
On the one hand, if $-\De_n+Q_n$ is diagonalizable, then
\[\mr{Tr}[\hat K_n(t)]=\sum_{j=1}^{n+1}\left(1-\frac{\la_{n;j}}{3m_n^2}\right)^{\lfloor m_n^2(3t/2)\rfloor}\]
for every $t>0$. On the other hand, by \cite[Theorem 2.23]{GaudreauLamarre},
for every $t>0$,
\[\mr{Tr}[\hat K(t)]=\sum_{j=1}^\infty\mr e^{-t\la_j(\hat H)}<\infty\qquad\text{almost surely}.\]
Consequently, by Theorem \ref{Theorem: Main Dirichlet}-(2), we need only prove that
\begin{align}\label{Equation: Matrix Model to Laplace}
\lim_{n\to\infty}\left(\sum_{j=1}^{n+1}\mr e^{-t_i\la_{n;j}/2}-\left(1-\frac{\la_{n;j}}{3m_n^2}\right)^{\lfloor m_n^2(3t_i/2)\rfloor}\right)_{0\leq i\leq k}=(0,\ldots,0)
\end{align}
in joint distribution.

By the Skorokhod representation theorem, if $\hat K_n(t)\to\hat K(t)$ in the sense of Theorem \ref{Theorem: Main Dirichlet}-(2),
then there exists a coupling of the sequence $(\la_{n;j})_{1\leq j\leq n+1,n\in\mbb N}$ and $\big(\la_j(\hat H)\big)_{j\in\mbb N}$
such that
\begin{align}\label{Label: Skorohkod on Trace}
\lim_{n\to\infty}\sum_{j=1}^{n+1}\left(1-\frac{\la_{n;j}}{3m_n^2}\right)^{\lfloor m_n^2(3t_i/2)\rfloor}=\sum_{j=1}^\infty\mr e^{-t\la_j(\hat H)}<\infty
\end{align}
almost surely for $1\leq i\leq k$.
By Remark \ref{Remark: Slightly Different Powers}, there is no loss of generality
in assuming that $\lfloor m_n^2(3t_i/2)\rfloor$ is even for all $n$; hence
\[\sum_{j=1}^{n+1}\left(1-\frac{\la_{n;j}}{3m_n^2}\right)^{\lfloor m_n^2(3t_i/2)\rfloor}=\sum_{j=1}^{n+1}\left|1-\frac{\la_{n;j}}{3m_n^2}\right|^{\lfloor m_n^2(3t_i/2)\rfloor}.\]
Let us fix $0<\de<1$ and $0<\eps<\mf d$, where $\mf d$ is as in \eqref{Equation: mn}.
We consider four different regimes of eigenvalues of $-\De_n+Q_n$:
\begin{enumerate}
\item $J_{n;1}:=\{j:\la_{n;j}<-n^\eps\}$;
\item $J_{n;2}:=\{j:-n^\eps\leq\la_{n;j}<n^\eps\}$;
\item $J_{n;3}:=\{j:n^\eps\leq\la_{n;j}<(6-\de)\lfloor m_n^2(3t_i/2)\rfloor\}$; and
\item $J_{n;4}:=\{j:(6-\de)\lfloor m_n^2(3t_i/2)\rfloor\leq\la_{n;j}\}$.
\end{enumerate}

Firstly, note that
\[\sum_{j\in J_{n;1}}\left|1-\frac{\la_{n;j}}{3m_n^2}\right|^{\lfloor m_n^2(3t_i/2)\rfloor}
\geq|J_{n;1}|\left(1+\frac{n^\eps}{3m_n^2}\right)^{\lfloor m_n^2(3t_i/2)\rfloor},\]
where $|J_{n;1}|$ denotes the cardinality of $J_{n;1}$. If $|J_{n;1}|>0$ for infinitely many
$n$, then this quantity diverges, contradicting the convergence of \eqref{Label: Skorohkod on Trace}.
Hence
$J_{n;1}$ does not contribute to \eqref{Equation: Matrix Model to Laplace}.

Secondly, recall the elementary inequalities
\[0<e^z-\left(1+\frac{z}{m}\right)^m<\left(1+\frac{z}{m}\right)^m\left(\left(1+\frac{z}{m}\right)^z-1\right),\quad \forall z,m>0\]
and
\[0<e^{-z}-\left(1-\frac{z}{m}\right)^m<\left(1-\frac{z}{m}\right)^m\left(\left(1-\frac{z}{m}\right)^{-z}-1\right),\quad \forall m>z>0,\]
which imply that
\[\left|\sum_{j\in J_{n;2}}\mr e^{-t_i\la_{n;j}/2}-\left(1-\frac{\la_{n;j}}{3m_n^2}\right)^{\lfloor m_n^2(3t_i/2)\rfloor}\right|
\leq\left(\left(1+\frac{n^\eps}{3m_n^2}\right)^{n^\eps}-1\right)\sum_{j\in J_{n;2}}\left|1-\frac{\la_{n;j}}{3m_n^2}\right|^{\lfloor m_n^2(3t_i/2)\rfloor}.
\]
Since $n^{2\eps}=o(m_n^2)$, we have $\left(1+\frac{n^\eps}{3m_n^2}\right)^{n^\eps}=1+o(1)$, and thus
\eqref{Label: Skorohkod on Trace} implies that the contribution of $J_{n;2}$ to
\eqref{Equation: Matrix Model to Laplace} vanishes.

Thirdly, one the one hand, we have that
\[\sum_{j\in J_{n;3}}\mr e^{-t_i\la_{n;j}/2}\leq|J_{n;3}|\,\mr e^{-t_i n^\eps/2},\]
and on the other hand, since $|1-z|\leq\max\{\mr e^{-z},\mr e^{z-2}\}$ ($z\in\mbb R$), we see that
\begin{multline*}
\sum_{j\in J_{n;3}}\left|1-\frac{\la_{n;j}}{3m_n^2}\right|^{\lfloor m_n^2(3t_i/2)\rfloor}
\leq|J_{n;3}|\max_{j\in J_{n;3}}\max\bigg\{\exp\left(-\frac{\lfloor m_n^2(3t_i/2)\rfloor\la_{m;j}}{3m_n^2}\right),\\
\exp\left(\frac{\lfloor m_n^2(3t_i/2)\rfloor\la_{m;j}}{3m_n^2}-2\lfloor m_n^2(3t_i/2)\rfloor\right)\bigg\}.
\end{multline*}
Note that $|J_{n;3}|\leq n+1$ and that there exists a constant $C>0$ independent of $n$
such that for every $j\in J_{n;3}$,
\begin{align*}
\exp\left(-\frac{\lfloor m_n^2(3t_i/2)\rfloor\la_{m;j}}{3m_n^2}\right)&\leq\mr e^{-Cn^{\eps}},\\
\exp\left(\frac{\lfloor m_n^2(3t_i/2)\rfloor\la_{m;j}}{3m_n^2}-2\lfloor m_n^2(3t_i/2)\rfloor\right)&\leq\exp\left(-\frac{\de}{3}\lfloor m_n^2(3t_i/2)\rfloor\right).
\end{align*}
Consequently, the contribution of $J_{n;3}$ to \eqref{Equation: Matrix Model to Laplace} vanishes.

Finally, we know from \eqref{Equation: Eigenvalue Upper Bound} that there is eventually
no eigenvalue in $J_{n;4}$, and thus it has no contribution to \eqref{Equation: Matrix Model to Laplace},
completing the proof Proposition \ref{Proposition: Convergence of Eigenvalues}.

\subsection{Proof of Proposition \ref{Proposition: Convergence of Eigenvalues Condition}}

According to the Gershgorin disc theorem (e.g., \cite[Corollary 9.11]{Zhan}),
\[\frac{\la_{n;n+1}}{m_n^2}\leq
\max_{0\leq a\leq n}\left(2+\frac{D_n(a)}{m_n^2}
+\left|-1+\frac{U_n(a)}{m_n^2}\right|
+\left|-1+\frac{L_n(a-1)}{m_n^2}\right|\right).\]
By combining this with \eqref{Equation: Eigenvalue Upper Bound Condition} and the triangle inequality,
we get
\[\frac{\la_{n;n+1}}{m_n^2}\leq6-\bar\de+\sum_{E=D,U,L}\max_{0\leq a\leq n}\frac{|\xi^E_n(a)|}{m_n^2}\]
for large enough $n$. By a union bound, \eqref{Equation: Noise Moment Bound}, and Markov's inequality,
we see that
\[\mbf P\left[\max_{0\leq a\leq n}\frac{|\xi^E_n(a)|}{m_n^2}\geq\tilde\de\right]=O\left(\frac{n}{m_n^{3q/2}}\right).\]
for any $\tilde\de\in(0,\bar\de)$ and $q\in\mbb N$.
By \eqref{Equation: mn}, we can take $q$ large enough
so that $\sum_nn/m_n^{3q/2}<\infty$; the result then follows by the Borel-Cantelli lemma.

\subsection{Proof of Proposition \ref{Proposition: Non-Symmetric Eigenvalues}}
\label{Subsection: Proof of Eigenvalue Convergence 3}

This is a direct consequence of the following classical result
in matrix theory:

\begin{lemma}[{\cite[3.1.P22; see also Page 585]{HornJohnson}}]
Let $M$ be a $(n+1)\times(n+1)$ real-valued tridiagonal matrix.
If $M(a,a+1)M(a+1,a)>0$ for every $0\leq a\leq n-1$, then $M$
is similar to a Hermitian matrix.
\end{lemma}

\subsection{Proof of Corollary \ref{Corollary: Hermite}}
\label{Section: Hermite Proof}

Thanks to \eqref{Equation: Hermite Rescaled},
straightforward computations reveal that we can write $R_n=-\De_n+Q_n$
and $R_n^w=-\De^w_n+Q_n$ with $m_n=n^{1/3}$, where the noise terms $\xi^E_n$ are as in
Definition \ref{Definition: Hermite}, and the potential terms are
\begin{align}
\label{Equation: Hermite Potential}
V^D_n(a)=0
\quad\text{and}\qquad
V^U_n(a)=V^L_n(a)=n^{1/6}\big(\sqrt{n}-\sqrt{n-a}\big)
\end{align}
for $0\leq a\leq n$. By Definitions \ref{Definition: Hermite}
and \ref{Definition: Spiked Hermite},
$\xi^E_n$ satisfy Assumptions
\ref{Assumption: Independence}--\ref{Assumption: Noise Convergence},
and Assumptions \ref{Assumption: wn} and \ref{Assumption: Robin}
hold for $H_n^w$ with $w_n=\mu_n/\sqrt{n}$. Thus, it only remains to prove
that \eqref{Equation: Hermite Potential} satisfies Assumptions
\ref{Assumption: Potential Convergence}--\ref{Assumption: Growth Lower Bounds}
with $V(x)$ in \eqref{Equation: Combination of Potentials} equal to $x/2$.

Note that $n^{1/6}\big(\sqrt{n}-\sqrt{n-a}\big)=n^{2/3}\big(1-\sqrt{1-a/n}\big)$;
hence Assumption \ref{Assumption: Growth Upper Bounds} is met.
Elementary calculus shows that
for any $0<\ka<1/2$ and $c>0$, the function
\[x\mapsto c^2 \big(1-\sqrt{1-x/c^3}\big)-\ka x/c\]
is nonnegative on $x\in[0,c^3]$. Taking $c=m_n$, this implies that
Assumption \ref{Assumption: Growth Lower Bounds} is met with $E=U,L$
in both \eqref{Equation: Log Potential Lower Bound} and
\eqref{Equation: Polynomial Potential Lower Bound}.
Finally, for $E=U,L$ and $x\geq0$,
\[V^E(x):=\lim_{n\to\infty}V^E_n(\lfloor n^{1/3} x\rfloor)
=\lim_{n\to\infty}n^{2/3}\big(1-\sqrt{1-\lfloor n^{1/3} x\rfloor/n}\big)
=x/2\quad\text{pointwise}.\]
Since $V^E_n(\lfloor n^{1/3} x\rfloor)$ is nondecreasing in $x$ for every $n$, the convergence is uniform
on compacts. Then, we are led to $V(x)=\tfrac12\big(V^U(x)+V^L(x)\big)=x/2$, as desired.

\subsection{Proof of Corollary \ref{Corollary: Laguerre}}
\label{Section: Proof of Laguerre}

\begin{remark}
Unless otherwise stated, $m_n$ in this proof refers to the
quantity $\left(\frac{\sqrt{np}}{\sqrt{n}+\sqrt{p}}\right)^{2/3}$ defined in \eqref{Equation: Laguerre Rescaling}.
If we invoke statements regarding quantities that satisfy
Assumptions \ref{Assumption: Independence}--\ref{Assumption: Noise Convergence}
with other values of $m_n$, we will explicitly state so.
\end{remark}

By definition of $p$ and $\nu$,
$n^{-1/3}m_n=\big(1+\sqrt{\nu}\big)^{-2/3}\big(1+o(1)\big)$,
and thus \eqref{Equation: mn} holds with $\mf d=1/3$.
With this in hand, straightforward computations using \eqref{Equation: Laguerre Rescaling}
reveal that we can write $\Si_n=-\De_n+Q_n$ with the potential terms
\[V^D_n(a)=2\frac{m_n^2}{\sqrt{np}}a
\qquad
V^U_n(a)=V^L_n(a)
=m_n^2\big(1-\sqrt{\left(1-a/n\right)\left(1-(a-1)/p\right)}\big)\]
and the noise terms
\begin{align*}
\xi^D_n(a)&=\frac{m_n^2}{\sqrt{np}}\Big(2\big(\sqrt{p-a}\,\frac{\tilde\xi^D_n(a)}{n^{1/6}}+\sqrt{n-a}\, \frac{\tilde\xi^U_n(a)}{n^{1/6}}\big)
-\frac{\tilde\xi^D_n(a)^2}{n^{2/3}}-\frac{\tilde\xi^U_n(a)^2}{n^{2/3}}\Big)\\
\xi^U_n(a)=\xi^L_n(a)&=\frac{m_n^2}{\sqrt{np}}\Big(\big(\sqrt{n-a}\,\frac{\tilde\xi^D_n(a+1)}{n^{1/6}}+\sqrt{p-a-1}\, \frac{\tilde\xi^U_n(a)}{n^{1/6}}\big)
-\frac{\tilde\xi^D_n(a+1)\tilde\xi^U_n(a)}{n^{2/3}}\Big)
\end{align*}
We can similarly write $\Si_n^w=-\De_n^w+Q_n$ with $w_n=\sqrt{p/n}\,(\ell_n-1)$,
the only difference in $Q_n$ being in the $(0,0)$ entry, which has
$V^D(0)=0$ and
\begin{align}
\label{Equation: Zero Entry Spiked Laguerre}
\xi^D_n(0)=\frac{m_n^2}{\sqrt{np}}\Big(2\big(\sqrt{p}\ell_n\,\frac{\tilde\xi^D_n(0)}{n^{1/6}}+\sqrt{n}\, \frac{\tilde\xi^U(0)}{n^{1/6}}\big)
-\ell_n\,\frac{\tilde\xi^D_n(0)^2}{n^{2/3}}-\frac{\tilde\xi^U_n(0)^2}{n^{2/3}}\Big).
\end{align}
We now check that the hypotheses of Theorems \ref{Theorem: Main Dirichlet} and \ref{Theorem: Main Robin}
are met.

Regarding the potential terms,
\eqref{Equation: Potential Absolute Bounds} and \eqref{Equation: Potential Growth Upper Bound}
are immediate from the definition of $V^E_n$ above.
Given that $\big(1-\sqrt{\left(1-a/n\right)\left(1-(a-1)/p\right)}\big)\geq\big(1-\sqrt{\left(1-a/n\right)}\big),$
the same argument used in the proof of Corollary \ref{Corollary: Hermite}
implies that \eqref{Equation: Log Potential Lower Bound} and \eqref{Equation: Polynomial Potential Lower Bound}
both hold with $E=U,L$.
Next, by writing $n=\nu p\big(1+o(1)\big)$, we observe that we have the following
pointwise limits in $x\geq0$:
\begin{align*}
V^D(x):=\lim_{n\to\infty}V^D_n(\lfloor m_nx\rfloor)&=\frac{2 \sqrt{\nu} x}{(1+\sqrt{\nu})^2}\\
V^E(x):=\lim_{n\to\infty}V^E_n(\lfloor m_nx\rfloor)&=\frac{(1+\nu)x}{2 (1+\sqrt{\nu})^2},&E=U,L.
\end{align*}
Once again the monotonicity in $x$ of the functions involved implies uniform convergence
on compacts, and we have $V(x):=\tfrac12\big(V^D(x)+V^U(x)+V^L(x)\big)=x/2$.

We now prove that the noise terms $\xi^E_n$ satisfy
Assumptions \ref{Assumption: Independence}--\ref{Assumption: Noise Convergence}.
Since
\begin{align}\label{Equation: p vs n Estimates}
m_n=O(n^{1/3})\qquad\text{and}\qquad m_n^2/\sqrt{n},~m_n^2/\sqrt{p}=O(m_n^{1/2})=O(n^{1/6}),
\end{align}
the fact that $\tilde\xi^E_n$ satisfies Assumptions \ref{Assumption: Independence}
and \ref{Assumption: Moments} with $m_n=n^{1/3}$ implies that $\xi_n^E$
satisfy Assumptions \ref{Assumption: Independence}
and \ref{Assumption: Moments} as well.
Recall that, by definition, $\tilde\xi_n^E$ satisfy Assumption \ref{Assumption: Noise Convergence}
with $m_n=n^{1/3}$ (and we denote the corresponding limiting Brownian motions as
$\tilde W^D$, $\tilde W^U=\tilde W^L$). Since $m_n/n^{1/3}\to(1+\sqrt{\nu})^{-2/3}$
converges to a constant, it then follows from
a straightforward Brownian scaling that $\frac1{m_n}\sum_{a=0}^{\lfloor m_n x\rfloor}\tilde\xi^E_n(a)\to(1+\sqrt{\nu})^{1/3}\tilde W^E(x)$
in distribution.
Combining this with the fact that for every $a=o(n)$, one has
\[\lim_{n\to\infty}\frac{m_n^2\sqrt{p-a}}{n^{1/6}\sqrt{np}}=\frac{1}{(1+\sqrt{\nu})^{4/3}}
\qquad\text{and}\qquad
\lim_{n\to\infty}\frac{m_n^2\sqrt{n-a}}{n^{1/6}\sqrt{np}}=\frac{\sqrt\nu}{(1+\sqrt{\nu})^{4/3}}\]
we then obtain that $ \xi_n^E$ satisfy Assumption \ref{Assumption: Noise Convergence} with
\[ W^D(x):=\lim_{n\to\infty}\frac1{m_n}\sum_{a=0}^{\lfloor m_n x\rfloor}\xi^D_n(a)
=\left(\frac{2}{1+\sqrt{\nu}}\right)\tilde W^D(x)+\left(\frac{2\sqrt\nu}{1+\sqrt{\nu}}\right)\tilde W^U(x),\]
and for $E=L,U$,
\[ W^E(x):=\lim_{n\to\infty}\frac1{m_n}\sum_{a=0}^{\lfloor m_n x\rfloor}\xi^E_n(a)
=\left(\frac{1}{1+\sqrt{\nu}}\right)\tilde W^U(x)+\left(\frac{\sqrt\nu}{1+\sqrt{\nu}}\right)\tilde W^D(x).\]
From this we immediately obtain that
$W:=\tfrac12( W^D+ W^U+ W^L)=\tilde W^D+\tilde W^L$
is a Brownian motion with variance $1/\be$, as desired.

We conclude the proof by checking the assumptions related to
the rank-one spike in $L^w_n$. That Assumption \ref{Assumption: wn}
is satisfied with $w_n=\sqrt{p/n}(\ell_n-1)$ is an immediate
consequence of \eqref{Equation: Laguerre Spike}. As for \eqref{Equation: Zero Entry Spiked Laguerre}
satisfying Assumption \ref{Assumption: Robin}, this is immediate from
the fact that $\tilde\xi^E_n(0)/n^{1/3}$ are uniformly sub-Gaussian,
the estimates \eqref{Equation: p vs n Estimates},
and the fact that $\ell_n=1+\sqrt{\nu}+O(m_n^{-1})$ (by \eqref{Equation: Laguerre Spike}).

\subsection{Proof of Corollary \ref{Corollary: Non Symmetric Example}}
\label{Subsection: Non-Symmetric Proof}

It is easy to see that $\tilde R_n$ is of the form
$-\De_n+Q_n$ (with $m_n=n^{1/3}=n^{1/6}\sqrt n$), where,
for $E=U,L$, one has
\[U_n(a)=n^{1/6}\big(\sqrt n-\sqrt{n-a}+\xi^U_n(a)/n^{1/6}\big),\]
and $D_n(a)=\xi^D_n(a)$.
Given that $-\sqrt{n-a}+\frac{\xi^U_n(a)}{n^{1/6}},-\sqrt{n-a}+\frac{\xi^L_n(a)}{n^{1/6}}<0$
(by Definition \ref{Definition: non-Symmetric Hermite-type}), $\tilde R_n$ satisfies
\eqref{Equation: Product Condition}. We can prove that $\tilde R_n$ satisfies
Assumptions \ref{Assumption: mn} and
and \ref{Assumption: Potential Convergence}--\ref{Assumption: Noise Convergence}
in the same way as Corollary \ref{Corollary: Hermite};
hence the result follows from Propositions \ref{Proposition: Convergence of Eigenvalues}, \ref{Proposition: Convergence of Eigenvalues Condition}, and \ref{Proposition: Non-Symmetric Eigenvalues}
(\eqref{Equation: Eigenvalue Upper Bound Condition} is easily seen to hold here).

\section{From Matrices to Feynman-Kac Functionals}
\label{Section: Method}

In this section, we derive probabilistic representations
for $\langle f, \hat K_n(t)g\rangle$, $\mr{Tr}[K_n(t)]$, and $\langle f, \hat K^w_n(t)g\rangle$
that serve as finite-dimensional analogs of \eqref{Equation: Dirichlet Semigroup}
and \eqref{Equation: Robin Semigroup}.

\subsection{Dirichlet Boundary Condition: Lazy Random Walk}

\begin{definition}[Lazy Random Walk]\label{Definition: Lazy Walk}
Let $S=\big(S(u)\big)_{u\in\mbb N_0}$ ($\mbb N_0:=\{0,1,2,\ldots\}$)
be a lazy random walk, i.e., the increments $S(u)-S(u-1)$
are i.i.d. uniform random variables on $\{-1,0,1\}$.
For every $a,b,u\in\mbb N_0$, we denote $S^a:=\big(S|S(0)=a\big)$
and $S^{a,b}_u:=\big(S|S(0)=a\text{ and }S(u)=b\big)$.
\end{definition}

\subsubsection{Inner Product}

Let $ M$ be a $(n+1)\times(n+1)$ random tridiagonal matrix, let $v\in\mbb R^{n+1}$ be a vector,
and let $\th\in\mbb N$ be a fixed integer. By definition of matrix product, for every $0\leq a\leq n$,
\begin{align}\label{Equation: Matrix to Expectation 1}
\big((\tfrac13 M)^\th v\big)(a)=\frac{1}{3^\th}\sum_{a_1,\ldots,a_{\th-1}} M(a,a_1) M(a_1,a_2)\cdots M(a_{\th-1},a_\th)v(a_\th),
\end{align}
where the sum is taken over all $a_1,\ldots,a_{\th}\in\mbb N_0$ such that $(a,a_1,\ldots,a_{\th})$
forms a path on the lattice $\{0,1,2,\ldots,n\}$ with self-edges (i.e., $|a_i-a_{i-1}|\in\{0,1\}$).
The probability that $S^a$ is equal to any such
path is $3^{-\th}$, and thus we see that
\begin{align}\label{Equation: Matrix to Expectation 2}
\big((\tfrac13 M)^\th v\big)(a)=\mbf E^a\left[\mbf 1_{\{\tau^{(n)}(S)>\th\}}\left(\prod_{u=0}^{\th-1} M\big(S(u),S(u+1)\big)\right) v\big(S(\th)\big)\right],
\end{align}
where the random walk $S$ is independent of the randomness in $M$,
$\mbf E^a$ denotes the expected value with respect to the law of $S^a$ conditional on $M$,
and
\[\tau^{(n)}(S):=\min\{u\geq0:S(u)=-1\text{ or }n+1\}.\]
We can think of the contribution
of $M$ to \eqref{Equation: Matrix to Expectation 2} as a type of random walk in random scenery process on the edges
of $\{0,1,2,\ldots,n\}$, that is, each passage of $S$ on an edge contributes
to the multiplication of the corresponding entry in $M$.
In particular, if we define the {\bf edge-occupation measures}
\begin{align}\label{Equation: Edge-Occupation Measure}
\La^{(a,b)}_\th(S):=\sum_{u=0}^{\th-1}\mbf 1_{\{S(u)=a\text{ and }S(u+1)=b\}},\qquad 0\leq a,b\leq n,
\end{align}
then we have that
\begin{align}\label{Equation: Matrix to Expectation 3}
\prod_{u=0}^{\th-1} M\big(S(u),S(u+1)\big)
=\prod_{a,b\in\mbb Z} M(a,b)^{\La_\th^{(a,b)}(S)}.
\end{align}

We now apply the above discussion to the study of $ \hat K_n(t)$.
We observe that
\begingroup
\allowdisplaybreaks
\begin{align}\label{Equation: Dirichlet Diagonal Entry}
\left( I_n-\frac{-\De_n+ Q_n}{3m_n^2}\right)(a,a)&=\frac13\left(1-\frac{D_n(a)}{m_n^2}\right)&0\leq a\leq n,\\
\label{Equation: Dirichlet Upper Diagonal Entry}
\left( I_n-\frac{-\De_n+ Q_n}{3m_n^2}\right)(a,a+1)&=\frac13\left(1-\frac{U_n(a)}{m_n^2}\right)&0\leq a\leq n-1,\\
\label{Equation: Dirichlet Lower Diagonal Entry}
\left( I_n-\frac{-\De_n+ Q_n}{3m_n^2}\right)(a+1,a)&=\frac13\left(1-\frac{L_n(a)}{m_n^2}\right)&0\leq a\leq n-1.
\end{align}
\endgroup
Let $t>0$ and $n\in\mbb N$
be fixed, and let us denote $\th=\th(n,t):=\lfloor m_n^2(3t/2)\rfloor$.
By combining \eqref{Equation: Dirichlet Diagonal Entry}--\eqref{Equation: Dirichlet Lower Diagonal Entry},
the combinatorial analysis in \eqref{Equation: Matrix to Expectation 1}--\eqref{Equation: Matrix to Expectation 3},
and the embedding $\pi_n$ in \eqref{Notation: Embedding},
we see that
\begin{align}\label{Equation: Matrix-Form Formula}
\langle f, \hat K_n(t)g\rangle
=\int_0^{(n+1)/m_n}f(x)\,\mbf E^{\lfloor m_nx\rfloor}\left[F_{n,t}(S)\, m_n\int_{S(\th)/m_n}^{(S(\th)+1)/m_n}g(y)\d y\right]\d x,
\end{align}
where $S$ is independent of $Q_n$,
we define the random functional
\begin{align}\label{Equation: Random Walk Functional}
F_{n,t}(S):=\mbf 1_{\{\tau^{(n)}(S)>\th\}}\prod_{a\in\mbb N_0}\left(1-\frac{D_n(a)}{m_n^2}\right)^{\La^{(a,a)}_\th(S)}
\left(1-\frac{U_n(a)}{m_n^2}\right)^{\La^{(a,a+1)}_\th(S)}\left(1-\frac{L_n(a)}{m_n^2}\right)^{\La_\th^{(a+1,a)}(S)},
\end{align}
and for any $x\geq0$, $\mbf E^{\lfloor m_nx\rfloor}$ denotes the expected value with respect to $S^{\lfloor m_nx\rfloor}$,
conditional on $Q_n$.

\subsubsection{Trace}

Letting $M$ be as in the previous section, it is easy to see that
\[\mr{Tr}\big[(\tfrac13 M)^\th\big]=\sum_{a=0}^n\mbf P[S^a(\th)=a]\mbf E^{a,a}_\th\left[
\mbf 1_{\{\tau^{(n)}(S)>\th\}}\,\prod_{u=0}^{\th-1} M\big(S(u),S(u+1)\big)\right]\]
where $S$ is independent of $M$, and
$\mbf E^{a,a}_\th$ denotes the expected value with respect to the law of $S^{a,a}_\th$,
conditional on $ M$. Given that
$\mbf P[S^a(\th)=a]=\mbf P[S^0(\th)=0]$
is independent of $a$, if we apply a Riemann sum on the grid $m_n^{-1}\mbb Z$ to the previous
expression for $\mr{Tr}\big[(\tfrac13 M)^\th\big]$, we note that
\[\mr{Tr}\big[(\tfrac13 M)^\th\big]=m_n\mbf P[S^0(\th)=0]
\,\int_0^{(n+1)/m_n}\mbf E^{\lfloor m_n x\rfloor,\lfloor m_n x\rfloor}_\th\left[
\mbf 1_{\{\tau^{(n)}(S)>\th\}}\,\prod_{u=0}^{\th-1} M\big(S(u),S(u+1)\big)\right]\d x.\]
Applying this to the model of interest $ \hat K_n(t)$,
we then see that
\begin{align}\label{Equation: Matrix-Trace Formula}
\mr{Tr}[ \hat K_n(t)]=m_n\mbf P[S^0(\th)=0]\int_0^{(n+1)/m_n}\mbf E^{\lfloor m_n x\rfloor,\lfloor m_n x\rfloor}_{\th}\big[F_{n,t}(S)\big]\d x,
\end{align}
where $\th=\th(n,t)=\lfloor m_n^2(3t/2)\rfloor$, $S$ is independent of $Q_n$,
$E^{\lfloor m_n x\rfloor,\lfloor m_n x\rfloor}_{\th}$ denotes the expected value of
$S^{\lfloor m_n x\rfloor,\lfloor m_n x\rfloor}_{\th}$ conditional on $Q_n$,
and $F_{n,t}$ is as in \eqref{Equation: Random Walk Functional}.

\subsection{Robin Boundary Condition: ``Reflected" Random Walk}

\begin{definition}\label{Definition: T Markov Process}
Let $T=\big(T(u)\big)_{u\in\mbb N_0}$ be the Markov chain on the state space $\mbb N_0$ with
the following transition probabilities:
\[\mbf P\big[T(u+1)=a+b\big|T(u)=a\big]=\frac13\qquad\text{if $a\in\mbb N_0\setminus\{0\}$ and $b\in\{-1,0,1\}$},\]
\[\mbf P\big[T(u+1)=0\big|T(u)=0\big]=\frac23,
\qquad\text{and}\qquad
\mbf P\big[T(u+1)=1\big|T(u)=0\big]=\frac13.\]
We denote $T^a:=\big(T|T(0)=a\big)$ and $T^{a,b}_u:=\big(T|T(0)=a\text{ and }T(u)=b\big)$.
\end{definition}

Let $M$ be a $(n+1)\times(n+1)$ tridiagonal matrix, and let $\tilde M$ be defined as
\[\tilde M(a,b)=\begin{cases}
\tfrac23M(a,b)&\text{if }a=b=0,\\
\tfrac13M(a,b)&\text{otherwise}.
\end{cases}\]
For any $\th\in\mbb N$, $0\leq a\leq n$, and vector $v\in\mbb R^{n+1}$,
\begin{align}\label{Equation: Spiked Combinatorial Analysis}
(\tilde M^\th v)(a)=\mbf E^a\left[\mbf 1_{\{\tau^{(n)}(T)>\th\}}\left(\prod_{a,b\in\mbb N_0} M(a,b)^{\La_\th^{(a,b)}(T)}\right)v\big(T(\th)\big)\right]
\end{align}
with $T$ independent of $M$,
$\mbf E^a$ denoting the expected value of $T^a$ conditioned on $M$,
and we define $\La^{(a,b)}_\th(T)$ in the same way as \eqref{Equation: Edge-Occupation Measure}.

We now apply this to the study of the matrix model $\hat K_n^w(t)$. The entries
of $I_n-(-\De^w_n+Q_n)/3m_n^2$ are the same as
\eqref{Equation: Dirichlet Diagonal Entry}--\eqref{Equation: Dirichlet Lower Diagonal Entry} except for the $(0,0)$ entry, which is equal to
\begin{multline*}
\left( I_n-\frac{-\De^w_n+ Q_n}{3m_n^2}\right)(0,0)=\frac13\left(1+w_n-\frac{D_n(0)}{m_n^2}\right)\\
=\frac13\left(2-(1-w_n)-\frac{D_n(0)}{m_n^2}\right)
=\frac23\left(1-\frac{(1-w_n)}2-\frac{D_n(0)}{2m_n^2}\right).
\end{multline*}
Therefore, if we let $\th=\th(n,t):=\lfloor m_n^2(3t/2)\rfloor$, then
\begin{align}\label{Equation: Spiked Matrix-Form Formula}
\langle f, \hat K^w_n(t)g\rangle
=\int_0^{(n+1)/m_n}f(x)\,\mbf E^{\lfloor m_nx\rfloor}\left[F^w_{n,t}(T)\, m_n\int_{T(\th)/m_n}^{(T(\th)+1)/m_n}g(y)\d y\right]\d x,
\end{align}
where $T$ is independent of $Q_n$,
we define the random functional
\begin{align}\label{Equation: Spiked Random Walk Functional}
F^w_{n,t}(T)&:=\mbf 1_{\{\tau^{(n)}(T)>\th\}}\left(1-\frac{(1-w_n)}2-\frac{D_n(0)}{2m_n^2}\right)^{\La_\th^{(0,0)}(T)}\\
\label{Equation: Spiked Random Walk Functional 2}
&\qquad\cdot\left(\prod_{a\in\mbb N}\left(1-\frac{D_n(a)}{m_n^2}\right)^{\La^{(a,a)}_\th(T)}\right)\\
\label{Equation: Spiked Random Walk Functional 3}
&\qquad\cdot\left(\prod_{a\in\mbb N_0}\left(1-\frac{U_n(a)}{m_n^2}\right)^{\La^{(a,a+1)}_\th(T)}\left(1-\frac{L_n(a)}{m_n^2}\right)^{\La_\th^{(a+1,a)}(T)}\right),
\end{align}
and $\mbf E^{\lfloor m_nx\rfloor}$ is the expected value of $T^{\lfloor m_nx\rfloor}$ conditional on $Q_n$.

\subsection{A Brief Comparison with Other Matrix Models}
\label{Subsection: Brief Comparison}

The assumptions made in Section \ref{Section: Main Results}
suggest that $\frac12(-\De_n+Q_n)\to\hat H$
and $\frac12(-\De_n^w+Q_n)\to\hat H^w$ as $n\to\infty$. Thus, by Remark \ref{Remark: Strong Continuity},
we expect that for any sequence of functions $(f_{n;t})_{n\in\mbb N}$
such that $f_{n;t}(x)\to\mr e^{-tx/2}$ in a suitable sense,
one has $f_{n;t}(-\De_n+Q_n)\to\hat K(t)$ and $f_{n;t}(-\De^w_n+Q_n)\to\hat K^w(t)$.
The difficulty involved in carrying this out rigorously in the generality aimed in this paper
is to choose $f_{n;t}$'s that are both
amenable to combinatorial analysis and applicable to general tridiagonal models.
The main insight of this paper is that the matrix models
$\hat K_n(t)$ and $\hat K^w_n(t)$
(which correspond to the choice $f_{n;t}(x):=(1-x/3m_n^2)^{\lfloor m_n^2(3t/2)\rfloor}$)
are in this sense better suited than arguably more ``obvious" choices of $f_{n;t}$.

In order to illustrate this claim, we compare our matrix models
with $f_{n;t}(x):=(1-x/2m_n^2)^{\lfloor m_n^2t/2\rfloor}$, which is what was used in
\cite{GorinShkolnikov,GaudreauLamarreShkolnikov}, and $f_{n;t}(x):=\mr e^{-tx/2}$,
which is arguably the most straightforward matrix model one could use in order to
obtain semigroup limits.
We begin with the latter:
If $Q_n$ is diagonal, then we can express the matrix exponential
$\mr e^{-t(-\De_n+Q_n)/2}$
in terms of a Feynman-Kac formula involving
the continuous-time simple random walk on $\mbb Z$ with exponential jump times.
This formula is very similar to \eqref{Equation: Dirichlet Semigroup}
and \eqref{Equation: Robin Semigroup} and is arguably easier to work with than
\eqref{Equation: Random Walk Functional} or \eqref{Equation: Spiked Random Walk Functional}.
However, for general tridiagonal $Q_n$, the Feynman-Kac formula becomes much more
unwieldy. In particular, the generator of the associated random walk depends on the entries of $Q_n$,
making a general unified treatment more difficult.

As for the matrix model used in \cite{GorinShkolnikov,GaudreauLamarreShkolnikov}, we note that
\begingroup
\allowdisplaybreaks
\begin{align*}
\textstyle\left( I_n-\frac{-\De_n+ Q_n}{2m_n^2}\right)(a,a)
&\textstyle=\frac{-D_n(a)}{m_n^2}
&0\leq a\leq n,\\
\textstyle\left( I_n-\frac{-\De_n+ Q_n}{2m_n^2}\right)(a,a+1)
&\textstyle=\frac12\left(1-\frac{U_n(a)}{m_n^2}\right)&0\leq a\leq n-1,\\
\textstyle\left( I_n-\frac{-\De_n+ Q_n}{2m_n^2}\right)(a+1,a)
&\textstyle=\frac12\left(1-\frac{L_n(a)}{m_n^2}\right)&0\leq a\leq n-1.
\end{align*}
\endgroup
If $D_n(a)=0$ for all $n$ and $a$, then a combinatorial analysis
similar to the one performed earlier in this section can relate the above
to a functional of simple symmetric random walks on $\mbb Z$
(i.e., i.i.d. uniform $\pm1$ increments). More generally, if $D_n(a)$ is of smaller order
than $m_n^2-U_n(a)$ and $m_n^2-L_n(a)$ for large $n$ (e.g., for $\be$-Hermite),
then a similar analysis holds, but with additional technical
difficulties (see \cite[Section 3.1]{GaudreauLamarreShkolnikov} and \cite[Section 3]{GorinShkolnikov} for
the details). However, if $D_n(a)$ is allowed to be of the same order as $m_n^2-U_n(a)$ and $m_n^2-L_n(a)$
(e.g., for $\be$-Laguerre),
then the analysis of \cite{GorinShkolnikov} and \cite{GaudreauLamarreShkolnikov} no longer applies.

\section{Strong Couplings for Theorem \ref{Theorem: Main Dirichlet}}\label{Section: Dirichlet Local Time Coupling}

Equations \eqref{Equation: Matrix-Form Formula} and \eqref{Equation: Matrix-Trace Formula}
suggest Theorem \ref{Theorem: Main Dirichlet}
relies on understanding how Brownian motion and its local time arises as the limit of the lazy
random walk and its edge-occupation measures. This is the subject of this section.

\begin{definition}
For every $x\geq0$,
let $\tilde B^x$ be a Brownian motion started at $x$ with
variance $2/3$,
and for every $t>0$, let $\tilde B^{x,x}_t:=\big(\tilde B^x|\tilde B^x(t)=x\big)$.
We define the local time process for $\tilde B$ in the same way as in \eqref{Equation: Interior Local Time}.
\end{definition}

The main result of this section is the following.

\begin{theorem}\label{Theorem: Split Coupling}
Let $t>0$ and $x\geq0$ be fixed. For every $0\leq s\leq t$ and $n\in\mbb N$,
let $\th_s=\th_s(n):=\lfloor m_n^2s\rfloor$ and $x^n:=\lfloor m_nx\rfloor$.
We use the shorthand $\th:=\th_t$.
For every $y\in\mbb R$, let $(y_n,\bar y_n)_{n\in\mbb N}$ be equal to one of the three sequences
\begin{align}\label{Equation: yn Sequences}
\big(\lfloor m_ny\rfloor,\lfloor m_ny\rfloor\big)_{n\in\mbb N},\,
\big(\lfloor m_ny\rfloor,\lfloor m_ny\rfloor+1\big)_{n\in\mbb N},\text{ or }\,\,
\big(\lfloor m_ny\rfloor+1,\lfloor m_ny\rfloor\big)_{n\in\mbb N}.
\end{align}
Finally, suppose that $(Z_n,Z)=(S^{x^n},\tilde B^x)$, or $(S^{x^n,x^n}_{\th},\tilde B^{x,x}_t)$
for each $n\in\mbb N$. For every $0<\eps<1/5$, there exists a coupling of $Z_n$ and $Z$ such that the following holds
almost surely as $n\to\infty$
\begin{align}
\label{Equation: KMT Coupling}
\sup_{0\leq s\leq t}\left|\frac{Z_n(\th_s)}{m_n}-Z(s)\right|
&=O\left(m_n^{-1}\log m_n\right),\\
\label{Equation: Local Time Coupling}
\sup_{0\leq s\leq t,~y\in\mbb R}\left|\frac{\La_{\th_s}^{(y_n,\bar y_n)}(Z_n)}{m_n}-\frac{L_s^y(Z)}3\right|
&=O\left(m_n^{-1/5+\eps}\log m_n\right).
\end{align}
\end{theorem}

Classical results on strong couplings of local time (such as \cite{BassKhoshnevisan})
concern the {\bf vertex-occupation measures} of a random walk:
\begin{align}\label{Equation: Vertex-Occupation Measure}
\La^a_u(S):=\sum_{j=0}^{u}\mbf 1_{\{S(j)=a\}},\qquad a\in\mbb Z,~u\in\mbb N.
\end{align}
Indeed, for any measurable $f:\mbb R\to\mbb R$,
the vertex-occupation measures satisfy
\begin{align}\label{Equation: Vertex-Occupation Property}
\sum_{j=0}^u f\big(S(j)\big)=\sum_{a\in\mbb Z}\La^a_u(S)\,f(a),
\end{align}
making a direct comparison with local time more convenient
by \eqref{Equation: Interior Local Time}. Thus, our strategy of proof for
Theorem \ref{Theorem: Split Coupling} has two steps: We first use
standard methods to construct a strong coupling of the vertex-occupation measures
of $S^{x^n}$ and $S^{x^n,x^n}_{\th}$ with the local time of their corresponding continuous
processes. Then, we prove that the occupation measure of a given edge $(a,b)$ is
very close to a multiple of the occupation measure of the vertices $a$ and $b$. More
precisely:

\begin{proposition}\label{Proposition: Strong Local Time Coupling}
For every $0<\eps<1/5$, there exists a coupling such that
\begin{align}\label{Equation: Strong Local Time Coupling}
\sup_{0\leq s\leq t,~y\in\mbb R}\left|\frac{\La_{\th_s}^{\lfloor m_ny\rfloor}(Z_n)}{m_n}-L_s^y(Z)\right|
=O\left(m_n^{-1/5+\eps}\log m_n\right)
\end{align}
and \eqref{Equation: KMT Coupling} hold almost surely as $n\to\infty$.
\end{proposition}

\begin{proposition}\label{Proposition: Splitting of Occupation Measure}
Almost surely, as $n\to\infty$, one has
\[\sup_{\substack{0\leq u\leq\th\\a,b\in\mbb Z,~|a-b|\leq 1}}
\frac1{m_n}\left|\La^{(a,b)}_u(Z_n)-\frac{\La^a_u(Z_n)}3\right|
=O\big(m_n^{-1/2}\log m_n\big).\]
\end{proposition}

\begin{notation}
In Propositions \ref{Proposition: Strong Local Time Coupling} and \ref{Proposition: Splitting of Occupation Measure},
and the remainder of Section \ref{Section: Dirichlet Local Time Coupling}, whenever we state
a result for $Z_n$ and $Z$, we mean that the result in question
applies to $(Z_n,Z)=(S^{x^n},\tilde B^x)$ and $(S^{x^n,x^n}_{\th},\tilde B^{x,x}_t)$.
\end{notation}

\subsection{Condition for Strong Local Time Coupling}

We begin with a criterion for local time couplings. The following lemma is essentially the content of the proof of
\cite[Theorem 3.2]{BassKhoshnevisan}; we provide a full proof since we need a
modification of the result in Section \ref{Section: Spiked Local Time Coupling}.

\begin{lemma}\label{Lemma: Local Time Invariance Criterion}
For any $0<\de<1$, the following holds almost surely as $n\to\infty$:
\begingroup
\allowdisplaybreaks
\begin{multline*}
\sup_{0\leq u\leq \th,~a\in\mbb Z}\left|\frac{\La_u^a(Z_n)}{m_n}-L_{u/m_n^2}^{a/m_n}(Z)\right|
=O\Bigg(\sup_{0\leq s\leq t,~|y-z|\leq m_n^{-\de}}\big|L^y_s(Z)-L^z_s(Z)\big|\\
+m_n^{2\de}\sup_{0\leq s\leq t}\left|\frac{Z_n(\th_s)}{m_n}-Z(s)\right|
+\sup_{0\leq u\leq\th,~|a-b|\leq m_n^{1-\de}}\frac{\big|\La^a_u(Z_n)-\La^b_u(Z_n)\big|}{m_n}
+\sup_{0\leq u\leq\th,~a\in\mbb Z}\frac{\La^a_u(Z_n)}{m_n^{2-\de}}\Bigg).
\end{multline*}
\endgroup
\end{lemma}

\begin{proof}
Let $n\in\mbb N$ and $a\in\mbb Z$ be fixed, and for each $\eps>0$, define the function
$f_\eps:\mbb R\to\mbb R$ as follows.
\begin{enumerate}
\item $f_\eps(a/m_n)=1/\eps$;
\item $f_\eps(z)=0$ whenever $|z-a/m_n|>\eps$; and
\item define $f_\eps(z)$ by linear interpolation for $|z-a/m_n|\leq\eps$.
\end{enumerate}
Since $f_\eps$ integrates to one, for every $0\leq u\leq\th$, we have that
\begin{multline*}
\left|\int_0^{u/m_n^2}f_\eps\big(Z(s)\big)\d s-L^{a/m_n}_{u/m_n^2}(Z)\right|
=\left|\int_{\mbb R}f_\eps(y)\big(L^y_{u/m_n^2}(Z)-L^{a/m_n}_{u/m_n^2}(Z)\big)\d y\right|\\
\leq\sup_{|y-a/m_n|\leq\eps}\big|L^y_{u/m_n^2}(Z)-L^{a/m_n}_{u/m_n^2}(Z)\big|.
\end{multline*}
Note that $|f_\eps(z)-f_\eps(y)|/|z-y|\leq\frac1{\eps^2}$ for all $z,y\in\mbb R$;
hence, for every $0\leq u\leq\th$,
\begin{multline*}
\left|\int_0^{u/m_n^2}f_\eps\big(Z(s)\big)\d s-\frac1{m_n^2}\sum_{j=1}^uf_\eps\big(Z_n(j)/m_n\big)\right|\\
=\left|\int_0^{u/m_n^2}f_\eps\big(Z(s)\big)-f_\eps\big(Z_n(\th_s)/m_n\big)\d s\right|
\leq\frac t{\eps^2}\sup_{0\leq s\leq u/m_n^2}
\left|\frac{Z_n(\th_s)}{m_n}-Z(s)\right|.
\end{multline*}
Finally,
\begin{multline*}
\frac1{m_n^2}\sum_{j=1}^uf_\eps\big(Z_n(j)/m_n\big)-\frac{\La^a_u(Z_n)}{m_n}
=\frac1{m_n^2}\sum_{b\in\mbb Z}f_\eps(b/m_n)\La^{b}_u(Z_n)-\frac{\La^a_u(Z_n)}{m_n}\\
=\frac1{m_n}\sum_{b\in\mbb Z}f_\eps(b/m_n)
\frac{\big(\La^{b}_u(Z_n)-\La^a_u(Z_n)\big)}{m_n}
+\frac{\La^a_u(Z_n)}{m_n}\left(\frac1{m_n}\sum_{b\in\mbb Z}f_\eps(b/m_n)-1\right).
\end{multline*}
By a Riemann sum approximation,
\[\left|\frac1{m_n}\sum_{b\in\mbb Z}f_\eps(b/m_n)-1\right|=O\left(\frac1{\eps m_n}\right),\]
and thus we conclude that
\[\frac1{m_n^2}\sum_{j=1}^uf_\eps\big(Z_n(j)/m_n\big)-\frac{\La^a_u(Z_n)}{m_n}
=O\left(\sup_{|a-b|\leq\eps m_n}\frac{\big|\La^{b}_u(Z_n)
-\La^a_u(Z_n)\big|}{m_n}+\frac{\La^a_u(Z_n)}{\eps m_n^2}\right).\]r
The result then follows by taking a supremum over $0\leq u\leq \th$ and $a\in\mbb Z$,
and taking $\eps=\eps(n)=m_n^{-\de}$. 
\end{proof}

\subsection{Proof of Proposition \ref{Proposition: Strong Local Time Coupling}}
\label{Subsubsection: Local Time Regularity}

We begin with the proof of \eqref{Equation: KMT Coupling}:

\begin{lemma}\label{Lemma: KMT}
There exists a coupling such that \eqref{Equation: KMT Coupling}
holds.
In particular, for any $0<\de<1/2$, it holds almost surely as $n\to\infty$ that
\[m_n^{2\de}\sup_{0\leq s\leq t}\left|\frac{Z_n(\th_s)}{m_n}-Z(s)\right|=O(m_n^{-1+2\de}\log m_n).\]
\end{lemma}
\begin{proof}
Suppose first that $x=0$ so that $(Z_n,Z)=(S^0,\tilde B^0)$ or $(S_{\th}^{0,0},\tilde B^{0,0}_t)$.
According to the classical KMT coupling
(e.g., \cite[Section 7]{LawlerLimic}) for Brownian motion and its extension
to the Brownian bridge (e.g., \cite[Theorem 2]{Borisov}), it  holds that
\[\sup_{0\leq u\leq\th}\left|\frac{Z_n(u)}{m_n}-Z(u/m_n^2)\right|=O(m_n^{-1}\log m_n)\]
almost surely. Thus it only remains to prove that
\[\sup_{0\leq s\leq t}\left|Z(\th_s/m_n^2)-Z(s)\right|=O(m_n^{-1}\log m_n).\]
For $Z=\tilde B^0$, this is
L\'evy's modulus of continuity theorem. For $Z=\tilde B^{0,0}_t$, we note that
the laws of $\big(\tilde B^{0,0}_t(s)\big)_{s\in[0,t/2]}$ and
$\big(\tilde B^{0,0}_t(t-s)\big)_{s\in[0,t/2]}$ are absolutely continuous with respect
the the law of $\big(\tilde B^0(s)\big)_{s\in[0,t/2]}$.

Suppose now that $x>0$. We can define $S^{x^n}:=x^n+S^0$ and $S^{x^n,x^n}_{\th}:=x^n+S^{0,0}_{\th}$,
and similarly for $\tilde B$. Since $x^n/m_n=x+O(m_n^{-1})$, our proof in the case $x=0$ yields
\[\sup_{0\leq s\leq t}\left|\frac{S^{x^n}(\th_s)}{m_n}-\tilde B^x(s)\right|=
O(m_n^{-1}\log n+m_n^{-1})\]
and similarly for the bridge, as desired.
\end{proof}

With \eqref{Equation: KMT Coupling} established,
the proof \eqref{Equation: Strong Local Time Coupling} is a straightforward
application of Lemma \ref{Lemma: Local Time Invariance Criterion}:

\begin{lemma}\label{Lemma: Local Time Regularity}
For every $\de>0$ and $0<\eps<\de/2$,
\[\sup_{0\leq s\leq t,~|y-z|\leq m_n^{-\de}}\big|L^y_s(Z)-L^z_s(Z)\big|
=O\left(m_n^{-\de/2+\eps}\log m_n\right)\]
almost surely as $n\to\infty$.
\end{lemma}
\begin{proof}
The result for $\tilde B^x$ is a direct application of
\cite[Equation (3.7)]{BassKhoshnevisan} (see also \cite[(2.1)]{Trotter}).
We obtain the same result for $\tilde B^{x,x}_t$ by the absolute continuity
of $\big(\tilde B^{x,x}_t(s)\big)_{s\in[0,t/2]}$ and
$\big(\tilde B^{x,x}_t(t-s)\big)_{s\in[0,t/2]}$ with respect to $\big(\tilde B^x(s)\big)_{s\in[0,t/2]}$,
and the fact that local time is additive and invariant under time reversal.
\end{proof}

\begin{lemma}\label{Lemma: Occupation Measure Regularity}
For every $\de>0$ and $0<\eps<\de/2$,
\[\sup_{0\leq u\leq\th,~|a-b|\leq m_n^{1-\de}}\frac{\big|\La^a_u(Z_n)-\La^b_u(Z_n)\big|}{m_n}
=O\left(m_n^{-\de/2+\eps}\log m_n\right)\]
almost surely as $n\to\infty$.
\end{lemma}
\begin{proof}
According to \cite[Proposition 3.1]{BassKhoshnevisan}, for every $0<\eta<1/2$, it holds that
\begin{align}\label{Equation: Unconditioned Occupation Measure Regularity}
\mbf P\left[\sup_{0\leq u\leq\th,~a,b\in\mbb Z}\frac{m_n^{-1}\big|\La^a_u(S^{x^n})-\La^b_u(S^{x^n})\big|}
{(|a/m_n-b/m_n|^{1/2-\eta}\land 1)}\geq\la\right]=O(\mr e^{-c\la}+m_n^{-14})
\end{align}
for every $\la>0$, where $c>0$ is independent of $n$ and $\la$. We recall that $m_n\asymp n^{\mf d}$
with $1/13<\mf d$, which implies in particular that $m_n^{-14}$ is summable in $n$.
Thus, if we take $\la=\la(n)=C\log m_n$ for a large enough $C>0$, then Borel-Cantelli yields
\[\sup_{0\leq u\leq\th,~|a-b|\leq m_n^{1-\de}}\frac{\big|\La_u^a(S^{x^n})-\La_u^b(S^{x^n})\big|}{m_n}
=O\left(m_n^{\de(\eta-1/2)}\log m_n\right)\]
almost surely, proving the result for $Z_n=S^{x^n}$.
In order to extend the result to $Z_n=S_{\th}^{x^n,x^n}$ we apply the local CLT
(i.e., $\mbf P[S^{x^n}(\th)=x^n]^{-1}=O(m_n)$; e.g., \cite[\S 49]{GnedenkoKolmogorov}) with
the elementary inequality $\mbf P[E_1|E_2]\leq\mbf P[E_1]/\mbf P[E_2]$
to \eqref{Equation: Unconditioned Occupation Measure Regularity}:
\[\mbf P\left[\sup_{0\leq u\leq\th,~a,b\in\mbb Z}\frac{m_n^{-1}\big|\La^a_u(S^{x^n,x^n}_{\th})-\La^b_u(S^{x^n,x^n}_{\th})\big|}
{(|a/m_n-b/m_n|^{1/2-\eta}\land 1)}\geq\la\right]=O(m_n\mr e^{-c_2\la}+m_n^{-13})\]
for all $\la>0$. Since $\sum_nm_n^{-13}<\infty$ the result follows by Borel-Cantelli.
\end{proof}

\begin{lemma}\label{Lemma: Occupation Measure Size}
For every $0<\de<1$, it holds almost surely as $n\to\infty$ that
\[\sup_{0\leq u\leq\th,~a\in\mbb Z}\frac{\La^a_u(Z_n)}{m_n^{2-\de}}=O\left(m_n^{-1+\de}\log m_n\right).\]
\end{lemma}
\begin{proof}
Note that, for any $n,u\in\mbb N$, $|S^{x^n}(u)|\leq|x^n|+O(m_n^2)$
Therefore, by taking a large $b$ in
\eqref{Equation: Unconditioned Occupation Measure Regularity} (i.e., large enough so that
$\La_{\th}^b(S^{x^n})=0$ surely), we see that
\begin{align}\label{Equation: Unconditioned Occupation Measure Tail}
\mbf P\left[\sup_{0\leq u\leq\th,~a\in\mbb Z}\frac{\La^a_u(S^{x^n})}{m_n}
\geq\la\right]=O(\mr e^{-C\la}+m_n^{-14})
\end{align}
for all $\la>0$. The proof then follows from the same arguments as in
Lemma \ref{Lemma: Occupation Measure Regularity}.
\end{proof}

By combining Lemmas \ref{Lemma: Local Time Invariance Criterion}--\ref{Lemma: Occupation Measure Size},
we obtain that
\[\sup_{0\leq u\leq\th,~a\in\mbb Z}\left|\frac{\La_u^a(Z_n)}{m_n}-L_{u/m_n^2}^{a/m_n}(Z)\right|
=O(m_n^\mf t\log m_n),\]
where, for every $0<\de<1/2$ and $0<\eps<\de/2$, we have
\[\mf t=\mf t(\de,\eps):=\max\big\{-1+2\de,-\de/2+\eps\big\}.\]
For any fixed $\eps>0$, the smallest possible $\mf t(\de,\eps)$ occurs at the intersection of the lines
$\de\mapsto-1+2\de$ and $\de\mapsto-\de/2+\eps$.
This is attained at $\de=2(1+\eps)/5$, in which case $\mf t=-1/5+4\eps/5$.
At this point, in order to get the statement of Proposition \ref{Proposition: Strong Local Time Coupling},
we must show that
\[\sup_{0\leq s\leq t,~y\in\mbb R}\left|L_{\th_s/m_n^2}^{\lfloor m_ny\rfloor/m_n}(Z)-L_s^y(Z)\right|=O(m_n^{-1/5+\eps}\log m_n)\]
as $n\to\infty$ for any $\eps>0$. This follows by a combination of Lemma \ref{Lemma: Local Time Regularity} with $\de=1$
and the estimate \cite[(2.3)]{Trotter}, which yields
\[\sup_{0\leq s,\bar s\leq t,~|s-\bar s|\leq m_n^{-2},~y\in\mbb R}\big|L^y_s(Z)-L^y_{\bar s}(Z)\big|=O\big(m_n^{-2/3}(\log m_n)^{2/3}\big).\]
(The results of \cite{Trotter} are only stated for the Brownian motion, but this can be extended to the Bridge by the absolute
continuity argument used in Lemma \ref{Lemma: Local Time Regularity}.)

\subsection{Proof of Proposition \ref{Proposition: Splitting of Occupation Measure}}
\label{Subsection: Splitting of Occupation Measure}

We may assume without loss of generality that $x=0$.
We begin with the case of the unconditioned random walk $Z_n=S^0$.

Let $(\ze^a_n)_{n\in\mbb N_0,a\in\mbb Z}$ be a collection of i.i.d. random variables with
uniform distribution on $\{-1,0,1\}$. We can define the random walk $S^0$ as follows: For
every $u,v\in\mbb N$ and $a\in\mbb Z$, if $S^0(u)=a$ and $\La^a_u(S^0)=v$,
then $S^0(u+1)=S^0(u)+\ze^a_v$. In doing so, up to an error of
at most 1, it holds that
\[\La^{(a,b)}_n(S^0)=\sum_{j=1}^{\La_n^a(S^0)}\mbf 1_{\{\ze^a_j=b-a\}},\qquad a,b\in\mbb Z.\]
Hence, by the Borel-Cantelli lemma,
it is enough to show that for any $z\in\{-1,0,1\}$,
\begin{align}\label{Equation: Summable Deviation 1}
\sum_{n\in\mbb N}\mbf P\left[\sup_{0\leq u\leq\th,~a\in\mbb Z}\left|\sum_{j=1}^{\La_u^a(S^0)}\mbf 1_{\{\ze^a_j=z\}}-\frac{\La^a_u(S^0)}3\right|
\geq Cm_n^{1/2}\log m_n\right]<\infty
\end{align}
for some suitable finite constant $C>0$.
In order to prove this, we need two auxiliary estimates.
Let us denote the range of a random walk by
\begin{align}\label{Equation: Random Walk Range}
\mc R_{u}(S):=\max_{0\leq j\leq u}S(j)-\min_{0\leq j\leq u}S(j),\qquad u\in\mbb N_0.
\end{align}

\begin{lemma}\label{Lemma: Range Deviation}
For every $\eps>0$,
\[\sum_{n\in\mbb N}\mbf P\left[\mc R_\th(S^0)\geq m_n^{1+\eps}\right]
<\infty.\]
\end{lemma}
\begin{proof}
According to \cite[(6.2.3)]{Chen}, there exists $C>0$ independent of $n$ such that
\begin{align}\label{Equation: Range Moments}
\mbf E\big[\mc R_{\th}(S^0)^q\big]\leq (Cm_n)^q\sqrt{q!},\qquad q\in\mbb N_0.
\end{align}
Consequently, for every $r<2$ and $C>0$,
\begin{align}\label{Equation: Range Exponential Moments}
\sup_{n\in\mbb N}\mbf E\left[\mr e^{C(\mc R_{\th}(S^0)/m_n)^r}\right]<\infty,
\end{align}
The result then follows from Markov's inequality.
\end{proof}

\begin{lemma}\label{Lemma: Local Time Deviation}
If $C>0$ is large enough,
\[\sum_{n\in\mbb N}\mbf P\left[\sup_{a\in\mbb Z}\La^a_\th(S^0)\geq Cm_n\log m_n\right]<\infty.\]
\end{lemma}
\begin{proof}
This follows directly from \eqref{Equation: Unconditioned Occupation Measure Tail}
since $\sum_nm_n^{-14}<\infty$.
\end{proof}

According to Lemmas \ref{Lemma: Range Deviation} and \ref{Lemma: Local Time Deviation},
to prove \eqref{Equation: Summable Deviation 1},
it is enough to consider the sum of probabilities in question intersected with the events
\[D_n:=\left\{\mc R_\th(S^0)\leq m_n^{1+\eps},
~\sup_{0\leq u\leq\th,~a\in\mbb Z}\La^a_n(S^0)\leq Cm_n\log m_n\right\}\]
for some large enough $C>0$. By a union bound,
\begingroup
\allowdisplaybreaks
\begin{align}\label{Equation: Summable Deviation 2}
&\mbf P\left[\left\{\sup_{0\leq u\leq\th,~a\in\mbb Z}\left|\sum_{u=1}^{\La_u^a(S^0)}\mbf 1_{\{\ze^a_j=z\}}-\frac{\La_u^a(S^0)}3\right|
\geq cm_n^{1/2}\log m_n\right\}\,\cap\, D_n\right]
\nonumber\\
&\leq\mbf P\left[\left\{\max_{\substack{-m_n^{1+\eps}\leq a\leq m_n^{1+\eps}\\0\leq h\leq Cm_n\log m_n}}
\left|\sum_{j=1}^h\mbf 1_{\{\ze^a_j=z\}}-\frac h3\right|
\geq cm_n^{1/2}\log m_n\right\}\,\cap\, D_n\right]
\nonumber\\
&\leq\sum_{\substack{-m_n^{1+\eps}\leq a\leq m_n^{1+\eps}\\0\leq h\leq Cm_n\log m_n}}
\mbf P\left[\left|\sum_{j=1}^h\mbf 1_{\{\ze^a_j=z\}}-\frac h3\right|\geq cm_n^{1/2}\log m_n\right].
\end{align}
\endgroup
By Hoeffding's inequality,
\[\mbf P\left[\left|\sum_{j=1}^h\mbf 1_{\{\ze^a_j=z\}}-\frac h3\right|\geq \bar Cm_n^{1/2}\log m_n\right]\leq 2\mr e^{-2\bar C^2\log m_n/C}\]
uniformly in $0\leq h\leq Cm_n\log m_n$.
Since the sum in \eqref{Equation: Summable Deviation 2}
involves a polynomially bounded number of summands in $m_n$
and the latter grows like a power of $n$,
\begin{align}\label{Equation: Summable Deviation 3}
\text{for any $q>0$,
we can choose $\bar C>0$ so that \eqref{Equation: Summable Deviation 1}
is of order $O(n^{-q})$.}
\end{align}
This concludes the proof of Proposition
\ref{Proposition: Splitting of Occupation Measure} in the case $Z_n=S^0$
by Borel-Cantelli.

In order to extend the result to the case $Z_n=S_{\th}^{x^n,x^n}$, it suffices to prove that
\eqref{Equation: Summable Deviation 1} holds with the additional conditioning $\{S^0(\th)=0\}$.
The same local limit theorem argument used at the end of the proof of Lemma
\ref{Lemma: Occupation Measure Regularity} implies that
\begin{multline*}
\mbf P\left[\sup_{0\leq u\leq\th,~a\in\mbb Z}\left|\sum_{j=1}^{\La_u^a(S^{x^n,x^n}_{\th})}\mbf 1_{\{\ze^a_j=z\}}-\frac{\La^a_u(S^{x^n,x^n}_{\th})}3\right|
\geq Cm_n^{1/2}\log m_n\right]\\
=O\left(m_n \, \mbf P\left[\sup_{0\leq u\leq\th,~a\in\mbb Z}\left|\sum_{j=1}^{\La_u^a(S^0)}\mbf 1_{\{\ze^a_j=z\}}-\frac{\La^a_u(S^0)}3\right|
\geq Cm_n^{1/2}\log m_n\right]\right).
\end{multline*}
The result then follows from \eqref{Equation: Summable Deviation 3} by taking
a large enough $q$.

\section{Strong Coupling for Theorem \ref{Theorem: Main Robin}}\label{Section: Spiked Local Time Coupling}

We now provide the counterpart of Theorem \ref{Theorem: Split Coupling}
for the Markov chain $T$ in Definition \ref{Definition: T Markov Process} that is needed for
Theorem \ref{Theorem: Main Robin}.

\begin{definition}
Let $\tilde X$ be a reflected Brownian motion on $\mbb R_+$ with variance $2/3$.
For every $x\geq0$, we denote $\tilde X^x:=\big(\tilde X|\tilde X(0)=x\big)$, and we define the local time and the boundary
local time of $\tilde X$ as in \eqref{Equation: Interior Local Time} and \eqref{Equation: Boundary Local Time},
respectively.
\end{definition}

Our main result in this section is the following.

\begin{theorem}\label{Theorem: Split Skorokhod Couplings}
Let $t>0$ and $x\geq0$ be fixed. Let
$\th$, $\th_s$ ($0\leq s\leq t$), $x^n$, and $(y_n,\bar y_n)$ ($y>0$)
be as in Theorem \ref{Theorem: Split Coupling}.
For every $0<\eps<1/5$, there exists a coupling of $T^{x^n}$ and $\tilde X^x$ such that
\begin{align}
\label{Equation: Skorokhod Coupling}
\sup_{0\leq s\leq t}\left|\frac{T^{x^n}(\th_s)}{m_n}-\tilde X^x(s)\right|
=O\left(m_n^{-1}\log m_n\right),\\
\label{Equation: Skorokhod Local Time at Zero}
\left|\frac{\La^{(0,0)}_{\th}(T^{x^n})}{m_n}-\frac{4\mf L^0_t(\tilde X^x)}3\right|=O\left(m_n^{-1/2}(\log m_n)^{3/4}\right),\\
\label{Equation: Skorokhod Local Time Coupling}
\sup_{0\leq s\leq t,~y>0}\left|\frac{\La_{\th_s}^{(y_n,\bar y_n)}(T^{x^n})}{m_n}(1-\tfrac12\mbf 1_{\{(y_n,\bar y_n)=(0,0)\}})-\frac{L_s^y(\tilde X^{x})}3\right|
=O\left(m_n^{-1/5+\eps}\log m_n\right)
\end{align}
almost surely as $n\to\infty$.
\end{theorem}

\begin{remark}
In contrast with Theorem \ref{Theorem: Split Coupling},
Theorem \ref{Theorem: Split Skorokhod Couplings} does not
include a strong invariance result for the $T$'s bridge process $T^{x^n,x^n}_\th$.
We discuss this omission (and state a related conjecture)
in Section \ref{Subsection: Bridge Coupling Conjecture} below.
\end{remark}

The first step in the proof for Theorem \ref{Theorem: Split Skorokhod Couplings} is to use a modification of the Skorokhod reflection trick developed in \cite[Section 2]{GaudreauLamarreShkolnikov}
to reduce \eqref{Equation: Skorokhod Coupling} to the KMT coupling stated in \eqref{Equation: KMT Coupling}.
As it turns out, this
step also provides a proof of \eqref{Equation: Skorokhod Local Time at Zero}.
The second step is to introduce
a suitable modification of Lemma \ref{Lemma: Local Time Invariance Criterion} that provides a criterion
for the strong convergence of the vertex-occupation measures of $T$
with the local time of $\tilde X$. The third step is to prove an analog of Proposition
\ref{Proposition: Splitting of Occupation Measure}.
We summarize the last two steps in the following propositions:

\begin{proposition}\label{Proposition: Reflected Splitting of Occupation Measure}
Almost surely, as $n\to\infty$, one has
\[\sup_{\substack{0\leq u\leq\th\\(a,b)\in\mbb N_0^2\setminus\{(0,0)\},~|a-b|\leq 1}}
\frac1{m_n}\left|\La^{(a,b)}_u(T^{x^n})-\frac{\La^a_u(T^{x^n})}3\right|
=O\big(m_n^{-1/2}\log m_n\big),\]
and
\[\sup_{\substack{0\leq u\leq\th}}
\frac1{m_n}\left|\La^{(0,0)}_u(T^{x^n})-\frac{2\La^0_u(T^{x^n})}3\right|
=O\big(m_n^{-1/2}\log m_n\big).\]
\end{proposition}

\begin{proposition}\label{Proposition: Reflected Strong Local Time Coupling}
For every $0<\eps<1/5$, under the same coupling as \eqref{Equation: Skorokhod Coupling},
it holds almost surely as $n\to\infty$ that
\[\sup_{0\leq s\leq t,~y>0}\left|\frac{\La_{\th_s}^{\lfloor m_ny\rfloor}(T^{x^n})}{m_n}-L_s^y(\tilde X^x)\right|
=O\left(m_n^{-1/5+\eps}\log m_n\right).\]
\end{proposition}

\subsection{Proof of \eqref{Equation: Skorokhod Coupling}}
\label{Subsection: Skorokhod Coupling}

\begin{definition}[Skorokhod Map]
Let $Z=\big(Z(t)\big)_{t\geq0}$ be a continuous-time stochastic process.
We define the Skorokhod map of $Z$, denoted $\Gamma_Z$, as the process
\begin{equation*}
\Gamma_Z(t):=Z(t)+\sup_{s\in[0,t]}\,\big(-Z(s)\big)_+,\qquad t\geq0,
\end{equation*}
where $(\cdot)_+:=\max\{0,\cdot\}$ denotes the positive part of a real number.
\end{definition}

\begin{notation}
In the sequel, whenever we discuss the Skorokhod map of the random walk $S$, $\Ga_S$,
we mean the Skorokhod map applied to the continuous-time process $s\mapsto S(\th_s)$
for $0\leq s\leq t$.
\end{notation}

Note that $Z\mapsto\Ga_Z$ is 2-Lipschitz with respect to the supremum norm
on compact time intervals. Therefore, \eqref{Equation: Skorokhod Coupling}
is a direct consequence of \eqref{Equation: KMT Coupling}
if we provide couplings $(T,S)$ and $(\tilde X,\tilde B)$
such that $T^{x_n}(\th_s)=\Ga_{S^{x_n}}(s)$ and $\tilde X^x(s)=\Ga_{\tilde B^x}(s)$.

Let us begin with the coupling of $\tilde X^x$ and $\tilde B^x$.
Note that we can define $\tilde X^x:=|\bar B^x|$,
where $\bar B$ is a Brownian motion with variance $2/3$.
Since the quadratic variation of $\bar B^x$ is $t\mapsto (2/3)t$,
it follows from Tanaka's formula that
\begin{align*}
\tilde X^x(t)=x+\int_0^t\mr{sgn}(\bar B^x(s))\d\bar B^x(s)+\frac{2\mf L_t^{0}(\bar B^x)}3,\qquad t\geq0
\end{align*}
(e.g., \cite[Chapter VI, Theorem 1.2 and Corollary 1.9]{RevuzYor}),
where
\[\mf L_t^{0}(\bar B^x):=\lim_{\eps\to0}\frac1{2\eps}\int_0^t\mbf 1_{\{-\eps<\bar B^x(s)<\eps\}}\d s=\mf L_t^{0}(\tilde X^x).\]
If we define
\begin{align*}
\tilde B^x_t:=x+\int_0^t\mr{sgn}(\bar B^0(s))\d\bar B^0(s),\qquad t\geq0,
\end{align*}
which is a Brownian motion with variance $2/3$ started at $x$, then
we get from \cite[Chapter VI, Lemma 2.1 and Corollary 2.2]{RevuzYor} that
$\tilde X^x_t=\Gamma_{\tilde B^x}(t)$ and
\begin{align}\label{Equation: Skorokhod Min}
\textstyle
(2/3)\mf L_t^{0}(\tilde X^x)=\sup_{s\in[0,t]}\big(-\tilde B^x(s)\big)_+
\end{align}
for every $t\geq0$, as desired.

We now provide the coupling of $T^{x^n}$ and $S^{x^n}$.
(See Figure \ref{Figure: Skorokhod Coupling} below for
an illustration of the procedure we are about to describe.)
Let $\mc C$ be the set of step functions of the form
\begin{align}\label{Equation: Skorokhod Step Functions}
A(s)=\sum_{u=0}^{\th}A_u\mbf1_{[u,u+1)}(s),
\end{align}
where $A_0,A_1,\ldots,A_{\th}\in\mbb Z$ are such that $A_0=x^n$
and $A_{u+1}-A_u\in\{-1,0,1\}$ for all $u$. Let $\mc C_+\subset\mc C$
be the subset of such functions that are nonnegative. For every $A\in\mc C$,
let us define
\begin{align}\label{Equation: Horizontal at Zero}
\mc H_0(A):=\sum_{u=0}^{\th-1}\mbf1_{\{A_u=A_{u+1}=0\}}.
\end{align}
By definition of $S$ and $T$, we see that for any $A\in\mc C$,
\begin{align*}
\mbf P\left[\big(S^{x^n}(\th_s)\big)_{0\leq s\leq t}=\big(A(\th_s)\big)_{0\leq s\leq t}\right]&=\frac1{3^{\th}},\\
\mbf P\left[\big(T^{x^n}(\th_s)\big)_{0\leq s\leq t}=\big(A(\th_s)\big)_{0\leq s\leq t}\right]&=\frac{2^{\mc H_0(A)}}{3^{\th}}\mbf 1_{\{A\in\mc C_+\}}.
\end{align*}
It is clear that $A\mapsto\Gamma_A$ maps $\mc C$ to $\mc C_+$ and that this map is surjective since
$\Gamma_A=A$ for any $A\in\mc C_+$.
Thus, in order to construct a coupling such that $T^{x_n}(\th_s)=\Ga_{S^{x_n}}(s)$,
it suffices to show that for every $A\in\mc C_+$, there are exactly $2^{\mc H_0(A)}$ distinct functions
$\tilde A\in\mc C$ such that $\Ga_{\tilde A}=A$.

Let $A\in \mc C_+$. If $\mc H_0(A)=0$, then there is no $\tilde A\neq A$ such that $\Gamma_{\tilde A}=A$, as desired.
Suppose then that $\mc H_0(A)=h>0$. Let $0\leq u_1,\ldots,u_h\leq \th-1$ be the integer coordinates such that $A_{u_j}=A_{u_j+1}=0$, $1\leq j\leq h$.
Then, $\Ga_{\tilde A}=A$ if and only if the following conditions hold:
\begin{enumerate}
\item $\tilde A_{u_j+1}-\tilde A_{u_j}=0$ or $\tilde A_{u_j+1}-\tilde A_{u_j}=-1$ for all $1\leq j\leq h$, and
\item $\tilde A_{u+1}-\tilde A_u=A_{u+1}-A_u$ for all integers $u$ such that $u\not\in\{u_1,\ldots,u_h\}$.
\end{enumerate}
Note that, up to choosing whether the increments $\tilde A_{u_j+1}-\tilde A_{u_j}$ ($1\leq j\leq h$)
are equal to $0$ or $-1$,
the above conditions completely determine $\tilde A$. Moreover, there are $2^h$ ways of choosing these increments,
each of which yields a different $\tilde A$. Therefore, there are $2^{\mc H_0(A)}$ distinct functions $\tilde A\in\mc C$
such that $\Ga_{\tilde A}=A$, as desired.
\begin{figure}[htbp]
\begin{center}
%
%
%
\begin{tikzpicture}

\begin{axis}[%
at={(0,0)},
width=2in,
height=0.5in,
scale only axis,
xmin=0.8,
xmax=19.1,
ymin=-2.3,
ymax=3.3,
xtick=\empty,
axis background/.style={fill=white}
]
\addplot [color=gray,dashed,forget plot]
  table[row sep=crcr]{%
0	0\\
20	0\\
};
\node[circle,fill,inner sep=1pt] at (axis cs:1,2) {};
\addplot [ forget plot,thick]
  table[row sep=crcr]{%
1	2\\
2	2\\
};
\node[circle,fill,inner sep=1pt] at (axis cs:2,3) {};
\addplot [ forget plot,thick]
  table[row sep=crcr]{%
2	3\\
3	3\\
};
\node[circle,fill,inner sep=1pt] at (axis cs:3,3) {};
\addplot [ forget plot,thick]
  table[row sep=crcr]{%
3	3\\
4	3\\
};
\node[circle,fill,inner sep=1pt] at (axis cs:4,2) {};
\addplot [ forget plot,thick]
  table[row sep=crcr]{%
4	2\\
5	2\\
};
\node[circle,fill,inner sep=1pt] at (axis cs:5,1) {};
\addplot [ forget plot,thick]
  table[row sep=crcr]{%
5	1\\
6	1\\
};
\node[circle,fill,inner sep=1pt] at (axis cs:6,1) {};
\addplot [ forget plot,thick]
  table[row sep=crcr]{%
6	1\\
7	1\\
};
\node[circle,fill,inner sep=1pt] at (axis cs:7,2) {};
\addplot [ forget plot,thick]
  table[row sep=crcr]{%
7	2\\
8	2\\
};
\node[circle,fill,inner sep=1pt] at (axis cs:8,1) {};
\addplot [ forget plot,thick]
  table[row sep=crcr]{%
8	1\\
9	1\\
};
\node[circle,fill,inner sep=1pt] at (axis cs:9,1) {};
\addplot [ forget plot,thick]
  table[row sep=crcr]{%
9	1\\
10	1\\
};
\node[circle,fill,inner sep=1pt] at (axis cs:10,0) {};
\addplot [forget plot,thick]
  table[row sep=crcr]{%
10	0\\
11	0\\
};
\node[color=blue,circle,fill,inner sep=1pt] at (axis cs:11,0) {};
\addplot [color=blue, forget plot,thick]
  table[row sep=crcr]{%
11	0\\
12	0\\
};
\node[color=blue,circle,fill,inner sep=1pt] at (axis cs:12,0) {};
\addplot [color=blue,forget plot,thick]
  table[row sep=crcr]{%
12	0\\
13	0\\
};
\node[circle,fill,inner sep=1pt] at (axis cs:13,1) {};
\addplot [ forget plot,thick]
  table[row sep=crcr]{%
13	1\\
14	1\\
};
\node[circle,fill,inner sep=1pt] at (axis cs:14,0) {};
\addplot [forget plot,thick]
  table[row sep=crcr]{%
14	0\\
15	0\\
};
\node[color=blue,circle,fill,inner sep=1pt] at (axis cs:15,0) {};
\addplot [color=blue,forget plot,thick]
  table[row sep=crcr]{%
15	0\\
16	0\\
};
\node[circle,fill,inner sep=1pt] at (axis cs:16,1) {};
\addplot [ forget plot,thick]
  table[row sep=crcr]{%
16	1\\
17	1\\
};
\node[circle,fill,inner sep=1pt] at (axis cs:17,2) {};
\addplot [ forget plot,thick]
  table[row sep=crcr]{%
17	2\\
18	2\\
};
\node[circle,fill,inner sep=1pt] at (axis cs:18,2) {};
\addplot [ forget plot,thick]
  table[row sep=crcr]{%
18	2\\
19	2\\
};
\end{axis}

\node at (5.5,1.4) {$\swarrow$};
\node at (5.5,-0.4) {$\nwarrow$};

\begin{axis}[%
at={(230,500)},
width=2in,
height=0.5in,
scale only axis,
xmin=0.8,
xmax=19.1,
ymin=-2.3,
ymax=3.3,
xtick=\empty,
axis background/.style={fill=white}
]
\addplot [color=gray,dashed,forget plot]
  table[row sep=crcr]{%
0	0\\
20	0\\
};
\node[circle,fill,inner sep=1pt] at (axis cs:1,2) {};
\addplot [ forget plot,thick]
  table[row sep=crcr]{%
1	2\\
2	2\\
};
\node[circle,fill,inner sep=1pt] at (axis cs:2,3) {};
\addplot [ forget plot,thick]
  table[row sep=crcr]{%
2	3\\
3	3\\
};
\node[circle,fill,inner sep=1pt] at (axis cs:3,3) {};
\addplot [ forget plot,thick]
  table[row sep=crcr]{%
3	3\\
4	3\\
};
\node[circle,fill,inner sep=1pt] at (axis cs:4,2) {};
\addplot [ forget plot,thick]
  table[row sep=crcr]{%
4	2\\
5	2\\
};
\node[circle,fill,inner sep=1pt] at (axis cs:5,1) {};
\addplot [ forget plot,thick]
  table[row sep=crcr]{%
5	1\\
6	1\\
};
\node[circle,fill,inner sep=1pt] at (axis cs:6,1) {};
\addplot [ forget plot,thick]
  table[row sep=crcr]{%
6	1\\
7	1\\
};
\node[circle,fill,inner sep=1pt] at (axis cs:7,2) {};
\addplot [ forget plot,thick]
  table[row sep=crcr]{%
7	2\\
8	2\\
};
\node[circle,fill,inner sep=1pt] at (axis cs:8,1) {};
\addplot [ forget plot,thick]
  table[row sep=crcr]{%
8	1\\
9	1\\
};
\node[circle,fill,inner sep=1pt] at (axis cs:9,1) {};
\addplot [ forget plot,thick]
  table[row sep=crcr]{%
9	1\\
10	1\\
};
\node[circle,fill,inner sep=1pt] at (axis cs:10,0) {};
\addplot [forget plot,thick]
  table[row sep=crcr]{%
10	0\\
11	0\\
};
\node[color=blue,circle,fill,inner sep=1pt] at (axis cs:11,-1) {};
\addplot [color=blue, forget plot,thick]
  table[row sep=crcr]{%
11	-1\\
12	-1\\
};
\node[color=blue,circle,fill,inner sep=1pt] at (axis cs:12,-1) {};
\addplot [color=blue,forget plot,thick]
  table[row sep=crcr]{%
12	-1\\
13	-1\\
};
\node[circle,fill,inner sep=1pt] at (axis cs:13,0) {};
\addplot [ forget plot,thick]
  table[row sep=crcr]{%
13	0\\
14	0\\
};
\node[circle,fill,inner sep=1pt] at (axis cs:14,-1) {};
\addplot [forget plot,thick]
  table[row sep=crcr]{%
14	-1\\
15	-1\\
};
\node[color=blue,circle,fill,inner sep=1pt] at (axis cs:15,-2) {};
\addplot [color=blue,forget plot,thick]
  table[row sep=crcr]{%
15	-2\\
16	-2\\
};
\node[circle,fill,inner sep=1pt] at (axis cs:16,-1) {};
\addplot [ forget plot,thick]
  table[row sep=crcr]{%
16	-1\\
17	-1\\
};
\node[circle,fill,inner sep=1pt] at (axis cs:17,0) {};
\addplot [ forget plot,thick]
  table[row sep=crcr]{%
17	0\\
18	0\\
};
\node[circle,fill,inner sep=1pt] at (axis cs:18,0) {};
\addplot [ forget plot,thick]
  table[row sep=crcr]{%
18	0\\
19	0\\
};
\end{axis}

\begin{axis}[%
at={(230,-500)},
width=2in,
height=0.5in,
scale only axis,
xmin=0.8,
xmax=19.1,
ymin=-2.3,
ymax=3.3,
xtick=\empty,
axis background/.style={fill=white}
]
\addplot [color=gray,dashed,forget plot]
  table[row sep=crcr]{%
0	0\\
20	0\\
};
\node[circle,fill,inner sep=1pt] at (axis cs:1,2) {};
\addplot [ forget plot,thick]
  table[row sep=crcr]{%
1	2\\
2	2\\
};
\node[circle,fill,inner sep=1pt] at (axis cs:2,3) {};
\addplot [ forget plot,thick]
  table[row sep=crcr]{%
2	3\\
3	3\\
};
\node[circle,fill,inner sep=1pt] at (axis cs:3,3) {};
\addplot [ forget plot,thick]
  table[row sep=crcr]{%
3	3\\
4	3\\
};
\node[circle,fill,inner sep=1pt] at (axis cs:4,2) {};
\addplot [ forget plot,thick]
  table[row sep=crcr]{%
4	2\\
5	2\\
};
\node[circle,fill,inner sep=1pt] at (axis cs:5,1) {};
\addplot [ forget plot,thick]
  table[row sep=crcr]{%
5	1\\
6	1\\
};
\node[circle,fill,inner sep=1pt] at (axis cs:6,1) {};
\addplot [ forget plot,thick]
  table[row sep=crcr]{%
6	1\\
7	1\\
};
\node[circle,fill,inner sep=1pt] at (axis cs:7,2) {};
\addplot [ forget plot,thick]
  table[row sep=crcr]{%
7	2\\
8	2\\
};
\node[circle,fill,inner sep=1pt] at (axis cs:8,1) {};
\addplot [ forget plot,thick]
  table[row sep=crcr]{%
8	1\\
9	1\\
};
\node[circle,fill,inner sep=1pt] at (axis cs:9,1) {};
\addplot [ forget plot,thick]
  table[row sep=crcr]{%
9	1\\
10	1\\
};
\node[circle,fill,inner sep=1pt] at (axis cs:10,0) {};
\addplot [forget plot,thick]
  table[row sep=crcr]{%
10	0\\
11	0\\
};
\node[color=blue,circle,fill,inner sep=1pt] at (axis cs:11,0) {};
\addplot [color=blue, forget plot,thick]
  table[row sep=crcr]{%
11	0\\
12	0\\
};
\node[color=blue,circle,fill,inner sep=1pt] at (axis cs:12,-1) {};
\addplot [color=blue,forget plot,thick]
  table[row sep=crcr]{%
12	-1\\
13	-1\\
};
\node[circle,fill,inner sep=1pt] at (axis cs:13,0) {};
\addplot [ forget plot,thick]
  table[row sep=crcr]{%
13	0\\
14	0\\
};
\node[circle,fill,inner sep=1pt] at (axis cs:14,-1) {};
\addplot [forget plot,thick]
  table[row sep=crcr]{%
14	-1\\
15	-1\\
};
\node[color=blue,circle,fill,inner sep=1pt] at (axis cs:15,-1) {};
\addplot [color=blue,forget plot,thick]
  table[row sep=crcr]{%
15	-1\\
16	-1\\
};
\node[circle,fill,inner sep=1pt] at (axis cs:16,0) {};
\addplot [ forget plot,thick]
  table[row sep=crcr]{%
16	0\\
17	0\\
};
\node[circle,fill,inner sep=1pt] at (axis cs:17,1) {};
\addplot [ forget plot,thick]
  table[row sep=crcr]{%
17	1\\
18	1\\
};
\node[circle,fill,inner sep=1pt] at (axis cs:18,1) {};
\addplot [ forget plot,thick]
  table[row sep=crcr]{%
18	1\\
19	1\\
};
\end{axis}
\end{tikzpicture}%
\caption{On the left is a step function $A\in\mc C_+$ (where $x^n=2$).
the segments contributing to $\mc H_0(A)$ are blue.
On the right are two (out of $2^{\mc H_0(S)}=8$)
step functions $\tilde A\in\mc C$ such that $\Ga_{\tilde A}=A$.}
\label{Figure: Skorokhod Coupling}
\end{center}
\end{figure}

\subsection{Proof of \eqref{Equation: Skorokhod Local Time at Zero}}

Since the map $Z\mapsto\sup_{s\in[0,t]}\big(-Z(s)\big)_+$ is Lipschitz with respect
to the supremum norm on $[0,t]$, if we prove that the coupling of $T$ and $S$
introduced in Section \ref{Subsection: Skorokhod Coupling} is such that
\begin{align}\label{Equation: Hoeffding on Minimum}
\left|\frac{\La^{(0,0)}_{\th}(T^{x^n})}{m_n}-2\max_{0\leq s\leq t}\frac{\big(-S^{x^n}(\th_s)\big)_+}{m_n}\right|=O\left(m_n^{-1/2}(\log m_n)^{3/4}\right)
\end{align}
almost surely as $n\to\infty$, then \eqref{Equation: Skorokhod Local Time at Zero} is proved by
a combination of \eqref{Equation: KMT Coupling} and \eqref{Equation: Skorokhod Min}.

Note that if $T^{x^n}(\th_s)=A(\th_s)$ for $s\leq t$,
where $A\in\mc C_+$ is a step function of the form \eqref{Equation: Skorokhod Step Functions},
then $\La^{(0,0)}_{\th}(T^{x^n})=\mc H_0(A)$, as defined in \eqref{Equation: Horizontal at Zero}.
By analyzing the construction of the coupling of $T$ and $S$
in Section \ref{Subsection: Skorokhod Coupling}, we see that, conditional on the event
$\{\La^{(0,0)}_{\th}(T^{x^n})=h\}$ ($h\in\mbb N_0$),
the quantity $\max_{0\leq s\leq t}\big(-S^{x^n}(\th_s)\big)_+$ is a binomial random
variable with $h$ trials and probability $1/2$. With this in mind, our strategy
is to prove \eqref{Equation: Hoeffding on Minimum} using a binomial concentration bound similar
to \eqref{Equation: Summable Deviation 2}. For this, we need a good control on the tails of $\La^{(0,0)}_{\th}(T^{x^n})$:

\begin{proposition}\label{Proposition: Tail of Horizontal at Zero}
There exists constants $C,c>0$ independent of $n$ such that for every $y\geq0$,
\[\sup_{n\in\mbb N,~x\geq0}\mbf P\left[\La_{\th}^{(0,0)}(T^{x^n})\geq m_ny\right]\leq C\mr e^{-cy^2}.\]
In particular, there exists $C>0$ large enough so that
\begin{align}\label{Equation: Horizontal at Zero Magnitude}
\sum_{n\in\mbb N}\mbf P\left[\La^{(0,0)}_{\th}(T^{x^n})\geq Cm_n\sqrt{\log m_n}\right]<\infty.
\end{align}
\end{proposition}

Indeed, with this result in hand, we obtain by Hoeffding's inequality that
\[\mbf P\left[\left|(h/2)-\max_{0\leq s\leq t}\big(-S^{x^n}(\th_s)\big)_+\right|\geq \tilde Cm_n^{1/2}(\log m_n)^{3/4}/2\bigg|\La^{(0,0)}_{\th}(T^{x^n})=h\right]
\leq2\mr e^{-\tilde C^2\log m_n/2C}\]
uniformly in $0\leq h\leq Cm_n\sqrt{\log m_n}$. By taking $\tilde C>0$ large enough,
we conclude that \eqref{Equation: Skorokhod Local Time at Zero} holds by an application of the Borel-Cantelli lemma
combined with \eqref{Equation: Horizontal at Zero Magnitude}.

\begin{proof}[Proof of Proposition \ref{Proposition: Tail of Horizontal at Zero}]
Let $T$ and $S$ be coupled as in Section \ref{Subsection: Skorokhod Coupling}, and let
\[\mu_\th(S):=\sum_{u=0}^{\th-1}\mbf1_{\big\{S(u+1)\leq\min\{S(0),S(1),\ldots,S(u)\}\big\}},\]
that is, the number of times that $S$ is smaller or equal to its running minimum
over the first $\th$ steps. Then, we see that
\[\La_\th^{(0,0)}(T)=\sum_{u=0}^{\th-1}\mbf1_{\big\{S(u)\leq 0,~S(u+1)\leq\min\{S(0),S(1),\ldots,S(u)\}\big\}}
\leq\mu_\th(S).\]
Given that $\mu_\th(S)$ is independent of $S$'s starting point, it suffices to prove that
\begin{align}\label{Equation: Ladder Epochs Tails}
\sup_{n\in\mbb N}\mbf P\left[\mu_\th(S^{0})\geq m_ny\right]\leq C\mr e^{-cy^2},\qquad y\geq0
\end{align}
for some constants $C,c>0$.

If $y>m_nt$, then $m_ny\geq\th$, hence $\mbf P\left[\mu_\th(S^{0})\geq m_ny\right]=0$.
Thus, it suffices to prove \eqref{Equation: Ladder Epochs Tails} for $y\leq m_nt$.
Our proof of this is inspired by \cite[Lemma 7]{Ladder}:
Let $0=\mf t_0<\mf t_1<\mf t_2<\cdots$ be the weak descending ladder epochs of $S^0$,
that is,
\[\mf t_{u+1}:=\min\{v>\mf t_u:S^0(v)\leq S^0(\mf t_u)\},\qquad u\in\mbb N_0.\]\
Then, for any $\nu>0$,
\begin{align*}
\mbf P\left[\mu_\th(S^{0})\geq m_ny\right]&=\mbf P\left[\mf t_{\lceil m_n y\rceil}\leq\th\right]
\leq\mbf P\left[S^0(\mf t_{\lceil m_n y\rceil})\geq\min_{0\leq u\leq\th}S^0(u)\right]\\
&\leq\mbf P\left[S^0(\mf t_{\lceil m_n y\rceil})\geq -\nu m_ny\right]+\mbf P\left[\min_{0\leq u\leq\th} S^0(u)<-\nu m_n y\right].
\end{align*}

On the one hand, we note that $S^0(\mf t_{\lceil m_n y\rceil})$ is equal in distribution to
the sum of $\lceil m_n y\rceil$ i.i.d. copies of $S^0(\mf t_1)$,
which we call the the ladder height of $S^0$.
Moreover, it is easily seen that the ladder height has distribution
$\mbf P[S^0(\mf t_1)=0]=2/3$ and $\mbf P[S^0(\mf t_1)=-1]=1/3$. In particular,
$\mbf E[S^0(\mf t_{\lceil m_n y\rceil})]=-\lceil m_n y\rceil/3$. Thus, if we choose $\nu$
small enough (namely $\nu<1/3$), then by combining Hoeffding's
inequality with $m_n\geq y/t$, we obtain
\[\mbf P\left[S^0(\mf t_{\lceil m_n y\rceil})\geq -\nu m_ny\right]
=\mbf P\left[S^0(\mf t_{\lceil m_n y\rceil})+\frac{\lceil m_n y\rceil}3\geq -\nu m_ny+\frac{\lceil m_n y\rceil}3\right]
\leq C_1\mr e^{-c_1m_ny}\leq C_1\mr e^{-c_1y^2/t}\]
for some $C_1,c_1>0$ independent of $n$.

On the other hand, by Etemadi's
and Hoeffding's inequalities,
\[\mbf P\left[\min_{0\leq u\leq\th} S^0(u)<-\nu m_n y\right]
\leq\mbf P\left[\max_{0\leq u\leq\th}|S^0(u)|>\nu m_n y\right]
\leq 3\max_{0\leq u\leq\th}\mbf P\left[|S^0(u)|>\nu m_ny/3\right]
\leq C_2\mr e^{-c_2y^2}\]
for some $C_2,c_2>0$ independent of $n$,
concluding the proof of \eqref{Equation: Ladder Epochs Tails} for $y\leq m_nt$.
\end{proof}

\subsection{Proof of Proposition \ref{Proposition: Reflected Splitting of Occupation Measure}}

By replicating the binomial concentration argument in the proof of Proposition \ref{Proposition: Splitting of Occupation Measure},
it suffices to prove that
\begin{align}\label{Equation: Range Deviation for T}
\sum_{n\in\mbb N}\mbf P\left[\mc R_\th(T^{x^n})\geq m_n^{1+\eps}\right]
<\infty
\end{align}
for every $\eps>0$, where we define $\mc R_\th(T^{x^n})$ as in \eqref{Equation: Random Walk Range}, and
\begin{align}\label{Equation: Occupation Deviation for T}
\sum_{n\in\mbb N}\mbf P\left[\sup_{a\in\mbb Z}\La^a_\th(T^{x^n})\geq Cm_n\log m_n\right]<\infty
\end{align}
provided $C>0$ is large enough. In order to prove this, we introduce another coupling
of $S$ and $T$, which will also be useful later in the paper:

\begin{definition}\label{Definition: Time Change Coupling}
Let $a\in\mbb N_0$ be fixed. Given a realization of $T^a$,
let us define the time change $\big(\tilde\rho^a(u)\big)_{u\in\mbb N_0}$
as follows:
\begin{enumerate}
\item $\tilde\rho^a(0)=0$.
\item If $T^a\big(\tilde\rho^a(u)\big)\neq0$ or $T^a\big(\tilde\rho^a(u)+1\big)\neq0$,
then $\tilde\rho^a(u+1)=\tilde\rho^a(u)+1$.
\item If $T^a\big(\tilde\rho^a(u)\big)=0$ and $T^a\big(\tilde\rho^a(u)+1\big)=0$
then we sample
\[\mbf P[\tilde\rho^a(u+1)=\tilde\rho^a(u)+1]=\frac14
\qquad\text{and}\qquad
\mbf P[\tilde\rho^a(u+1)=\tilde\rho^a(u)+2]=\frac34,\]
independently of the increments in $T^a$.
\end{enumerate}
In words, we go through the path of $T^a$ and skip every visit
to the self-edge $(0,0)$ independently with probability $3/4$.
Then, we define $\rho^a$ as the inverse of $\tilde\rho^a$,
which is well defined since the latter is strictly increasing.
\end{definition}

By a straightforward geometric sum calculation, it is easy to see
that we can couple $S$ and $T$ in such a way that
\begin{align}\label{Equation: Step Removal Coupling}
T^{x^n}(u)=\big|S^{x^n}\big(\rho^{x^n}(u)\big)\big|,\qquad u\in\mbb N_0.
\end{align}
For the remainder of the proof of Proposition
\ref{Proposition: Reflected Splitting of Occupation Measure} we adopt this coupling.

On the one hand,
$\mc R_\th(T^{x^n})=\mc R_{\rho^{x^n}_\th}(|S^{x^n}|)\leq\mc R_{\th}(S^{x^n}).$
Thus \eqref{Equation: Range Deviation for T} follows directly from
Lemma \ref{Lemma: Range Deviation}. On the other hand,
for every $a\neq0$,
\begin{align}
\La^a_u(T^{x^n})=\La^a_{\rho^{x^n}_\th}(S^{x^n})+\La^{-a}_{\rho^{x^n}_\th}(S^{x^n})
\leq\La^a_{\th}(S^{x^n})+\La^{-a}_{\th}(S^{x^n}),
\end{align}
and
\[\La^0_u(T^{x^n})
=\La^{(0,0)}_{u}(T^{x^n})+\La^{(0,-1)}_{\rho^{x^n}_u}(S^{x^n})
+\La^{(0,1)}_{\rho^{x^n}_u}(S^{x^n})
\leq\La^{(0,0)}_{u}(T^{x^n})+\La^{0}_{u}(S^{x^n}).\]
Thus \eqref{Equation: Occupation Deviation for T}
follows from \eqref{Equation: Horizontal at Zero Magnitude} and Lemma \ref{Lemma: Local Time Deviation}.

\subsection{Proof of Proposition \ref{Proposition: Reflected Strong Local Time Coupling}}

The following extends Lemma \ref{Lemma: Local Time Invariance Criterion} to $T$.

\begin{lemma}\label{Lemma: Reflected Local Time Invariance Criterion}
For any $0<\de<1$, the following holds almost surely as $n\to\infty$:
\begin{multline*}
\sup_{0\leq u\leq \th,~a\in\mbb N_0}\left|\frac{\La_u^a(T^{x^n})}{m_n}-L_{u/m_n^2}^{a/m_n}(\tilde X^x)\right|
=O\Bigg(\sup_{\substack{0\leq s\leq t\\y,z\geq0,~|y-z|\leq m_n^{-\de}}}\big|L^y_s(\tilde X^x)-L^z_s(\tilde X^x)\big|\\
+m_n^{2\de}\sup_{0\leq s\leq t}\left|\frac{T^{x^n}(\th_s)}{m_n}-\tilde X^x(s)\right|
+\sup_{\substack{0\leq u\leq\th\\a,b\in\mbb N_0,~|a-b|\leq m_n^{1-\de}}}\frac{\big|\La^a_u(T^{x^n})-\La^b_u(T^{x^n})\big|}{m_n}
+\sup_{0\leq u\leq\th,~a\in\mbb N_0}\frac{\La^a_u(T^{x^n})}{m_n^{2-\de}}\Bigg).
\end{multline*}
\end{lemma}
\begin{proof}
Let $a\in\mbb N$ be fixed. For every $\eps>0$, let $f_\eps:\mbb R\to\mbb R$ be defined as in the proof of Lemma
\ref{Lemma: Local Time Invariance Criterion}, and let us define $g_\eps:\mbb R_+\to\mbb R$ as
\[g_\eps(z):=f_\eps(z)\left(\int_0^{\infty}f_\eps(z)\d z\right)^{-1},\qquad z\geq0.\]
$g_\eps$ integrates to one on $\mbb R_+$, and $|g_\eps(z)-g_\eps(y)|/|z-y|\leq\frac2{\eps^2}$
for $y,z\geq0$. By repeating the proof
of Lemma \ref{Lemma: Local Time Invariance Criterion} verbatim with $g_\eps$ instead of $f_\eps$, we obtain the result.
\end{proof}

We now apply Lemma \ref{Lemma: Reflected Local Time Invariance Criterion}. \eqref{Equation: Skorokhod Coupling} yields
\[m_n^{2\de}\sup_{0\leq s\leq t}\left|\frac{T^{x^n}(\th_s)}{m_n}-\tilde X^x(s)\right|=O\left(m_n^{1-2\de}\log m_n\right).\]
As for the regularity of the vertex-occupation measures and local time of $T^{x^n}$ and $\tilde X^x$,
they follow directly from the proof of Proposition \ref{Proposition: Strong Local Time Coupling}
using Lemma \ref{Lemma: Local Time Invariance Criterion}
by applying some carefully chosen couplings of $T^{x_n}$ with $S^{x_n}$,
and $\tilde X^x$ with $\tilde B^x$:

We begin with the latter. If we define $\tilde X^x(s)=|\tilde B^x(s)|$, then for every $y\geq0$ and $s\geq0$, we have that
$L_s^y(\tilde X^x)=L_s^y(\tilde B^x)+L_s^{-y}(\tilde B^x)$. Consequently,
\begin{align}\label{Equation: Reflected Local Time Regularity 1}
\big|L^y_s(\tilde X^x)-L^z_{\bar s}(\tilde X^x)\big|\leq\big|L^y_s(\tilde B^x)-L^z_{\bar s}(\tilde B^x)\big|
+\big|L^{-y}_s(\tilde B^x)-L^{-z}_{\bar s}(\tilde B^x)\big|.
\end{align}
The regularity estimates for $L_s^y(\tilde X^x)$ then follow from the
same results for $L_s^y(\tilde B^x)$.

To prove the desired estimates on the occupation measures, we use the coupling
introduced in Definition \ref{Definition: Time Change Coupling}. This immediately yields an adequate
control of the supremum of $\La^a_\th(T^{x^n})$ by \eqref{Equation: Occupation Deviation for T}.
As for regularity, one the one hand, we note that
\[\big|\La^a_u(T^{x^n})-\La^b_u(T^{x^n})\big|
\leq\big|\La^a_{\rho^{x^n}_u}(S^{x^n})-\La^b_{\rho^{x^n}_u}(S^{x^n})\big|
+\big|\La^{-a}_{\rho^{x^n}_u}(S^{x^n})-\La^{-b}_{\rho^{x^n}_u}(S^{x^n})\big|
\]
for any $a,b\neq0$. On the other hand,
for any $a\neq0$, 
\[\big|\La^0_u(T^{x^n})-\La^a_u(T^{x^n})\big|
\leq\big|\tfrac12\La^0_u(T^{x^n})-\La^a_{\rho^{x^n}_u}(S^{x^n})\big|
+\big|\tfrac12\La^0_u(T^{x^n})-\La^{-a}_{\rho^{x^n}_u}(S^{x^n})\big|.\]
Hence we get the desired estimate by Lemma \ref{Lemma: Occupation Measure Regularity}
if we prove that
\begin{align}\label{Equation: Proportion of Truncated Occupation at Zero}
\sup_{0\leq u\leq\th}\big|\tfrac12\La^0_u(T^{x^n})-\La^0_{\rho^{x^n}_\th}(S^{x^n})\big|=O\big(m_n^{-1/2}\log m_n\big)
\end{align}
almost surely as $n\to\infty$.
By Propositions \ref{Proposition: Splitting of Occupation Measure} and
\ref{Proposition: Reflected Splitting of Occupation Measure},
\eqref{Equation: Proportion of Truncated Occupation at Zero} can be reduced to
\begin{align}\label{Equation: Proportion of Truncated Occupation at Zero 2}
\sup_{0\leq u\leq\th}3\big|\tfrac14\La^{(0,0)}_u(T^{x^n})-\La^{(0,0)}_{\rho^{x^n}_\th}(S^{x^n})\big|=O\big(m_n^{-1/2}\log m_n\big).
\end{align}
By Definition \ref{Definition: Time Change Coupling}, conditional on $\La^{(0,0)}_u(T^{x^n})$,
we note that $\La^{(0,0)}_{\rho^{x^n}(u)}(S^{x^n})$ is a binomial random variable
with $\La^{(0,0)}_u(T^{x^n})$ trials and probability $1/4$.
Hence we obtain \eqref{Equation: Proportion of Truncated Occupation at Zero 2} by
combining \eqref{Equation: Horizontal at Zero Magnitude} with Hoeffding's
inequality similarly to \eqref{Equation: Summable Deviation 2}.

\subsection{Coupling of $T_\th^{a,b}$}
\label{Subsection: Bridge Coupling Conjecture}

In light of Theorems \ref{Theorem: Split Coupling} and \ref{Theorem: Split Skorokhod Couplings},
the following conjecture is natural.

\begin{conjecture}
\label{Conjecture: Trace Convergence 2}
The statement of Theorem \ref{Theorem: Split Skorokhod Couplings}
holds with every instance of $T^{x^n}$ replaced by $T^{x^n,x^n}_{\th_t}$, and
every instance of $\tilde X^{x^n}$ replaced by $\tilde X^{x^n,x^n}_t$.
\end{conjecture}

However, if we couple $T$ in $S$ as in Section \ref{Subsection: Skorokhod Coupling},
then conditioning on the endpoint of $T$
corresponds to an unwieldy conditioning of the path of $S$:
\begin{align*}
\mbf P[T^a(\th)=a]=\mbf P\left[S^a(\th)=\max_{0\leq u\leq\th}\big(-S^a(u)\big)_++a\right].
\end{align*}
There seems to be no existing strong invariance result (such as KMT) applicable to this conditioning.
Consequently, it appears that a proof of Conjecture
\ref{Conjecture: Trace Convergence 2} relies on a strong invariance result for conditioned
random walks that is outside the scope of the current literature, or that it requires an
altogether different reduction to a classical coupling (which we were not able to find).

\section{Proof of Theorem \ref{Theorem: Main Dirichlet}-(1)}
\label{Section: Main Dirichlet 1}

For the remainder of this section, we fix some times
$t_1,\ldots,t_k>0$ and uniformly continuous and bounded functions $f_1,g_1,\ldots,f_k,g_k$.

\subsection{Step 1: Convergence of Mixed Moments}\label{Subsection: Moments}

Consider a mixed moment
\[\mbf E\left[\prod_{i=1}^k\langle f_i, \hat K_n(t_i)g_i\rangle^{n_i}\right],\qquad n_1,\ldots,n_k\in\mbb N_0.\]
Up to making some $f_i$'s, $g_i$'s, and $t_i$'s equal
to each other and reindexing, there is no loss of generality in writing the above in the form
\begin{align}\label{Equation: Mixed Moment Formula 1}
\mbf E\left[\prod_{i=1}^k\langle f_i, \hat K_n(t_i)g_i\rangle\right].
\end{align}
By applying Fubini's theorem to \eqref{Equation: Matrix-Form Formula},
we can write \eqref{Equation: Mixed Moment Formula 1} as
\begin{align}\label{Equation: Mixed Moment Prelimit}
\int_{[0,(n+1)/m_n)^k}\left(\prod_{i=1}^kf_i(x_i)\right)
\mbf E\left[\prod_{i=1}^kF_{n,t_i}(S^{i;x_i^n})
\,m_n\int_{S^{i;x_i^n}(\th_i)/m_n}^{(S^{i;x_i^n}(\th_i)+1)/m_n}g_i(y)\d y\right]\d x_1\cdots\dd x_k
\end{align}
and the corresponding limiting expression as
\begin{align}\label{Equation: Mixed Moment Limit}
\mbf E\left[\prod_{i=1}^k\langle f_i, \hat K(t_i)g_i\rangle\right]
=\int_{\mbb R_+^k}\left(\prod_{i=1}^kf_i(x_i)\right)\,
\mbf E\left[\prod_{i=1}^k\mbf 1_{\{\tau_0(B^{i;x_i})>t\}}\mr e^{-\langle L_t(B^{i;x_i}),Q'\rangle} g_i\big(B^{i;x_i}(t)\big)\right]\d x_1\cdots\dd x_k,
\end{align}
where
\begin{enumerate}
\item $\th_i=\th_i(n,t_i):=\lfloor m_n^2 (3t_i/2)\rfloor$ for every $n\in\mbb N$ and $1\leq i\leq k$;
\item $x_i^n:=\lfloor m_n x_i\rfloor$ for every $n\in\mbb N$ and $1\leq i\leq k$;
\item $S^{1;x_1^n},\ldots,S^{k;x_k^n}$ are independent copies of $S$ with respective
starting points $x_1^n,\ldots,x_k^n$; and
\item $B^{1;x_1},\ldots,B^{k;x_k}$ are independent copies of $B$
with respective starting points $x_1,\ldots,x_k$.
\end{enumerate}
We further assume that the $S^{i;x_i^n}$ are independent of $Q_n$,
and that the $B^{i;x_i}$ are independent of $Q$. The proof of moment convergence is based on the following:

\begin{proposition}\label{Proposition: Moment Convergence}
Let $x_1,\ldots,x_n\geq0$ be fixed.
There is a coupling of the $S^{i;x_i^n}$ and $B^{i;x_i}$ such that
the following limits hold jointly in distribution over $1\leq i\leq k$:
\begin{enumerate}
\item $\displaystyle \lim_{n\to\infty}\sup_{0\leq s\leq t_i}\left|\frac{S^{i;x_i^n}(\lfloor m_n^2(3s/2)\rfloor)}{m_n}-B^{i;x_i}(s)\right|=0.$
\item $\displaystyle \lim_{n\to\infty}\sup_{y\in\mbb R}\left|\frac{\La^{(y_n,\bar y_n)}_{\th_i}(S^{i;x_i^n})}{m_n}-\frac12L^y_{t_i}(B^{i;x_i})\right|=0$,
\\jointly in $(y_n,\bar y_n)_{n\in\mbb N}$ equal to the three sequences in \eqref{Equation: yn Sequences}.
\item $\displaystyle \lim_{n\to\infty}m_n\int_{S^{i;x_i^n}(\th_i)/m_n}^{(S^{i;x_i^n}(\th_i)+1)/m_n}g_i(y)\d y=g_i\big(B^{i;x_i}(t)\big).$
\item The convergences in \eqref{Equation: Skorokhod Noise Terms}.
\item $\displaystyle\lim_{n\to\infty}\sum_{a\in\mbb N_0}\frac{\La^{(a_E,\bar a_E)}_{\th_i}(S^{i;x_i^n})}{m_n}\frac{\xi^E_n(a)}{m_n}
=\frac12\int_{\mbb R_+}L^y_{t_i}(B^{i;x_i})\d W^E(y)$\\
jointly in $E\in\{D,U,L\}$, where for every $a\in\mbb N_0$,
\begin{align}\label{Equation: Product (a,a) Notation}
(a_E,\bar a_E):=\begin{cases}
(a,a)&\text{if }E=D,\\
(a,a+1)&\text{if }E=U,\\
(a+1,a)&\text{if }E=L.
\end{cases}
\end{align}
\end{enumerate}
\end{proposition}
\begin{proof}
According to Theorem \ref{Theorem: Split Coupling} in the case
of the lazy random walk, we can couple $S^{i;x^n_i}$ with
a Brownian motion with variance $2/3$ started at $x_i$, $\tilde B^{i;x_i}$, in such a way that
\[\frac{S^{i;x_i^n}(\lfloor m_n^2(3s/2)\rfloor)}{m_n}\to\tilde B^{i;x_i}(3s/2)
\quad\text{and}\quad
\frac{\La^{(y_n,\bar y_n)}_{\th_i}(S^{i;x_i^n})}{m_n}\to\frac13L^y_{3t_i/2}(\tilde B^{i;x_i})\]
uniformly almost surely. Let $B^{i;x_i}(s):=\tilde B^{i;x_i}(3s/2)$. By the Brownian scaling property,
 $B^{i;x_i}$ is standard,
and $L^y_{3t_i/2}(\tilde B^{i;x_i})=\tfrac32L^y_{t_i}(B^{i;x_i})$.
Hence (1) and (2) hold almost surely.
Since $g_i$ is uniformly continuous, (3) holds almost surely by (1) and the Lebesgue differentiation theorem.
With this given, (4) and (5) follow from Assumption \ref{Assumption: Noise Convergence}.
\end{proof}

\begin{remark}\label{Remark Slightly Different Powers 2}
Since the strong invariance principles in Theorem \ref{Theorem: Split Coupling}
are uniform in the time parameter, it is clear that Proposition \ref{Proposition: Moment Convergence}
remains valid of we take $\th_i:=\lfloor m_n^2 (3t_i/2)\rfloor\pm1$ instead of $\lfloor m_n^2 (3t_i/2)\rfloor$.
Referring back
to Remark \ref{Remark: Slightly Different Powers}, there is no loss of generality in assuming that
the $\th_i$ have a particular parity. The same comment applies to our proof of Theorem
\ref{Theorem: Main Dirichlet}-(2) and Theorem \ref{Theorem: Main Robin}.
\end{remark}

\subsubsection{Convergence Inside the Expected Value}
\label{Subsubsection: Convergence Inside the Expected Value}

We first prove that for every fixed $x_1,\ldots,x_k\geq0$,
there exists a coupling such that
\begin{align}\label{Equation: Convergence of Functional}
\lim_{n\to\infty}\prod_{i=1}^kF_{n,t_i}(S^{i;x_i^n})
\, m_n\int_{S^{i;x_i^n}(\th_i)/m_n}^{(S^{i;x_i^n}(\th_i)+1)/m_n}g_i(y)\d y
=\prod_{i=1}^k\mbf 1_{\{\tau_0(B^{i;x_i})>t\}}\mr e^{-\langle L_t(B^{i;x_i}),Q'\rangle} g_i\big(B^{i;x_i}(t)\big)
\end{align}
in probability.
According to the Skorokhod representation theorem (e.g., \cite[Theorem 6.7]{Billingsley99}),
there is a coupling such that Proposition \ref{Proposition: Moment Convergence} holds almost surely.
For the remainder of Section \ref{Subsubsection: Convergence Inside the Expected Value},
we adopt such a coupling.

Since $m_n^{-1}S^{i;x_i^n}(\lfloor m_n^2(3s/2)\rfloor)\to B^{i,x_i}(s)$ uniformly on $s\in[0,t_i]$, and $m_n^2=o(n)$,
\[\lim_{n\to\infty}\mbf 1_{\{\tau^{(n)}(S^{i;x_i^n})>\th_i\}}=\mbf 1_{\{\tau_0(B^{i;x_i})>t_i\}}\]
almost surely.
By combining this with Proposition \ref{Proposition: Moment Convergence}-(3),
it only remains to prove that the terms involving the matrix entries $D_n$, $U_n$, and $L_n$ in the
functional $F_{n,t_i}$ converge to $\mr e^{-\langle L_t(B^{i;x_i}),Q'\rangle}$.
To this effect, we note that for $E\in\{D,U,L\}$,
\[\prod_{a\in\mbb N_0}\left(1-\frac{E_n(a)}{m_n^2}\right)^{\La^{(a_E,\bar a_E)}_{\th_i}(S^{i;x_i^n})}
=\exp\left(\sum_{a\in\mbb N_0}\La^{(a_E,\bar a_E)}_{\th_i}(S^{i;x_i^n})\log\left(1-\frac{E_n(a)}{m_n^2}\right)\right),\]
where we recall that $(a_E,\bar a_E)$ are defined as in \eqref{Equation: Product (a,a) Notation}.
By using the Taylor expansion $\log(1+z)=z+O(z^2)$, this is equal to
\begin{align}\label{Equation: E Terms}
\exp\left(-\sum_{a\in\mbb N_0}\La^{(a_E,\bar a_E)}_{\th_i}(S^{i;x_i^n})\frac{E_n(a)}{m_n^2}+O\left(\sum_{a\in\mbb N_0}
\La^{(a_E,\bar a_E)}_{\th_i}(S^{i;x_i^n})\frac{E_n(a)^2}{m_n^4}\right)\right).
\end{align}

We begin by analyzing the leading order term in \eqref{Equation: E Terms}.
On the one hand, the uniform convergence of Proposition \ref{Proposition: Moment Convergence}-(2)
(which implies in particular that $y\mapsto\La^{(y_n,\bar y_n)}_{\th_i}(S^{i;x_i^n})/m_n$
and $y\mapsto L^y(B^{i;x_i})$ are supported on a common compact interval almost surely)
together with the fact that $V^E_n(\lfloor m_ny\rfloor)\to V^E(y)$ uniformly on compacts
(by Assumption \ref{Assumption: Potential Convergence}) implies that
\begin{align}\label{Equation: Potential Leading Terms}
\lim_{n\to\infty}\sum_{a\in\mbb N_0}\La^{(a_E,\bar a_E)}_{\th_i}(S^{i;x_i^n})\frac{V^E_n(a)}{m_n^2}
=\lim_{n\to\infty}\int_0^\infty\frac{\La^{(y_n,\bar y_n)}_{\th_i}(S^{i;x_i^n})}{m_n}V^E_n(\lfloor m_ny\rfloor)\d y
=\tfrac12\langle L_{t_i}(B^{x_i}),V^E\rangle
\end{align}
almost surely (where we choose the appropriate sequence $(y_n,\bar y_n)$ as defined
in \eqref{Equation: yn Sequences} depending on $(a_E,\bar a_E)$).
By combining this with Proposition \ref{Proposition: Moment Convergence}-(5),
we get
\begin{align}\label{Equation: Noise Leading Terms}
\lim_{n\to\infty}\sum_{a\in\mbb N_0}\La^{(a_E,\bar a_E)}_{\th_i}(S^{i;x_i^n})\frac{E_n(a)}{m_n^2}=
\tfrac12\langle L_{t_i}(B^{x_i}),(Q^E)'\rangle
\end{align}
almost surely, where $\dd Q^E(y):=V^E(y)\dd y+\dd W^E(y)$.

Next, we control the error term in \eqref{Equation: E Terms}. By using
$(z+\bar z)^2\leq2(z^2+\bar z^2)$, for this it suffices control
\[\sum_{a\in\mbb N_0}\La^{(a_E,\bar a_E)}_{\th_i}(S^{i;x_i^n})\frac{V^E_n(a)^2}{m_n^4}
\qquad\text{and}\qquad
\sum_{a\in\mbb N_0}\La^{(a_E,\bar a_E)}_{\th_i}(S^{i;x_i^n})\frac{\xi^E_n(a)^2}{m_n^4}\]
separately. On the one hand, the argument used in \eqref{Equation: Potential Leading Terms} yields
\[\sum_{a\in\mbb N_0}\La^{(a_E,\bar a_E)}_{\th_i}(S^{i;x_i^n})\frac{V^E_n(a)^2}{m_n^4}
=m_n^{-2}\,\tfrac12\big(1+o(1)\big)\langle L_{t_i}(B^{x_i}),(V^E)^2\rangle.\]
Since $V^E$ is continuous and $L_{t_i}(B^{x_i})$ is compactly supported with probability one,
this converges to zero almost surely. On the other hand,
by definition of \eqref{Equation: Edge-Occupation Measure},
\[\sum_{a\in\mbb N_0}\La^{(a,b)}_{\th_i}(S^{i;x_i^n})\leq\th_i=O(m_n^2)\]
uniformly in $b\in\mbb Z$. Therefore, it follows from the tower property and \eqref{Equation: Noise Moment Bound} that
\[\mbf E\left[\sum_{a\in\mbb N_0}\La^{(a_E,\bar a_E)}_{\th_i}(S^{i;x_i^n})\frac{\xi^E_n(a)^2}{m_n^4}\right]
=\mbf E\left[\sum_{a\in\mbb N_0}\La^{(a_E,\bar a_E)}_{\th_i}(S^{i;x_i^n})\frac{\mbf E[\xi^E_n(a)^2]}{m_n^4}\right]
=O(m_n^{-1});\]
hence we have convergence to zero in probability.

By combining the convergence of the leading terms \eqref{Equation: Noise Leading Terms},
our analysis of the error terms, and \eqref{Equation: Combination of Potentials} and \eqref{Equation: Combination of BM},
we conclude that \eqref{Equation: Convergence of Functional} holds.

\subsubsection{Convergence of the Expected Value}\label{Subsubsection: Convergence of the Expected Value}

Next, we prove that
\begin{multline}\label{Equation: Convergence of Expected Values}
\lim_{n\to\infty}\mbf E\left[\prod_{i=1}^kF_{n,t_i}(S^{i;x_i^n})
\, m_n\int_{S^{i;x_i^n}(\th_i)/m_n}^{(S^{i;x_i^n}(\th_i)+1)/m_n}g_i(y)\d y\right]\\
=\mbf E\left[\prod_{i=1}^k\mbf 1_{\{\tau_0(B^{i;x_i})>t\}}\mr e^{-\langle L_t(B^{i;x_i}),Q'\rangle} g_i\big(B^{i;x_i}(t)\big)\right]
\end{multline}
pointwise in $x_1,\ldots,x_k\geq0$. Given \eqref{Equation: Convergence of Functional}, we must
prove that the sequence of variables inside the expected value on the left-hand side of
\eqref{Equation: Convergence of Expected Values} are uniformly integrable.
For this, we prove that
\begin{multline*}
\sup_{n\geq N}\mbf E\left[\prod_{i=1}^k\left(F_{n,t_i}(S^{i;x_i^n})
\, m_n\int_{S^{i;x_i^n}(\th_i)/m_n}^{(S^{i;x_i^n}(\th_i)+1)/m_n}g_i(y)\d y\right)^2\right]\\
\leq\sup_{n\geq N}\prod_{i=1}^k\mbf E\left[\left(F_{n,t_i}(S^{i;x_i^n})
\, m_n\int_{S^{i;x_i^n}(\th_i)/m_n}^{(S^{i;x_i^n}(\th_i)+1)/m_n}g_i(y)\d y\right)^{2k}\right]^{1/k}<\infty
\end{multline*}
for large enough $N$, where the first upper bound is due to H\"older's inequality.

Since the $g_i$'s are bounded,
\[m_n\int_{S^{i;x_i^n}(\th_i)/m_n}^{(S^{i;x_i^n}(\th_i)+1)/m_n}g_i(y)\d y\leq\|g_i\|_\infty<\infty,\]
uniformly in $n$, and thus we need only prove that
\begin{align}\label{Equation: Uniformly Integrability 1}
\sup_{n\geq N}\mbf E\left[\big|F_{n,t_i}(S^{i;x_i^n})\big|^{2k}\right]<\infty,\qquad 1\leq i\leq k.
\end{align}
Since indicator functions are bounded by 1, their contribution to $\eqref{Equation: Uniformly Integrability 1}$ may
be ignored. For the other terms, we note that for $E\in\{D,U,L\}$ we can write
\begin{align}\label{Equation: Uniform Integrability Decomposition}
1-\frac{E_n(a)}{m_n^2}=\frac{m_n^2-V^E_n(a)-\xi^E_n(a)}{m_n^2}
=\left(1-\frac{V_n^E(a)}{m_n^2}\right)\left(1-\frac{\xi^E_n(a)}{m_n^2-V^E_n(a)}\right).
\end{align}
By \eqref{Equation: Potential Absolute Bounds}, for large $n$ we have $|1-V_n^E(a)/m_n^2|\leq 1$,
hence by applying H\"older's inequality in \eqref{Equation: Uniformly Integrability 1}, we need only prove that
\begin{align}\label{Equation: Uniformly Integrability 2}
\sup_{n\geq N}\mbf E\left[\prod_{a\in\mbb N_0}\left|1-\frac{\xi^E_n(a)}{m_n^2-V^E_n(a)}\right|^{6k\La^{(a_E,\bar a_E)}_{\th_i}(S^{i;x_i^n})}\right]<\infty,\quad
E\in\{D,U,L\}.
\end{align}

Let us fix $E\in\{D,U,L\}$ and define
\[\ze_n(a):=\frac{\xi^E_n(a)}{m_n^{1/2}}
\qquad\text{and}\qquad
r_n(a):=\frac{m_n^{1/2}}{m_n^2-V^E_n(a)}.\]
By \eqref{Equation: Noise Moment Bound}, we know that there exists $C>0$ and $0<\ga<2/3$
such that $\mbf E[|\ze_n(a)|^q]\leq C^qq^{\ga q}$ for every $q\in\mbb N$ and $n$ large enough.
Thus, since the variables $\xi^E_n(0),\ldots,\xi^E_n(n)$ are independent,
it follows from the upper bound \cite[(4.25)]{GorinShkolnikov} that there exists $C'>0$ and $2<\ga'<3$
both independent of $n$ such that \eqref{Equation: Uniformly Integrability 2} is bounded above by
\begin{multline}\label{Equation: Noise Exponential Bound 1}
\mbf E\Bigg[\exp\Bigg(C'\bigg(
\sum_{a\in\mbb N_0}|r_n(a)|\,\La^{a}_{\th_i}(S^{i;x_i^n})\,\big|\mbf E[\ze_n(a)]\big|\\
+\sum_{a\in\mbb N_0}r_n(a)^2\,\La^{a}_{\th_i}(S^{i;x_i^n})^2
+\sum_{a\in\mbb N_0}|r_n(a)|^{\ga'}\,\La^{a}_{\th_i}(S^{i;x_i^n})^{\ga'}
\bigg)\Bigg)\Bigg],
\end{multline}
where
we use the trivial bound $\La^{(a,b)}_{\th},\La^{(b,a)}_{\th}\leq\La^{a}_{\th}$ for all $a,b$.

For any fixed $x_i$, we know that $S^{i;x^n_i}(u)=O(m_n^2)$ uniformly in $0\leq u\leq \th_i$
because $\th_i=O(m_n^2)$. Thus, the only values of $a$ for which $\La^{a}_{\th_i}$ is possibly nonzero
are at most of order $O(m_n^2)=o(n)$. For any such values of $a$, the assumption \eqref{Equation: Potential Growth Upper Bound}
implies that $V_n^E(a)=o(m_n^2)$, hence $r_n(a)=O(m_n^{-3/2})$. By combining all of these estimates with \eqref{Equation: Noise Expected Decay},
\eqref{Equation: Uniformly Integrability 2} is then a consequence of the following proposition.

\begin{proposition}\label{Proposition: Local Time Uniform Exponential Moments}
Let $\th=\th(n,t):=\lfloor m_n^2t\rfloor$ and $x^n:=\lfloor m_nx\rfloor$ for some $t>0$ and $x\geq0$.
For every $C>0$ and $1\leq q<3$,
\[\sup_{n\in\mbb N,~x\geq0}
\mbf E\left[\exp\left(\frac{C}{m_n}\sum_{a\in\mbb Z}\frac{\La^{a}_{\th}(S^{x^n})^q}{m_n^q}\right)\right]<\infty.\]
\end{proposition}

Since the proof of Proposition \ref{Proposition: Local Time Uniform Exponential Moments}
is rather long and technical, we provide it later in Section
\ref{Subsection: Local Time Uniform Exponential Moments} so as to not interrupt the
flow of the present argument.

\subsubsection{Convergence of the Integral}\label{Subsubsection: Convergence of Integral}

We now complete the proof that \eqref{Equation: Mixed Moment Prelimit} converges to
\eqref{Equation: Mixed Moment Limit}. With \eqref{Equation: Convergence of Expected Values}
established, it only remains to justify passing the limit inside the integral in $\dd x_1\cdots\dd x_k$.
In order to prove this, we aim to use the Vitali convergence theorem (e.g., \cite[Theorem 2.24]{FonsecaLeoni}).
For this, we need a more refined version
of the uniform integrability estimate used in Section \ref{Subsubsection: Convergence of the Expected Value}.
By H\"older's inequality,
\begin{multline}\label{Equation: Vitali 0}
\left(\prod_{i=1}^kf_i(x_i)\right)
\,\mbf E\left[\prod_{i=1}^kF_{n,t_i}(S^{i;x_i^n})
\,m_n\int_{S^{i;x_i^n}(\th_i)/m_n}^{(S^{i;x_i^n}(\th_i)+1)/m_n}g_i(y)\d y\right]\\
\leq\prod_{i=1}^k\|f_i\|_\infty\|g_i\|_\infty\mbf E\left[|F_{n,t_i}(S^{i;x_i^n})|^k\right]^{1/k}.
\end{multline}
Our aim is to find a suitable upper bounds for the functions
\[x\mapsto\mbf E\left[|F_{n,t_i}(S^{i;x^n})|^k\right]^{1/k},\qquad1\leq i\leq k.\]
In order to achieve this, we fix a small $\eps>0$ (precisely how small will be determined in the following
paragraphs), and we consider separately the two cases $x\in[0,n^{1-\eps}/m_n)$ and $x\in[n^{1-\eps}/m_n,(n+1)/m_n)$.

Let us first consider the case $x\in[0,n^{1-\eps}/m_n)$.
Note that for any $E\in\{D,U,L\}$,
\[
\mbf 1_{\{\tau^{(n)}(S)>\th\}}\prod_{a\in\mbb N_0}\left|1-\frac{E_n(a)}{m_n^2}\right|^{\La^{(a_E,\bar a_E)}_\th(S)}
\leq\prod_{a\in\mbb Z}\left|1-\frac{E_n(|a|)}{m_n^2}\right|^{\La^{(a_E,\bar a_E)}_\th(S)}.
\]
Then, by combining H\"older's inequality with a rearrangement similar to \eqref{Equation: Uniform Integrability Decomposition},
$\mbf E\left[|F_{n,t_i}(S^{i;x^n})|^k\right]^{1/k}$
is bounded above by the product of the two terms
\begingroup
\allowdisplaybreaks
\begin{align}
\label{Equation: Vitali Convergence Noise 1}
&\prod_{E\in\{D,U,L\}}\mbf E\left[\prod_{a\in\mbb Z}\left|1-\frac{\xi^E_n(|a|)}{m_n^2-V^E_n(|a|)}\right|
^{6k\La^{(a_E,\bar a_E)}_{\th_i}(S^{i;x^n})}
\right]^{1/6k},\\
\label{Equation: Vitali Convergence Potential 1}
&\prod_{E\in\{D,U,L\}}\mbf E\left[\prod_{a\in\mbb Z}\left|1-\frac{V_n^E(|a|)}{m_n^2}\right|^{6k\La^{(a_E,\bar a_E)}_{\th_i}(S^{i;x^n})}\right]^{1/6k}.
\end{align}
\endgroup
Since $m_n x=O(n^{1-\eps})=o(n)$, the random walk $S^{i;x^n}$ can only attain values of order $o(n)$ in $\th_i=O(m_n^2)=o(n)$ steps.
Thus, for $E\in\{D,U,L\}$, it follows from \eqref{Equation: Potential Growth Upper Bound} that $V^E_n(a)=o(m_n^2)$ for any
value attained by the walk
when $x\in[0,n^{1-\eps}/m_n)$.
By using the same argument as for
\eqref{Equation: Uniformly Integrability 2} (namely, the inequality \cite[(4.25)]{GorinShkolnikov} followed by
Proposition \ref{Proposition: Local Time Uniform Exponential Moments}), we conclude that
\eqref{Equation: Vitali Convergence Noise 1} is bounded
by a constant for large $n$. For \eqref{Equation: Vitali Convergence Potential 1},
let us assume without loss of generality that $V^D_n$ is the sequence (or at least one of the sequenes)
that satisfies \eqref{Equation: Log Potential Lower Bound}.
According to \eqref{Equation: Potential Absolute Bounds}, we have
\[\prod_{E\in\{U,L\}}\mbf E\left[\prod_{a\in\mbb Z}\left|1-\frac{V_n^E(|a|)}{m_n^2}\right|^{6k\La^{(a_E,\bar a_E)}_{\th_i}(S^{i;x^n})}\right]^{1/6k}\leq1\]
for large enough $n$.
For the terms involving $V^D_n$, since $|1-y|\leq\mr e^{-y}$ for any $y\in[0,1]$,
it follows from \eqref{Equation: Log Potential Lower Bound}
that, up to a constant $C$ independent of $n$ (depending on $\theta$ through $c=c(\theta)$ in
\eqref{Equation: Log Potential Lower Bound}), we have the upper bound
\begin{align}\label{Equation: Vitali Convergence Potential 1-2}
\mbf E\left[\prod_{a\in\mbb Z}\left|1-\frac{V_n^D(|a|)}{m_n^2}\right|^{6k\La^{(a,a)}_{\th_i}(S^{i;x^n})}\right]^{1/6k}
\leq C\,\mbf E\left[\exp\left(-\frac{6k\theta}{m_n^2}\sum_{a\in\mbb Z}\log(1+|a|/m_n)\La^{(a,a)}_{\th_i}(S^{i;x^n})\right)\right]^{1/6k}
\end{align}
for large enough $n$.
If we define $S^{i;x^n}:=x^n+S^0$ for all $x\geq0$,
then $\La^{(a,a)}_{\th_i}(S^{i;x^n})=\La^{(a-x^n,a-x^n)}_{\th_i}(S^0)$.
By combining this change of variables with the inequality
\[\log(1+|z+\bar z|)\geq\log(1+|z|)-\log(1+|\bar z|)\geq\log(1+|z|)-|\bar z|,\]
which is valid for all $z,\bar z\in\mbb R$,
we obtain that, up to a multiplicative constant independent of $n$,
\eqref{Equation: Vitali Convergence Potential 1-2} is bounded by
\[\mbf E\left[\exp\left(-\frac{6k\theta}{m_n^2}\sum_{a\in\mbb Z}\big(\log(1+x)-|a/m_n|\big)\La^{(a,a)}_{\th_i}(S^0)\right)\right]^{1/6k}.\]
Noting that $\La^{(a,a)}_{\th_i}\leq\La^{a}_{\th_i}$ for every $a\in\mbb Z$ and that
the vertex-occupation measures satisfy \eqref{Equation: Vertex-Occupation Property},
an application of H\"older's
inequality then implies that \eqref{Equation: Vitali Convergence Potential 1-2}
is bounded above by the product of the two terms
\begin{align}
\label{Equation: Vitali Convergence Potential 1-3}
&\mbf E\left[\exp\left(-\frac{12k\theta \log(1+x)}{m_n^2}
\sum_{a\in\mbb Z}\La^{(a,a)}_{\th_i}(S^0)\right)\right]^{1/12k}\\
\label{Equation: Vitali Convergence Potential 1-4}
&\mbf E\left[\exp\left(\frac{12k\theta}{m_n^2}\sum_{0\leq u\leq\th_i}\frac{|S^0(u)|}{m_n}\right)\right]^{1/12k}.
\end{align}
Recall the definition of the range $\mc R_{\th_i}(S^0)$
in \eqref{Equation: Random Walk Range}. Since
\[\mc R_{\th_i}(S^0)\geq\max_{0\leq u\leq\th_i}|S^0(u)|,\]
we conclude that there exists $C>0$ independent
of $n$ such that \eqref{Equation: Vitali Convergence Potential 1-4} is bounded by
the exponential moment $\mbf E\left[\mr e^{C\mc R_{\th_i}(S^0)/m_n}\right]^{1/12k}$.
Thus, by \eqref{Equation: Range Exponential Moments}, we see that
\eqref{Equation: Vitali Convergence Potential 1-4} is bounded
by a constant independent of $n$.
It now remains to control \eqref{Equation: Vitali Convergence Potential 1-3}.
To this end, we note that $\sum_{a\in\mbb Z}\La^{(a,a)}_{\th_i}(S^0)$,
which represents the total number of visits on the self-edges of $\mbb Z$
by $S^0$ before the $\th_i^{\mr{th}}$ step,
is a Binomial random variable with $\th_i$ trials and probability $1/3$. Thus, for small enough $\nu>0$, it
follows from Hoeffding's inequality that
\begin{align}\label{Equation: Uniform Integrability Hoeffding}
\mbf P\left[\sum_{a\in\mbb Z}\La^{(a,a)}_{\th_i}(S^0)<\nu m_n^2\right]\leq\mr e^{-cm_n^2}
\end{align}
for some $c>0$ independent of $n$.
By separating the expectation in \eqref{Equation: Vitali Convergence Potential 1-3} with respect to whether or not
the walk has taken less than $\nu m_n^2$ steps on self-edges, we may bound it above by
\[\left(\mr e^{-12k\nu\theta \log(1+x)}+\mr e^{-cm_n^2}\right)^{1/12k}
\leq(1+x)^{-\nu\theta}+\mr e^{-(c/12k)m_n^2}.\]
Combining all of these bounds together
with the fact that $m_n$ is of order $n^{\mf d}$ by \eqref{Equation: mn},
we finally conclude that for every $1\leq i\leq k$, there exists constants $c_1,c_2,c_3>0$ independent
of $n$ such that, for large enough $n$,
\begin{align}\label{Equation: Vitali for Small x}
\mbf E\left[|F_{n,t_i}(S^{i;x^n})|^k\right]^{1/k}
\leq c_1\left((1+x)^{-c_2\theta}+\mr e^{-c_3n^{2\mf d}}\right),\qquad x\in[0,n^{1-\eps}/m_n).
\end{align}

\begin{remark}
We emphasize that $c_2$ does not depend on $\theta$, and thus
the assumption \eqref{Equation: Log Potential Lower Bound} implies that we can make
$c_2\theta$ arbitrarily large by taking a large enough $\theta$. In particular, if we take $\theta>1/c_2$,
then $(1+x)^{-c_2\theta}$ is integrable on $[0,\infty)$.
\end{remark}

We now turn to the estimate in the case where $x\in[n^{1-\eps}/m_n,(n+1)/m_n)$.
By taking $\eps>0$ small enough (more specifically, such that $1-\eps>2\mf d$,
with $\mf d$ as in \eqref{Equation: mn}),
we can ensure that $m_nx\geq n^{1-\eps}$ implies that,
for any constant $0<C<1$, we have
$S^{i;x^n}(u)\geq Cn^{1-\eps}$ for all $0\leq u\leq\th_i$ and $n$ large enough.
Let us assume without loss of generality that $V^D_n$ satisfies
\eqref{Equation: Polynomial Potential Lower Bound}.
Provided $\eps>0$ is small enough (namely, at least as small as the $\eps$
in \eqref{Equation: Polynomial Potential Lower Bound}), for any $a\in\mbb N_0$ that
can be visited by the random walk, we have that $V^D_n(a)\geq\ka(Cn^{1-\eps}/m_n)^{\al}$;
hence
\begin{multline}\label{Equation: Vitali Convergence 2-0}
\left|1-\frac{D_n(a)}{m_n^2}\right|\leq
\frac{m_n^2-V^D_n(a)}{m_n^2}
+\frac{|\xi^D_n(a)|}{m_n^2}
\leq\frac{m_n^2-\kappa(Cn^{1-\eps}/m_n)^\al}{m_n^2}
+\frac{|\xi^D_n(a)|}{m_n^2}\\
=\left(1-\frac{\kappa(Cn^{1-\eps}/m_n)^\al}{m_n^2}\right)\left(1+\frac{|\xi_n(a)|}{m_n^2-\kappa(Cn^{1-\eps}/m_n)^\al}\right).
\end{multline}
According to \eqref{Equation: mn}, we know that
$(n^{1-\eps}/m_n)^\al\asymp n^{\al(1-\mf d)-\al\eps}$.
Since $\al$ is chosen such that $\mf d/2<\al(1-\mf d)\leq2\mf d$ in
Assumption \ref{Assumption: Growth Lower Bounds}, we can always choose $\eps>0$
small enough so as to guarantee that
\begin{align}\label{Equation: Comparison of Powers of n}
n^{\mf d/2}=o(n^{\al(1-\mf d)-\al\eps})
\qquad\text{and}\qquad
(n^{1-\eps}/m_n)^\al=o(n^{2\mf d})=o(m_n^2).
\end{align}
As a consequence of the second equation in \eqref{Equation: Comparison of Powers of n},
for $n$ large enough \eqref{Equation: Vitali Convergence 2-0} yields
\begin{align*}
\left|1-\frac{D_n(a)}{m_n^2}\right|\leq
\left(1-\frac{\kappa(Cn^{1-\eps}/m_n)^\al}{m_n^2}\right)\left(1+\frac{2|\xi_n(a)|}{m_n^2}\right).
\end{align*}
As for $E\in\{U,L\}$, we have from \eqref{Equation: Potential Absolute Bounds} that
\[\left|1-\frac{E_n(a)}{m_n^2}\right|\leq
\frac{|m_n^2-V^E_n(a)|}{m_n^2}
+\frac{|\xi^E_n(a)|}{m_n^2}
\leq1+\frac{|\xi^E_n(a)|}{m_n^2}.\]
Thus, for any $x\in[n^{1-\eps}/m_n,(n+1)/m_n)$ and large enough $n$, it follows from H\"older's inequality that the expectation
$\mbf E\left[|F_{n,t_i}(S^{i;x^n})|^k\right]^{1/k}$ is bounded above by the product of the following three terms:
\begingroup
\allowdisplaybreaks
\begin{align}
\label{Equation: Vitali Convergence 2-1}
&\mbf E\left[\prod_{a\in\mbb Z}\left(1-\frac{\kappa(Cn^{1-\eps}/m_n)^\al}{m_n^2}\right)^{4 k\La^{(a,a)}_{\th_i}(S^{i;x^n})}\right]^{1/4 k}\\
\label{Equation: Vitali Convergence 2-2}
&\mbf E\left[\prod_{a\in\mbb Z}\left(1+\frac{2|\xi^D_n(|a|)|}{m_n^2}\right)^{4 k\La^{(a,a)}_{\th_i}(S^{i;x^n})}\right]^{1/4 k}\\
\label{Equation: Vitali Convergence 2-3}
&\prod_{E\in\{U,L\}}\mbf E\left[\prod_{a\in\mbb Z}\left(1+\frac{|\xi^E_n(|a|)|}{m_n^2}\right)^{4 k\La^{(a_E,\bar a_E)}_{\th_i}(S^{i;x^n})}\right]^{1/4 k}.
\end{align}
\endgroup
By repeating the bound \eqref{Equation: Uniform Integrability Hoeffding}
and the argument thereafter,
we conclude that there exists $c_4,c_5>0$ independent of $n$ such that \eqref{Equation: Vitali Convergence 2-1} is bounded by
$\mr e^{-c_4n^{\al(1-\mf d)-\al\eps}}+\mr e^{-c_5n^{2\mf d}}$. For \eqref{Equation: Vitali Convergence 2-2},
let us define $\ze_n(a):=|\xi^D_n(a)|/m_n^{1/2}$. By applying \cite[(4.25)]{GorinShkolnikov}
in similar fashion to \eqref{Equation: Noise Exponential Bound 1}, we see that
\eqref{Equation: Vitali Convergence 2-2} is bounded above by
\begin{multline}\label{Equation: Noise Exponential Bound 2}
\mbf E\Bigg[\exp\Bigg(C'\bigg(
\frac{1}{m_n^{1/2}}\sum_{a\in\mbb Z}\frac{\La^{a}_{\th_i}(S^{i;x_i^n})}{m_n}\mbf E[|\ze_n(|a|)|]\\
+\frac{1}{m_n}\sum_{a\in\mbb Z}\frac{\La^{a}_{\th_i}(S^{i;x_i^n})^2}{m_n^2}
+\frac{1}{m_n^{{\ga'}/2}}\sum_{a\in\mbb Z}\frac{\La^{a}_{\th_i}(S^{i;x_i^n})^{\ga'}}{m_n^{\ga'}}
\bigg)\Bigg)\Bigg]
\end{multline}
for some $C'>0$ and $2<{\ga'}<3$ independent of $n$.
By \eqref{Equation: Noise Moment Bound},
the moments $\mbf E[|\ze_n(a)|]$ are uniformly bounded in $n$, and thus
\[\frac{1}{m_n^{1/2}}\sum_{a\in\mbb Z}\frac{\La^{a}_{\th_i}(S^{i;x_i^n})}{m_n}\mbf E[|\ze_n(|a|)|]=O(m_n^{1/2})=O(n^{\mf d/2}).\]
By applying the uniform exponential moment bounds of Proposition \ref{Proposition: Local Time Uniform Exponential Moments}
to the remaining terms in \eqref{Equation: Noise Exponential Bound 2}, we conclude that there exists a constant $c_6>0$
independent of $n$ such that \eqref{Equation: Vitali Convergence 2-2} is bounded by $\mr e^{c_6n^{\mf d/2}}$.
A similar bound applies to \eqref{Equation: Vitali Convergence 2-3}. Then, by using the first
equality in \eqref{Equation: Comparison of Powers of n} and combining the inequalities for
\eqref{Equation: Vitali Convergence 2-1}--\eqref{Equation: Vitali Convergence 2-3},
we see that there exists $\bar c_4,\bar c_5>0$ independent of $n$ such that
\begin{align}\label{Equation: Vitali for Big x}
\mbf E\left[|F_{n,t_i}(S^{i;x^n})|^k\right]^{1/k}
\leq \mr e^{-\bar c_4 n^{\al(1-\mf d)-\al\eps}}
+\mr e^{-\bar c_5 n^{2\mf d}},\qquad x\in[n^{1-\eps}/m_n,(n+1)/m_n)
\end{align}

By combining \eqref{Equation: Vitali for Small x} and \eqref{Equation: Vitali for Big x},
we conclude that, for large $n$, the integral of the absolute value of \eqref{Equation: Vitali 0}
on the set $[0,(n+1)/m_n)^k$ is bounded above by
\[\left(\prod_{i=1}^k\|f_i\|_\infty\|g_i\|_\infty\right)
\Bigg(c_1\int_0^{n^{1-\eps}/m_n}(1+x)^{-c_2\theta}+\mr e^{-c_3n^{2\mf d}}\d x
+\int_{n^{1-\eps}/m_n}^{(n+1)/m_n}\mr e^{-\bar c_4 n^{\al(1-\mf d)-\al\eps}}
+\mr e^{-\bar c_5 n^{2\mf d}}\d x\Bigg)^k\]
for some $c_1,c_2,c_3,\bar c_4,\bar c_5>0$ independent of $n$. If we take $\theta>0$
large enough so that $(1+x)^{-c_2\theta}$ is integrable,
then the sequence of functions
\[\left(\prod_{i=1}^k\mbf1_{[0,(n+1)/m_n)}(x_i)\,f_i(x_i)\right)\mbf E\left[\prod_{i=1}^kF_{n,t_i}(S^{i;x_i^n})
\, m_n\int_{S^{i;x_i^n}(\th_i)/m_n}^{(S^{i;x_i^n}(\th_i)+1)/m_n}g_i(y)\d y\right]\]
is uniformly integrable in the sense of \cite[Theorem 2.24-(ii),(iii)]{FonsecaLeoni},
concluding the proof of the convergence of moments in Theorem \ref{Theorem: Main Dirichlet}-(1).

\subsection{Step 2: Convergence in Distribution}\label{Subsection: Distribution}

Up to writing each $f_i$ and $g_i$ as the difference of their positive
and negative parts, there is no loss of generality in assuming that $f_i,g_i\geq0$.
The convergence in joint distribution follows from the
convergence in moments proved in Section \ref{Subsection: Moments}. The argument we use
to prove this is essentially the same as \cite[Lemma 4.4]{GorinShkolnikov}:

For any $\ubar R\in[-\infty,0]$ and $\bar R\in[0,\infty]$, let us define
\[ \hat K_n^{\ubar R,\bar R}(t)g(x):=\mbf E^{\lfloor m_nx\rfloor}\left[\big(\ubar R\lor F_{n,t}(S)\land\bar R\big)\, m_n\int_{S(\th)/m_n}^{S(\th)+1)/m_n}g(y)\d y\right]\]
and
\[ \hat K^{\ubar R,\bar R}(t)g(x):=\mbf E^x\left[\left(\ubar R\lor\mbf 1_{\{\tau_0(B)>t\}}\mr e^{-\langle L_t(B),Q'\rangle}\land\bar R\right)g\big(B(t)\big)\right],\]
where we use the convention $\ubar R\lor y\land\bar R:=\max\big\{\ubar R,\min\{y,\bar R\}\big\}$ for any $y\in\mbb R$.
We note a few elementary properties of these truncated operators:
\begin{enumerate}
\item $ \hat K_n^{-\infty,\infty}(t)= \hat K_n(t)$, and $\hat K^{-\ubar R,\infty}(t)= \hat K(t)$
for all $\ubar R\leq 0$.
\item Arguing as in Section \ref{Subsection: Moments},
for every $\ubar R\in[-\infty,0]$ and $\bar R\in[0,\infty]$,
\begin{align}\label{Equation: Truncated Joint Moments}
\lim_{n\to\infty}\langle f_i, \hat K_n^{\ubar R,\bar R}(t_i)g_i\rangle=\langle f_i, \hat K^{\ubar R,\bar R}(t_i)g_i\rangle,\qquad 1\leq i\leq k
\end{align}
in joint moments.
\item If $|\ubar R|,\bar R<\infty$, then the $\langle f_i, \hat K_n^{\ubar R,\bar R}(t_i)g_i\rangle$ are bounded
uniformly in $n$; hence the moment convergence of \eqref{Equation: Truncated Joint Moments} implies convergence
in joint distribution.
\end{enumerate}

Let $\ubar R>-\infty$ be fixed.
Since $\langle f_i, \hat K^{\ubar R,\infty}_n(t_i)g_i\rangle\to\langle f_i, \hat K^{\ubar R,\infty}(t_i)g_i\rangle$ in joint moments,
the sequences in question are tight (e.g., \cite[Problem 25.17]{Billingsley95}).
Therefore, it suffices to prove that every subsequence that converges in joint distribution
has $\langle f_i, \hat K^{\ubar R,\infty}(t_i)g_i\rangle$ as a limit (e.g., \cite[Theorem--Corollary 25.10]{Billingsley95}).
Let $\mc A_1^{\ubar R},\ldots,\mc A_k^{\ubar R}$ be limit points of $\langle f_1, \hat K^{\ubar R,\infty}(t_1)g_1\rangle,
\ldots,\langle f_k, \hat K^{\ubar R,\infty}(t_k)g_k\rangle$. Since $f_i,g_i\geq0$,
the variables $\langle f_i, \hat K_n^{\ubar R,\bar R}(t_i)g_i\rangle$ and $\langle f_i, \hat K^{\ubar R,\bar R}(t_i)g_i\rangle$
are increasing in $\bar R$. Therefore, for every $\bar R<\infty$, we have
\begin{align}\label{Equation: Stochastic Dominance}
(\mc A_1^{\ubar R},\ldots,\mc A_k^{\ubar R})
\geq\big(\langle f_1, \hat K^{\ubar R,\bar R}(t_1)g_1\rangle,\ldots,\langle f_k, \hat K^{\ubar R,\bar R}(t_k)g_k\rangle\big)
\end{align}
in the sense of stochastic dominance in the space $\mbb R^k$ with the componentwise order
(e.g. \cite[Theorem 1 and Proposition 3]{KamaeKrengelOBrien}).
 By the monotone convergence theorem,
\[\lim_{\bar R\to\infty}\langle f_i, \hat K^{\ubar R,\bar R}(t_i)g_i\rangle=\langle f_i, \hat K^{\ubar R,\infty}(t_i)g_i\rangle,\qquad 1\leq i\leq k\]
almost surely; hence the stochastic dominance \eqref{Equation: Stochastic Dominance} also holds for
$\bar R=\infty$.
Since $\mc A_i^{\ubar R}$ and
$\langle f_i, \hat K^{\ubar R,\infty}(t_i)g_i\rangle$ have the same
mixed moments, we thus infer that their joint distributions coincide.
In conclusion, for any finite $\ubar R$, we have that
\[\lim_{n\to\infty}\langle f_i, \hat K_n^{\ubar R,\infty}(t_i)g_i\rangle=\langle f_i, \hat K_n^{\ubar R,\infty}(t_i)g_i\rangle\]
in joint distribution. In order to get the result for $\ubar R=-\infty$, we use the same stochastic
domination argument by sending $\ubar R\to-\infty$.

\subsection{Proof of Proposition \ref{Proposition: Local Time Uniform Exponential Moments}}
\label{Subsection: Local Time Uniform Exponential Moments}

If we prove that
\[\sup_{n\in\mbb N}
\mbf E\left[\exp\left(\frac{C}{m_n}\sum_{a\in\mbb Z}\frac{\La^{a}_{\th}(S^0)^q}{m_n^q}\right)\right]<\infty,\]
then we get the desired result by a simple change of variables.
Similarly to \cite[Proposition 4.3]{GorinShkolnikov}, a crucial tool for proving
this consists of combinatorial identities
involving the quantile transform for random walks derived in
\cite{AssafFormanPitman}. However, such results only apply to the simple symmetric random walk.

In order to get around this requirement, we decompose the vertex-occupation
measures in terms of the edge-occupations measures as follows:
By combining
\[\La_\th^a(S^0)=\La_\th^{(a,a-1)}(S^0)+\La_\th^{(a,a)}(S^0)+\La_\th^{(a,a+1)}(S^0)+\mbf 1_{\{S^0(\th)=a\}},\qquad a\in\mbb Z\]
with the inequality $(z+\bar z)^q\leq 2^{q-1}(z^q+\bar z^q)$ (for $z,\bar z\geq0$ and $q\geq1$),
it suffices by an application of H\"older's inequality to prove that the exponential moments of
\begin{align}\label{Equation: Vertical Steps}
\frac{1}{m_n}\sum_{a\in\mbb Z}\frac{\big(\La^{(a,a-1)}_{\th}(S^0)+\La^{(a,a+1)}_{\th}(S^0)\big)^q}{m_n^q}
\end{align}
and
\begin{align}\label{Equation: Horizontal Steps}
\frac{1}{m_n}\sum_{a\in\mbb Z}\frac{\La^{(a,a)}_{\th}(S^0)^q}{m_n^q}
\end{align}
are uniformly bounded in $n$.

\subsubsection{Non-Self-Edges}
\label{Subsubsection: Non-Self-Edges}

Let us begin with \eqref{Equation: Vertical Steps}.

\begin{definition}\label{Definition: SSRW}
Let $\mf S$ be a simple symmetric random walk on
$\mbb Z$, that is, the increments
$\mf S(u)-\mf S(u-1)$ are i.i.d. uniform on $\{-1,1\}$.
For any $a,b\in\mbb Z$ and $u\in\mbb N_0$, we denote $\mf S^a:=\big(\mf S|\mf S(0)=a\big)$
and $\mf S^{a,b}_u:=\big(\mf S|\mf S(0)=a\text{ and }\mf S(u)=b\big)$
(note that the latter only makes sense if $|b-a|$ and $u$ have the same parity).
\end{definition}

For every $u\in\mbb N_0$, let
\begin{align}
\mc H_u(S^0):=\sum_{a\in\mbb Z}\La^{(a,a)}_u(S^0),
\end{align}
i.e., the number of times $S^0$ visits self-edges by the $u^{\mr{th}}$ step.
Then, it is easy to see that we can couple $S^0$ and $\mf S^0$ in such a way that
\[S^0(u)=\mf S^0\big(u-\mc H_u(S^0)\big),\qquad u\in\mbb N,\]
i.e., $\mf S^0$ is the same path as $S^0$ with the visits to self-edges removed.
If we define the edge-occupation measures for $\mf S^0$ in the same way as
\eqref{Equation: Vertex-Occupation Measure},
then it is clear that the coupling of $S$ and $\mf S$ satisfies
\[\La^{(a,a-1)}_{\th}(S)+\La^{(a,a+1)}_{\th}(S)\leq\La^a_{\th}(\mf S).\]
Thus, for \eqref{Equation: Vertical Steps} we need only prove that the exponential moments of
\begin{align}\label{Equation: Vertical Steps Reduction}
\frac{1}{m_n}\sum_{a\in\mbb Z}\frac{\La^{a}_{\th}(\mf S^0)^q}{m_n^q}
\end{align}
are uniformly bounded in $n$.

By the total probability rule, we note that
\[\mbf E\left[\exp\left(\frac{C}{m_n}\sum_{a\in\mbb Z}\frac{\La^{a}_{\th}(\mf S^0)^q}{m_n^q}\right)\right]
=\sum_{b\in\mbb Z}\mbf E\left[\exp\left(\frac{C}{m_n}\sum_{a\in\mbb Z}\frac{\La^{a}_{\th}(\mf S_\th^{0,b})^q}{m_n^q}\right)\right]
\mbf P[\mf S^0(\th)=b].\]
According to the proof of \cite[Proposition 4.3]{GorinShkolnikov} (more specifically, \cite[(4.19)]{GorinShkolnikov}
and the following paragraph, explaining the distribution of the quantity denoted $M(N,\tilde T)$ in \cite[(4.19)]{GorinShkolnikov}),
there exists a constant $\bar C>0$ that only depends on $C$, $q$, and the number
$t$ in $\th=\lfloor m_n^2t\rfloor$ such that
\begin{align}\label{Equation: Local Time to Range Inequality}
\frac1{m_n}\sum_{a\in\mbb Z}\frac{\La^{a}_{\th}(\mf S_\th^{0,b})^q}{m_n^q}\leq \bar C\big((\mf R_\th^{0,b}/m_n)^{q-1}+\big((|b|+2)/m_n\big)^{q-1}\big),
\end{align}
where $\mf R_\th^{0,b}$ is equal in distribution to the range of $\mf S_\th^{0,b}$, that is,
\[\mf R_\th^{0,b}\deq
\mc R_\th(\mf S^{0,b}_\th):=\max_{0\leq u\leq\th}\mf S_\th^{0,b}(u)-\min_{0\leq u\leq\th}\mf S_\th^{0,b}(u).\]
Hence, if $\mc R_\th(\mf S^0)$ denotes the range of the unconditioned random walk $\mf S^0$, then
\[\mbf E\left[\exp\left(\frac{\bar C}{m_n}\sum_{a\in\mbb Z}\frac{\La^{a}_{\th}(\mf S^0)^q}{m_n^q}\right)\right]
\leq\mbf E\left[\exp\left(C\left(
\frac{(\mc R_\th(\mf S^0))^{q-1}}{m_n^{q-1}}+\frac{(|\mf S^0(\th)|+2)^{q-1}}{m_n^{q-1}}
\right)\right)\right].
\]
Since $q-1<2$, the result then follows from the same moment estimate leading up to
\eqref{Equation: Range Exponential Moments}, but by applying \cite[(6.2.3)]{Chen}
to the random walk $\mf S^0$ instead of $S^0$.

\subsubsection{Self-Edges}
\label{Subsubsection: Self-Edges}

We now control the exponential moments of \eqref{Equation: Horizontal Steps}.
By referring to the uniform boundedness of the exponential moments of \eqref{Equation: Vertical Steps} that we
have just proved, we know that for any $b\in\{-1,1\}$, the exponential moments of
\[\frac{1}{m_n}\sum_{a\in\mbb Z}\frac{\La^{(a,a+b)}_{\th}(S^0)^q}{m_n^q}
\qquad\text{and}\qquad
\frac{1}{m_n}\sum_{a\in\mbb Z}\frac{\La^{(a+b,a)}_{\th}(S^0)^q}{m_n^q}\]
are uniformly bounded in $n$. Thus, by applying $(x+y)^q\leq 2^q\big(|x|^q+|y|^q\big)$,
the exponential moments of 
\[\frac{1}{m_n}\sum_{a\in\mbb Z}\frac{\big(\La^{(a+1,a)}_{\th}(S^0)+\La^{(a-1,a)}_{\th}(S^0)+\La^{(a,a+1)}_{\th}(S^0)+\La^{(a,a-1)}_{\th}(S^0)\big)^q}{m_n^q}\]
are uniformly bounded in $n$. Consequently, it suffices to prove that
there exists $c,\bar c>0$ such that for every $n\in\mbb N$ and $y$ large enough (independently of $n$),
\begin{multline}\label{Equation: Stochastic Domination-Type}
\mbf P\left[\sum_{a\in\mbb Z}\La^{(a,a)}_{\th}(S^0)^q>y\right]\\
\leq\mbf P\left[\sum_{a\in\mbb Z}\big(\La^{(a+1,a)}_{\th}(S^0)+\La^{(a-1,a)}_{\th}(S^0)+\La^{(a,a+1)}_{\th}(S^0)+\La^{(a,a-1)}_{\th}(S^0)\big)^q>cy-\bar c\right].
\end{multline}
We now prove \eqref{Equation: Stochastic Domination-Type}.

\begin{definition}
If $\th$ is even, let $\mc S_0,\mc S_1,\ldots,\mc S_{\th/2-1}$ be defined as
the path segments
\[\mc S_u=\big(S^0(2u),S^0(2u+1),S^0(2u+2)\big),\qquad 0\leq u\leq \th/2-1.\]
If $\th$ is odd, then we similarly define $\mc S_0,\mc S_1,\ldots,\mc S_{(\th-1)/2-1},\mc S_{(\th-1)/2}$ as
\[\mc S_u=\begin{cases}\big(S^0(2u),S^0(2u+1),S^0(2u+2)\big)&\qquad\text{if }0\leq u\leq (\th-1)/2-1,\\
\big(S^0(2u),S^0(2u+1)\big)&\qquad\text{if }u= (\th-1)/2.\end{cases}\]
In words, we partition the path formed by the first $\th$ steps of $S^0$ into successive segments
of two steps, with the exception that the very last segment may contain only one step if $\th$ is odd
(see Figure \ref{Figure: Horizontal Coupling} below for an illustration of this partition).
\end{definition}

\begin{definition}
Let $\mc S_u$ be a path segment as in the previous definition. We say that $\mc S_u$ is
a {\bf type 1} segment if there exists some $a\in\mbb Z$ and $b\in\{-1,1\}$ such that
\[\mc S_u=\begin{cases}
(a,a,a+b),\\
(a+b,a,a),\text{ or}\\
(a,a),
\end{cases}\]
we say that $\mc S_u$ is
a {\bf type 2} segment if there exists some $a\in\mbb Z$ such that
\[\mc S_u=(a,a,a),\]
and we say that $\mc S_u$ is
a {\bf type 3} segment if there exists some $a\in\mbb Z$ such that
\[\mc S_u=(a,a+1,a).\]
\end{definition}

Given a realization of the first $\th$ steps of the lazy random walk $S^0$,
we define the transformed path $\big(\hat S^0(u)\big)_{0\leq u\leq\th}$
by replacing every type 2 segment $(a,a,a)$ in $\big(S^0(u)\big)_{0\leq u\leq\th}$ by the corresponding type 3
segment $(a,a+1,a)$, and vice versa.
(see Figure \ref{Figure: Horizontal Coupling} below for an illustration of this transformation).
Given that this path transformation is a bijection on the set of all possible realizations of $\big(S^0(u)\big)_{0\leq u\leq\th}$,
$\big(\hat S^0(u)\big)_{0\leq u\leq\th}$ is also a lazy random walk.

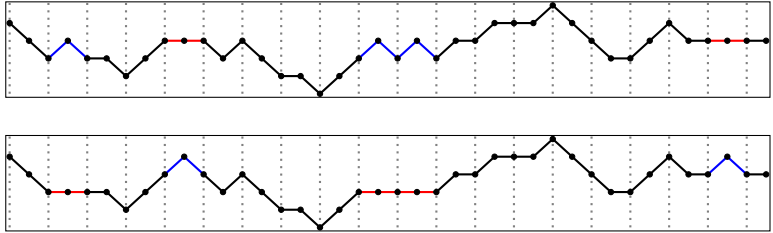
\begin{figure}[htbp]
\begin{center}
%
%
\definecolor{mycolor1}{rgb}{0.00000,0.44700,0.74100}%
\definecolor{mycolor2}{rgb}{0.85000,0.32500,0.09800}%
\begin{tikzpicture}

\begin{axis}[%
width=4in,
height=0.5in,
at={(0in,0in)},
scale only axis,
xmin=0.8,
xmax=40.2,
ymin=-4.2,
ymax=1.2,
xtick=\empty,ytick=\empty,
axis background/.style={fill=white}
]
\addplot [thick,forget plot,dotted,gray]
  table[row sep=crcr]{%
1	-4\\
1	1.2\\
};
\addplot [thick,forget plot,dotted,gray]
  table[row sep=crcr]{%
3	-4\\
3	1.2\\
};
\addplot [thick,forget plot,dotted,gray]
  table[row sep=crcr]{%
5	-4\\
5	1.2\\
};
\addplot [thick,forget plot,dotted,gray]
  table[row sep=crcr]{%
7	-4\\
7	1.2\\
};
\addplot [thick,forget plot,dotted,gray]
  table[row sep=crcr]{%
9	-4\\
9	1.2\\
};
\addplot [thick,forget plot,dotted,gray]
  table[row sep=crcr]{%
11	-4\\
11	1.2\\
};
\addplot [thick,forget plot,dotted,gray]
  table[row sep=crcr]{%
13	-4\\
13	1.2\\
};
\addplot [thick,forget plot,dotted,gray]
  table[row sep=crcr]{%
15	-4\\
15	1.2\\
};
\addplot [thick,forget plot,dotted,gray]
  table[row sep=crcr]{%
17	-4\\
17	1.2\\
};
\addplot [thick,forget plot,dotted,gray]
  table[row sep=crcr]{%
19	-4\\
19	1.2\\
};
\addplot [thick,forget plot,dotted,gray]
  table[row sep=crcr]{%
21	-4\\
21	1.2\\
};
\addplot [thick,forget plot,dotted,gray]
  table[row sep=crcr]{%
23	-4\\
23	1.2\\
};
\addplot [thick,forget plot,dotted,gray]
  table[row sep=crcr]{%
25	-4\\
25	1.2\\
};
\addplot [thick,forget plot,dotted,gray]
  table[row sep=crcr]{%
27	-4\\
27	1.2\\
};
\addplot [thick,forget plot,dotted,gray]
  table[row sep=crcr]{%
29	-4\\
29	1.2\\
};
\addplot [thick,forget plot,dotted,gray]
  table[row sep=crcr]{%
31	-4\\
31	1.2\\
};
\addplot [thick,forget plot,dotted,gray]
  table[row sep=crcr]{%
33	-4\\
33	1.2\\
};
\addplot [thick,forget plot,dotted,gray]
  table[row sep=crcr]{%
35	-4\\
35	1.2\\
};
\addplot [thick,forget plot,dotted,gray]
  table[row sep=crcr]{%
37	-4\\
37	1.2\\
};
\addplot [thick,forget plot,dotted,gray]
  table[row sep=crcr]{%
39	-4\\
39	1.2\\
};
\addplot [color=mycolor1, draw=none, mark=*, mark size=1pt, mark options={solid,black}, forget plot]
  table[row sep=crcr]{%
1	0\\
2	-1\\
3	-2\\
4	-1\\
5	-2\\
6	-2\\
7	-3\\
8	-2\\
9	-1\\
10	-1\\
11	-1\\
12	-2\\
13	-1\\
14	-2\\
15	-3\\
16	-3\\
17	-4\\
18	-3\\
19	-2\\
20	-1\\
21	-2\\
22	-1\\
23	-2\\
24	-1\\
25	-1\\
26	0\\
27	0\\
28	0\\
29	1\\
30	0\\
31	-1\\
32	-2\\
33	-2\\
34	-1\\
35	0\\
36	-1\\
37	-1\\
38	-1\\
39	-1\\
40	-1\\
};
\addplot [thick,forget plot]
  table[row sep=crcr]{%
1	0\\
2	-1\\
3	-2\\
};
\addplot [thick,forget plot,blue]
  table[row sep=crcr]{%
3	-2\\
4	-1\\
5	-2\\
};
\addplot [thick,forget plot]
  table[row sep=crcr]{%
5	-2\\
6	-2\\
7	-3\\
8	-2\\
9	-1\\
};
\addplot [thick,forget plot,red]
  table[row sep=crcr]{%
9	-1\\
10	-1\\
11	-1\\
};
\addplot [thick,forget plot]
  table[row sep=crcr]{%
11	-1\\
12	-2\\
13	-1\\
14	-2\\
15	-3\\
16	-3\\
};
\addplot [thick,forget plot]
  table[row sep=crcr]{%
16	-3\\
17	-4\\
18	-3\\
19	-2\\
};
\addplot [thick,forget plot,blue]
  table[row sep=crcr]{%
19	-2\\
20	-1\\
21	-2\\
22	-1\\
23	-2\\
};
\addplot [thick,forget plot]
  table[row sep=crcr]{%
23	-2\\
24	-1\\
25	-1\\
26	0\\
27	0\\
28	0\\
29	1\\
30	0\\
31	-1\\
32	-2\\
33	-2\\
34	-1\\
35	0\\
36	-1\\
37	-1\\
};
\addplot [thick,forget plot,red]
  table[row sep=crcr]{%
37	-1\\
38	-1\\
39	-1\\
};
\addplot [thick,forget plot]
  table[row sep=crcr]{%
39	-1\\
40	-1\\
};
\end{axis}

\begin{axis}[%
width=4in,
height=0.5in,
at={(0in,-0.7in)},
scale only axis,
xmin=0.8,
xmax=40.2,
ymin=-4.2,
ymax=1.2,
xtick=\empty,ytick=\empty,
axis background/.style={fill=white}
]
\addplot [thick,forget plot,dotted,gray]
  table[row sep=crcr]{%
1	-4\\
1	1.2\\
};
\addplot [thick,forget plot,dotted,gray]
  table[row sep=crcr]{%
3	-4\\
3	1.2\\
};
\addplot [thick,forget plot,dotted,gray]
  table[row sep=crcr]{%
5	-4\\
5	1.2\\
};
\addplot [thick,forget plot,dotted,gray]
  table[row sep=crcr]{%
7	-4\\
7	1.2\\
};
\addplot [thick,forget plot,dotted,gray]
  table[row sep=crcr]{%
9	-4\\
9	1.2\\
};
\addplot [thick,forget plot,dotted,gray]
  table[row sep=crcr]{%
11	-4\\
11	1.2\\
};
\addplot [thick,forget plot,dotted,gray]
  table[row sep=crcr]{%
13	-4\\
13	1.2\\
};
\addplot [thick,forget plot,dotted,gray]
  table[row sep=crcr]{%
15	-4\\
15	1.2\\
};
\addplot [thick,forget plot,dotted,gray]
  table[row sep=crcr]{%
17	-4\\
17	1.2\\
};
\addplot [thick,forget plot,dotted,gray]
  table[row sep=crcr]{%
19	-4\\
19	1.2\\
};
\addplot [thick,forget plot,dotted,gray]
  table[row sep=crcr]{%
21	-4\\
21	1.2\\
};
\addplot [thick,forget plot,dotted,gray]
  table[row sep=crcr]{%
23	-4\\
23	1.2\\
};
\addplot [thick,forget plot,dotted,gray]
  table[row sep=crcr]{%
25	-4\\
25	1.2\\
};
\addplot [thick,forget plot,dotted,gray]
  table[row sep=crcr]{%
27	-4\\
27	1.2\\
};
\addplot [thick,forget plot,dotted,gray]
  table[row sep=crcr]{%
29	-4\\
29	1.2\\
};
\addplot [thick,forget plot,dotted,gray]
  table[row sep=crcr]{%
31	-4\\
31	1.2\\
};
\addplot [thick,forget plot,dotted,gray]
  table[row sep=crcr]{%
33	-4\\
33	1.2\\
};
\addplot [thick,forget plot,dotted,gray]
  table[row sep=crcr]{%
35	-4\\
35	1.2\\
};
\addplot [thick,forget plot,dotted,gray]
  table[row sep=crcr]{%
37	-4\\
37	1.2\\
};
\addplot [thick,forget plot,dotted,gray]
  table[row sep=crcr]{%
39	-4\\
39	1.2\\
};
\addplot [color=mycolor1, draw=none, mark=*, mark size=1pt, mark options={solid,black}, forget plot]
  table[row sep=crcr]{%
1	0\\
2	-1\\
3	-2\\
4	-2\\
5	-2\\
6	-2\\
7	-3\\
8	-2\\
9	-1\\
10	0\\
11	-1\\
12	-2\\
13	-1\\
14	-2\\
15	-3\\
16	-3\\
17	-4\\
18	-3\\
19	-2\\
20	-2\\
21	-2\\
22	-2\\
23	-2\\
24	-1\\
25	-1\\
26	0\\
27	0\\
28	0\\
29	1\\
30	0\\
31	-1\\
32	-2\\
33	-2\\
34	-1\\
35	0\\
36	-1\\
37	-1\\
38	0\\
39	-1\\
40	-1\\
};
\addplot [thick,forget plot]
  table[row sep=crcr]{%
1	0\\
2	-1\\
3	-2\\
};
\addplot [thick,forget plot,red]
  table[row sep=crcr]{%
3	-2\\
4	-2\\
5	-2\\
};
\addplot [thick,forget plot]
  table[row sep=crcr]{%
5	-2\\
6	-2\\
7	-3\\
8	-2\\
9	-1\\
};
\addplot [thick,forget plot,blue]
  table[row sep=crcr]{%
9	-1\\
10	0\\
11	-1\\
};
\addplot [thick,forget plot]
  table[row sep=crcr]{%
11	-1\\
12	-2\\
13	-1\\
14	-2\\
15	-3\\
16	-3\\
};
\addplot [thick,forget plot]
  table[row sep=crcr]{%
16	-3\\
17	-4\\
18	-3\\
19	-2\\
};
\addplot [thick,forget plot,red]
  table[row sep=crcr]{%
19	-2\\
20	-2\\
21	-2\\
22	-2\\
23	-2\\
};
\addplot [thick,forget plot]
  table[row sep=crcr]{%
23	-2\\
24	-1\\
25	-1\\
26	0\\
27	0\\
28	0\\
29	1\\
30	0\\
31	-1\\
32	-2\\
33	-2\\
34	-1\\
35	0\\
36	-1\\
37	-1\\
};
\addplot [thick,forget plot,blue]
  table[row sep=crcr]{%
37	-1\\
38	0\\
39	-1\\
};
\addplot [thick,forget plot]
  table[row sep=crcr]{%
39	-1\\
40	-1\\
};
\end{axis}
\end{tikzpicture}%
\caption{The partition into two-step segments is represented by dashed gray lines.
type 2 segments are red, and type 3 segments are blue. The two paths
represent $S^0$ and $\hat S^0$, as related to each other by the permutation of type 2 and 3 segments.}
\label{Figure: Horizontal Coupling}
\end{center}
\end{figure}

Every contribution of $S^0$ to $\sum_a\La_\th^{(a,a)}(S^0)$ comes from
type 1 and 2 segments. Moreover, if a type 1 segment $\mc S_u$ is not at the end
of the path and adds a contribution of one to $\La_\th^{(a,a)}(S^0)$ for some $a\in\mbb Z$,
then it must also add one to
\begin{align}\label{Equation: Transformed Contribution}
\La^{(a+1,a)}_{\th}(S^0)+\La^{(a-1,a)}_{\th}(S^0)+\La^{(a,a+1)}_{\th}(S^0)+\La^{(a,a-1)}_{\th}(S^0).
\end{align}
Lastly, for every type 2 segment, a contribution of two to $\La_\th^{(a,a)}(S^0)$ for some $a\in\mbb Z$
is turned into a contribution of two to \eqref{Equation: Transformed Contribution} in $\hat S^0$.
In short, we observe that there is at most one $a_0\in\mbb Z$ (i.e., the one level, if any, where a type 1 segment
occurs at the very end of the path of $S^0(u)$, $u\leq\th$) such that
\[\La^{(a,a)}_\th(S^0)\leq\La^{(a+1,a)}_{\th}(\hat S^0)+\La^{(a-1,a)}_{\th}(\hat S^0)+\La^{(a,a+1)}_{\th}(\hat S^0)+\La^{(a,a-1)}_{\th}(\hat S^0)\]
for every $a\in\mbb Z\setminus\{a_0\}$, and
\[\La^{(a_0,a_0)}_\th(S^0)\leq\La^{(a_0+1,a_0)}_{\th}(\hat S^0)+\La^{(a_0-1,a_0)}_{\th}(\hat S^0)+\La^{(a_0,a_0+1)}_{\th}(\hat S^0)+\La^{(a_0,a_0-1)}_{\th}(\hat S^0)+1\]
Given that $(z+1)^q\leq 2^{q-1}z^q+2^{q-1}$ for every $z,q\geq1$, we obtain \eqref{Equation: Stochastic Domination-Type}.

\section{Proof of Theorem \ref{Theorem: Main Dirichlet}-(2)}
\label{Section: Main Dirichlet 2}

This proof is very similar to that of Theorem \ref{Theorem: Main Dirichlet}-(1),
except that we deal with random walks and Brownian motions
conditioned on their endpoint.

\subsection{Step 1: Convergence of Moments}
\label{Subsection: Bridge Moment Convergence}

We begin with a generic mixed moment of traces, which we can always write in the form
\[\mbf E\left[\prod_{i=1}^k\mr{Tr}\big[ \hat K_n(t_i)\big]\right].\]
By Fubini's theorem, this is equal to
\begin{align}\label{Equation: Trace Integral}
\int_{[0,(n+1)/m_n]^k}\mbf E\left[\prod_{i=1}^km_n\mbf P[S^0(\th_i)=0]F_{n,t_i}(S^{i;x_i^n,x_i^n}_{\th_i})\right]\d x_1\cdots\dd x_k,
\end{align}
and by the trace formula in Remark \ref{Remark: Strong Continuity} the corresponding continuum limit is
\[\mbf E\left[\prod_{i=1}^k\mr{Tr}\big[ \hat K(t_i)\big]\right]
=\int_{\mbb R_+}\mbf E\left[\prod_{i=1}^k\frac{1}{\sqrt{2\pi t_i}}\mbf 1_{\{\tau_0(B^{i;x_i,x_i}_{t_i})>t\}}\mr e^{-\langle L_t(B^{i;x_i,x_i}_{t_i}),Q'\rangle}\right]\d x_1\cdots\dd x_k,\]
where $\th_i$ and $x_i^n$ are as in Section \ref{Subsection: Moments}, and
\begin{enumerate}
\item $S^{1;x_1^n,x_1^n}_{\th_1},\ldots,S^{k;x_k^n,x_k^n}_{\th_k}$ are independent copies
of random walk bridges $S^{x,x}_\th$ with $x=x_i^n$ and $\th=\th_i$;
\item $B^{1;x_1,x_1}_{t_1},\ldots,B^{k;x_k,x_k}_{t_k}$ are independent copies of standard Brownian bridges $B^{x,x}_t$
with $x=x_i$ and $t=t_i$.
\end{enumerate}
Also, $S^{i;x_i^n,x_i^n}_{\th_i}$ are independent of $Q_n$,
and $B^{i;x_i,x_i}_{t_i}$ are independent of $Q$.

According to the local central limit theorem,
\[\lim_{n\to\infty}m_n\mbf P[S^0(\th_i)=0]=\frac{1}{\sqrt{2\pi t_i}},\qquad 1\leq i\leq k.\]
Moreover, we have the following analog of Proposition \ref{Proposition: Moment Convergence}:

\begin{proposition}\label{Proposition: Trace Moment Convergence}
The conclusion of Proposition \ref{Proposition: Moment Convergence}
holds with every instance of $S^{i;x_i^n}$ replaced by $S^{i;x_i^n,x_i^n}_{\th_i}$,
and every instance of $B^{i;x_i}$ replaced by $B^{i;x_i,x_i}_{t_i}$.
\end{proposition}
\begin{proof}
Arguing as in the proof of Proposition \ref{Proposition: Moment Convergence}, this follows from
coupling $S^{i;x_i^n,x_i^n}$ with a Brownian bridge $\tilde B^{i;x_i,x_i}_{3t_i/2}$ with
variance $2/3$ using Theorem \ref{Theorem: Split Coupling}, and then defining
$B^{i;x_i,x_i}_{t_i}(s):=\tilde B^{i;x_i,x_i}_{3t_i/2}(3s/2)$.
\end{proof}

With these results in hand, by repeating the arguments in Section \ref{Subsubsection: Convergence Inside the Expected Value},
for any $x_1,\ldots,x_k\geq0$, we can find a coupling such that
\[\lim_{n\to\infty}m_n\mbf P[S^0(\th_i)=0]F_{n,t_i}(S^{i;x_i^n,x_i^n}_{\th_i})
=\frac{1}{\sqrt{2\pi t_i}}\mbf 1_{\{\tau_0(B^{i;x_i,x_i}_{t_i})>t\}}\mr e^{-\langle L_t(B^{i;x_i,x_i}_{t_i}),Q'\rangle}\]
in probability for $1\leq i\leq k$.
Then, by arguing as in Section \ref{Subsubsection: Convergence of the Expected Value}
(more specifically, the estimate for \eqref{Equation: Uniformly Integrability 1}), we get the convergence
\[\lim_{n\to\infty}\mbf E\left[\prod_{i=1}^km_n\mbf P[S^0(\th_i)=0]F_{n,t_i}(S^{i;x_i^n,x_i^n}_{\th_i})\right]
=\mbf E\left[\prod_{i=1}^k\frac{1}{\sqrt{2\pi t_i}}\mbf 1_{\{\tau_0(B^{i;x_i,x_i}_{t_i})>t\}}\mr e^{-\langle L_t(B^{i;x_i,x_i}_{t_i}),Q'\rangle}\right]\]
pointwise in $x_1,\ldots,x_k$ thanks to the following proposition, which we prove at
the end of this section.

\begin{proposition}\label{Proposition: Local Time Uniform Exponential Moments (Bridge)}
Let $\th=\th(n,t):=\lfloor m_n^2t\rfloor$ and $x^n:=\lfloor m_nx\rfloor$ for some $t>0$ and $x\geq0$.
For every $C>0$ and $1\leq q<3$,
\[\sup_{n\in\mbb N,~x\geq0}
\mbf E\left[\exp\left(\frac{C}{m_n}\sum_{a\in\mbb Z}\frac{\La^{a}_{\th}(S^{x^n,x^n}_\th)^q}{m_n^q}\right)\right]<\infty.\]
\end{proposition}

It only remains to prove that we can pass the limit outside the integral \eqref{Equation: Trace Integral}.
We once again use \cite[Theorem 2.24]{FonsecaLeoni}.
For this, it is enough to prove that, for $n$ large enough,
there exists constants $c_1,c_2,c_3,\bar c_4,\bar c_5>0$ such that
\begingroup
\allowdisplaybreaks
\begin{align}\label{Equation: Vitali Convergence Bridge}
\nonumber
&\int_{[0,(n+1)/m_n]^k}\left|\mbf E\left[\prod_{i=1}^kF_{n,t_i}(S^{i;x_i^n,x_i^n}_{\th_i})\right]\right|\d x_1\cdots\dd x_k
\leq\prod_{i=1}^k\int_0^{(n+1)/m_n}\,\mbf E\left[|F_{n,t_i}(S^{i;x^n,x^n}_{\th_i})|^k\right]^{1/k}\d x\\
&\leq\Bigg(c_1\int_0^{n^{1-\eps}/m_n}\left((1+x)^{-c_2\theta}+\mr e^{-c_3n^{2\mf d}}\right)\d x
+\int_{n^{1-\eps}/m_n}^{(n+1)/m_n}\left(\mr e^{-\bar c_4 n^{\al(1-\mf d)-\al\eps}}+\mr e^{-2\bar c_5 n^{2\mf d}}\right)\d x\Bigg)^k,
\end{align}
\endgroup
where $\theta$ is taken large enough so that $(1+x)^{-c_2\theta}$ is integrable. To this end,
for every $\th\in\mbb N$, let us define $\mc R_\th(S^{0,0}_\th)$ as the range of $S^{0,0}_\th$.
By replicating the estimates in Section \ref{Subsubsection: Convergence of Integral},
we see that \eqref{Equation: Vitali Convergence Bridge} is the consequence of the following
two propositions, concluding the proof of the convergence of moments.

\begin{proposition}\label{Proposition: Range of Bridge}
Let $\th=\th(n,t):=\lfloor m_n^2 t\rfloor$ for some $t>0$. For every $C>0$,
\[\sup_{n\in\mbb N}\mbf E\left[\mr e^{C\mc R_\th(S^{0,0}_\th)/m_n}\right]<\infty.\]
\end{proposition}

\begin{proposition}\label{Proposition: Bridge Binomial Concentration}
Let $\th=\th(n,t):=\lfloor m_n^2 t\rfloor$ for some $t>0$. For small enough $\nu>0$,
there exists some $c>0$ independent of $n$ such that
\[\mbf P\left[\sum_{a\in\mbb Z}\La^{(a,a+b)}_{\th}(S_\th^{0,0})<\nu m_n^2\right]\leq\mr e^{-cm_n^2},\qquad b\in\{-1,0,1\}.\]
\end{proposition}

\begin{proof}[Proof of Proposition \ref{Proposition: Range of Bridge}]
Let us define
\[\mc M(S^{0,0}_\th):=\max_{0\leq u\leq\th}|S^{0,0}_\th(u)|.\]
It is easy to see that $\mc R_\th(S^{0,0}_\th)\leq2\mc M(S^{0,0}_\th)$, and thus it suffices to prove that
the exponential moments of $\mc M(S^{0,0}_\th)/m_n$ are uniformly bounded in $n$.

Let $\mf S$ be as in Definition \ref{Definition: SSRW}, and define
\[\mc M(\mf S^{0,0}_v):=\max_{0\leq u\leq v}|\mf S^{0,0}_u|,\qquad v\in2\mbb N_0.\]
According to
\cite[(4.7)]{GorinShkolnikov} (up to normalization, the quantity denoted $\tilde M(N,\tilde T)$
in \cite[(4.7)]{GorinShkolnikov} is essentially the same as what we denote by $\mc M(\mf S^{0,0}_\th)$;
see the definition of the former on \cite[Page 2302]{GorinShkolnikov}) we know that for every $0<q<2$ and $C>0$,
\begin{align}\label{Equation: Range of Bridge 1}
\sup_{u\in\mbb N}\mbf E\left[\mr e^{C(\mc M(\mf S^{0,0}_u)/\sqrt{u})^q}\right]<\infty.
\end{align}
Let us define
\begin{align}\label{Equation: Number of Horizontal}
\mc H(S^{0,0}_u):=\sum_{a\in\mbb Z}\La^{(a,a)}_\th(S^{0,0}_u),\qquad u\in2\mbb N_0.
\end{align}
For any $h\in\mbb N_0$, we can couple the bridges of $S$ and $\mf S$
in such a way that
\[\big(S^{0,0}_\th(u)|\mc H(S^{0,0}_\th)=h\big)=\mf S_{\th-h}^{0,0}(u-\mc H(S^{0,0}_u)).\]
In words, we obtain $\mf S_{\th-\cdot}^{0,0}$ from $S^{0,0}_\th(u)$
by removing all segments that visit self-edges.
Since visits to self-edges do not contribute to the magnitude of $S^{0,0}_\th$,
\[(\mc M(S^{0,0}_\th)|\mc H(S^{0,0}_\th)=h)=\mc M(\mf S^{0,0}_{\th-h}).\]
Thus, \eqref{Equation: Range of Bridge 1} for $q=1$
implies that
\begin{align}\label{Equation: Range of Bridge 2}
&\sup_{n\in\mbb N}\mbf E\left[\mr e^{C\mc M(S^{0,0}_\th)/m_n}\right]
=\sup_{n\in\mbb N}\sum_{h\in\mbb N_0}\mbf E\left[\mr e^{C\mc M(S^{0,0}_\th)/m_n}\Big|\mc H(S^{0,0}_\th)=h\right]
\mbf P[\mc H(S^{0,0}_\th)=h]\\
\nonumber
&\leq\sup_{n\in\mbb N}\sup_{1\leq u\leq\th}\mbf E\left[\mr e^{(\sqrt{u}/m_n)C\mc M(\mf S^{0,0}_u)/\sqrt{u}}\right]<\infty
\end{align}
for every $C>0$, as desired.
\end{proof}

\begin{proof}[Proof of Proposition \ref{Proposition: Bridge Binomial Concentration}]
Note that
\begin{align*}
\mbf P\left[\sum_{a\in\mbb N_0}\La^{(a,a+b)}_{\th}(S_\th^{0,0})<\nu m_n^2\right]
\leq\mbf P\left[\sum_{a\in\mbb N_0}\La^{(a,a+b)}_{\th}(S^0)<\nu m_n^2\right]\mbf P\left[S^0(\th)=0\right]^{-1}.
\end{align*}
By the local central limit theorem, $\mbf P[S^0(\th)=0]^{-1}=O(m_n)$, and thus the result
follows from the same binomial concentration argument used for \eqref{Equation: Uniform Integrability Hoeffding}.
\end{proof}

\begin{proof}[Proof of Proposition \ref{Proposition: Local Time Uniform Exponential Moments (Bridge)}]
In similar fashion to the proof of Proposition \ref{Proposition: Local Time Uniform Exponential Moments},
it suffices to prove that the exponential moments of
\begin{align}\label{Equation: Bridge Vertical Steps}
\frac{1}{m_n}\sum_{a\in\mbb Z}\frac{\big(\La^{(a,a-1)}_{\th}(S^{0,0}_\th)+\La^{(a,a+1)}_{\th}(S^{0,0}_\th)\big)^q}{m_n^q}
\qquad\text{and}\qquad
\frac{1}{m_n}\sum_{a\in\mbb Z}\frac{\La^{(a,a)}_{\th}(S^{0,0}_\th)^q}{m_n^q}
\end{align}
are uniformly bounded in $n$.
We start with the first term in \eqref{Equation: Bridge Vertical Steps}.
Under the coupling in the proof of Proposition \ref{Proposition: Range of Bridge},
\[\left(\sum_{a\in\mbb Z}\big(\La^{(a,a-1)}_{\th}(S^{0,0}_\th)+\La^{(a,a+1)}_{\th}(S^{0,0}_\th)\big)^q
\bigg|\mc H(S^{0,0}_\th)=h\right)\leq\sum_{a\in\mbb Z}\La_{\th-h}^a(\mf S_{\th-h}^{0,0})^q\]
for every $h\in\mbb N_0$.
By conditioning on $\mc H(S^{0,0}_\th)$ as in \eqref{Equation: Range of Bridge 2},
we need only prove that
\[\sup_{n\in\mbb N}\mbf E\left[\exp\left(\frac{C}{m_n}\sum_{a\in\mbb Z}\frac{\La^{a}_{\th}(\mf S_\th^{0,0})^q}{m_n^q}\right)\right]<\infty.\]
By using \eqref{Equation: Local Time to Range Inequality} in the case $b=0$
(i.e., \cite[(4.19)]{GorinShkolnikov}), this follows from \eqref{Equation: Range of Bridge 1}.
With this established, the exponential moments of the second term in
\eqref{Equation: Bridge Vertical Steps} can be
controlled by using the same argument in Section \ref{Subsubsection: Self-Edges}
(the path transformation used therein
does not change the endpoint of the path that is being modified; hence the transformed version
of $S^{0,0}_\th$ is a random walk bridge).
\end{proof}

\subsection{Step 2: Convergence in Distribution}

The convergence in distribution follows from the convergence of mixed moments
by using the same truncation/stochastic domination argument as in Section \ref{Subsection: Distribution}.

\section{Proof of Theorem \ref{Theorem: Main Robin}}
\label{Section: Main Robin}

This follows roughly the same steps as the proof of Theorem
\ref{Theorem: Main Dirichlet}-(1).

\subsection{Step 1: Convergence of Moments}

\subsubsection{Expression for Mixed Moments and Convergence Result}

By Fubini's theorem, any mixed moment
$\mbf E\left[\prod_{i=1}^k\langle f_i, \hat K^w_n(t_i)g_i\rangle\right]$
can be written as
\begin{align}\label{Equation: Spiked Mixed Moment Prelimit}
\int_{[0,(n+1)/m_n)^k}\left(\prod_{i=1}^kf_i(x_i)\right)
\mbf E\left[\prod_{i=1}^kF_{n,t_i}(T^{i;x_i^n})
\,m_n\int_{T^{i;x_i^n}(\th_i)/m_n}^{(T^{i;x_i^n}(\th_i)+1)/m_n}g_i(y)\d y\right]\d x_1\cdots\dd x_k,
\end{align}
and the corresponding continuum limit is
\begin{align}\label{Equation: Spiked Mixed Moment Limit}
\mbf E\left[\prod_{i=1}^k\langle f_i, \hat K(t_i)g_i\rangle\right]
=\int_{\mbb R_+^k}\left(\prod_{i=1}^kf_i(x_i)\right)\,
\mbf E\left[\prod_{i=1}^k\mr e^{-\langle L_t(X^{i;x_i}),Q'\rangle-w\mf L^0_{t_i}(X^{i;x_i})} g_i\big(X^{i;x_i}(t)\big)\right]\d x_1\cdots\dd x_k,
\end{align}
where $\th_i$ and $x_i^n$ are as in Section \ref{Subsection: Moments},
\begin{enumerate}
\item $T^{1;x_1^n},\ldots,T^{k;x_k^n}$ are independent copies of the Markov chain $T$ with respective
starting points $x_1^n,\ldots,x_k^n$; and
\item $X^{1;x_1},\ldots,X^{k;x_k}$ are independent copies of $X$
with respective starting points $x_1,\ldots,x_k$.
\end{enumerate}
$T^{i;x_i^n}$ are independent of $Q_n$
and $X^{i;x_i}$ are independent of $Q$.

\begin{proposition}\label{Proposition: Spiked Moment Convergence}
Let $x_1,\ldots,x_n\geq0$ be fixed.
The following limits hold jointly in distribution over $1\leq i\leq k$:
\begin{enumerate}
\item $\displaystyle \lim_{n\to\infty}\sup_{0\leq s\leq t_i}\left|\frac{T^{i;x_i^n}(\lfloor m_n^2(3s/2)\rfloor)}{m_n}-X^{i;x_i}(s)\right|=0.$
\item $\displaystyle \lim_{n\to\infty}\sup_{y>0}\left|\frac{\La^{(y_n,\bar y_n)}_{\th_i}(T^{i;x_i^n})}{m_n}(1-\tfrac12\mbf 1_{\{(y_n,\bar y_n)=(0,0)\}})-\frac12L^y_{t_i}(X^{i;x_i})\right|=0$,
\\jointly in $(y_n,\bar y_n)_{n\in\mbb N}$ as in \eqref{Equation: yn Sequences}.
\item$\displaystyle \lim_{n\to\infty}\left|\frac{\La^{(0,0)}_{\th_i}(T^{i;x_i^n})}{m_n}-2\mf L^0_{t_i}(X^{i;x_i})\right|=0$.
\item $\displaystyle \lim_{n\to\infty}m_n\int_{T^{i;x_i^n}(\th_i)/m_n}^{(T^{i;x_i^n}(\th_i)+1)/m_n}g_i(y)\d y=g_i\big(X^{i;x_i}(t)\big).$
\item The convergences in \eqref{Equation: Skorokhod Noise Terms}.
\item $\displaystyle\lim_{n\to\infty}\sum_{a\in\mbb N_0}\frac{\La^{(a_E,\bar a_E)}_{\th_i}(X^{i;x_i^n})}{m_n}\frac{\xi^E_n(a)}{m_n}
=\frac12\int_{\mbb R_+}L^y_{t_i}(T^{i;x_i})\d W^E(y)$\\
for $E\in\{D,U,L\}$, where, for every $a\in\mbb N_0$, $(a_E,\bar a_E)$ are as in \eqref{Equation: Product (a,a) Notation}.
\end{enumerate}
\end{proposition}
\begin{proof}
Arguing as in Proposition \ref{Proposition: Moment Convergence},
the result follows by using Theorem \ref{Theorem: Split Skorokhod Couplings} to
couple the $T^{i;x^n_i}$ with reflected Brownian motions with variance $2/3$, $\tilde X^{i;x^n_i}$,
and then defining $X^{i;x^n_i}(s):=\tilde X^{i;x^n_i}(3s/2)$, which yields a standard reflected
Brownian motion such that 
$L^y_{3t_i/2}(\tilde X^{i;x_i})=\tfrac32L^y_{t_i}(X^{i;x_i})$ and $\mf L^0_{3t_i/2}(\tilde X^{i;x_i})=\tfrac32\mf L^0_{t_i}(X^{i;x_i})$.
\end{proof}

\subsubsection{Convergence Inside the Expected Value}

We begin with the proof that for every $x_1,\ldots,x_k\geq0$, there is a coupling such that
\begin{align}\label{Equation: Spiked Convergence Inside Expectation}
\lim_{n\to\infty}\prod_{i=1}^kF_{n,t_i}(T^{i;x_i^n})
\,m_n\int_{T^{i;x_i^n}(\th_i)/m_n}^{(T^{i;x_i^n}(\th_i)+1)/m_n}g_i(y)\d y
=\prod_{i=1}^k\mr e^{-\langle L_t(X^{i;x_i}),Q'\rangle-w\mf L^0_{t_i}(X^{i;x_i})} g_i\big(X^{i;x_i}(t)\big)
\end{align}
in probability. Proposition \ref{Proposition: Spiked Moment Convergence} provides a coupling such that
\begin{multline*}
\prod_{i=1}^k\mbf 1_{\{\tau^{(n)}(T^{i;x^n_i})>{\th_i}\}}\left(\prod_{a\in\mbb N}\left(1-\frac{D_n(a)}{m_n^2}\right)^{\La^{(a,a)}_{\th_i}(T^{i;x^n_i})}\right)\\
\cdot\left(\prod_{a\in\mbb N_0}\left(1-\frac{U_n(a)}{m_n^2}\right)^{\La^{(a,a+1)}_{\th_i}(T^{i;x^n_i})}\left(1-\frac{L_n(a)}{m_n^2}\right)^{\La_{\th_i}^{(a+1,a)}(T^{i;x^n_i})}\right)
\end{multline*}
converges in probability to $\prod_{i=1}^k\mr e^{-\langle L_t(X^{i;x_i}),Q'\rangle}$.
Combining this with Proposition \ref{Proposition: Spiked Moment Convergence}-(4),
it only remains to show that 
\[\lim_{n\to\infty}\prod_{i=1}^k\left(1-\frac{(1-w_n)}2-\frac{D_n(0)}{2m_n^2}\right)^{\La_{\th_i}^{(0,0)}(T^{i;x^n_i})}
=\prod_{i=1}^k\mr e^{-w\mf L^0_{t_i}(X^{i;x_i})}.\]
To this effect, the Taylor expansion $\log(1+z)=z+O(z^2)$ yields
\begin{multline*}
\left(1-\frac{(1-w_n)}2-\frac{D_n(0)}{2m_n^2}\right)^{\La_{\th_i}^{(0,0)}(T^{i;x^n_i})}\\
=\exp\left(-\La_{\th_i}^{(0,0)}(T^{i;x^n_i})\left(\frac{(1-w_n)}2+\frac{D_n(0)}{2m_n^2}+O\left(\frac{(1-w_n)^2}4+\frac{D_n(0)^2}{2m_n^4}\right)\right)\right).
\end{multline*}
By Proposition \ref{Proposition: Spiked Moment Convergence}-(3) and Assumption \ref{Assumption: wn},
\[\lim_{n\to\infty}\La_{\th_i}^{(0,0)}(T^{i;x^n_i})\left(\frac{(1-w_n)}2+\frac{D_n(0)}{2m_n^2}\right)=w\mf L^0_{t_i}(X^{i;x_i})\]
and
\[\lim_{n\to\infty}\La_{\th_i}^{(0,0)}(T^{i;x^n_i})\left(\frac{(1-w_n)^2}4+\frac{D_n(0)^2}{2m_n^4}\right)=0\]
almost surely, as desired.

\subsubsection{Convergence of the Expected Value}

Next we prove
\begin{multline}\label{Equation: Spiked Convergence of Expectation}
\lim_{n\to\infty}\mbf E\left[\prod_{i=1}^kF^w_{n,t_i}(T^{i;x_i^n})
\,m_n\int_{T^{i;x_i^n}(\th_i)/m_n}^{(T^{i;x_i^n}(\th_i)+1)/m_n}g_i(y)\d y\right]\\
=\mbf E\left[\prod_{i=1}^k\mr e^{-\langle L_t(X^{i;x_i}),Q'\rangle-w\mf L^0_{t_i}(X^{i;x_i})} g_i\big(X^{i;x_i}(t)\big)\right]
\end{multline}
pointwise in $x_1,\ldots,x_k\geq0$.
Similarly to Section \ref{Subsubsection: Convergence of the Expected Value},
this is done by combining \eqref{Equation: Spiked Convergence Inside Expectation}
with the uniform integrability estimate
\begin{align}\label{Equation: Spiked Uniformly Integrability}
\sup_{n\geq N}\mbf E\left[\big|F^w_{n,t_i}(T^{i;x_i^n})\big|^{2k}\right]<\infty,\qquad 1\leq i\leq k
\end{align}
for large enough $N$.
To achieve this we combine Proposition
\ref{Proposition: Tail of Horizontal at Zero} and the following:

\begin{proposition}\label{Proposition: Spiked Local Time Uniform Exponential Moments}
Let $\th=\th(n,t)=\lfloor m_n^2t\rfloor$ for some $t>0$. For every $C>0$ and $1\leq q<3$,
\[\sup_{n\in\mbb N,~x\geq0}\mbf E\left[\exp\left(\frac{C}{m_n}\sum_{a\in\mbb N}\frac{\La^{a}_{\th}(T^{x^n})^q}{m_n^q}\right)\right]<\infty.\]
\end{proposition}
\begin{proof}
If we couple $X$ and $S$ as in Definition \ref{Definition: Time Change Coupling},
then we see that
\begin{align*}
\mbf E\left[\exp\left(\frac{C}{m_n}\sum_{a\in\mbb N}\frac{\La^{a}_{\th}(T^{x^n})^q}{m_n^q}\right)\right]
&\leq\mbf E\left[\exp\left(\frac{2^{q-1}C}{m_n}\sum_{a\in\mbb Z\setminus\{0\}}\frac{\La^{a}_{\rho^{x^n}_\th}(S^{0})^q}{m_n^q}\right)\right]\\
&\leq\mbf E\left[\exp\left(\frac{2^{q-1}C}{m_n}\sum_{a\in\mbb Z}\frac{\La^{a}_{\th}(S^{0})^q}{m_n^q}\right)\right].
\end{align*}
Thus Proposition \ref{Proposition: Spiked Local Time Uniform Exponential Moments} follows directly
from Proposition \ref{Proposition: Local Time Uniform Exponential Moments}.
\end{proof}

Indeed, the arguments of Section \ref{Subsubsection: Convergence of the Expected Value}
show that the contribution of the terms of the form \eqref{Equation: Spiked Random Walk Functional 2}
and \eqref{Equation: Spiked Random Walk Functional 3} to \eqref{Equation: Spiked Uniformly Integrability}
can be controlled by Proposition \ref{Proposition: Spiked Local Time Uniform Exponential Moments}. Thus,
it suffices to prove that for every $C>0$, there is some $N\in\mbb N$ large enough so that
\begin{align}\label{Equation: Spiked Moments}
\sup_{n\geq N,~x\geq0}\mbf E\left[\left|1-\frac{(1-w_n)}2-\frac{D_n(0)}{2m_n^2}\right|^{C\La_{\th_i}^{(0,0)}(T^{i;x^n_i})}\right]<\infty.
\end{align}
By using the bound $|1-z|\leq\mr e^{|z|}$, it suffices to control the exponential moments of
\begin{align}\label{Equation: Spiked Moments 1}
\La_\th^{(0,0)}(T^{x^n})|1-w_n|
\qquad\text{and}\qquad
\frac{\La_\th^{(0,0)}(T^{x^n})|D_n(0)|}{m_n^2}.
\end{align}

We begin with the first term in \eqref{Equation: Spiked Moments 1}. According to Proposition \ref{Proposition: Tail of Horizontal at Zero},
for every $C>0$,
\[\sup_{n\in\mbb N,~x\geq0}\mr E\left[\mr e^{C\La_\th^{(0,0)}(T^{x^n})/m_n}\right]<\infty.\]
Thus, given that $|1-w_n|=O(m_n^{-1})$ by Assumption \ref{Assumption: wn}, we conclude that 
\[\sup_{n\in\mbb N,~x\geq0}\mr E\left[\mr e^{C\La_\th^{(0,0)}(T^{x^n})|1-w_n|}\right]<\infty.\]
Let us now consider the second term in \eqref{Equation: Spiked Moments 1}. By the tower property and
Assumption \ref{Assumption: Robin}, there exists $\bar C,\bar c>0$ independent of
$n$ such that
\[\mbf E\left[\mr e^{C(\La_\th^{(0,0)}(T^{x^n})/m_n^{3/2}\big)(|D_n(0)|/m_n^{1/2})}\right]
\leq\bar C\mbf E\left[\mr e^{\bar c\,(C^2/m_n)\big(\La_\th^{(0,0)}(T^{x^n})/m_n\big)^2}\right].\]
Since $\bar c\,(C^2/m_n)\to0$, it follows from
Proposition \ref{Proposition: Tail of Horizontal at Zero} that
\[\sup_{n\geq N,~x\geq0}\mbf E\left[\mr e^{\bar c\,(C^2/m_n)\big(\La_\th^{(0,0)}(T^{x^n})/m_n\big)^2}\right]<\infty\]
for large enough $N$,
concluding the proof of \eqref{Equation: Spiked Moments}.

\subsubsection{Convergence of the Integral}

With \eqref{Equation: Spiked Convergence of Expectation} established, once more
we aim to prove that \eqref{Equation: Spiked Mixed Moment Prelimit} converges to
\eqref{Equation: Spiked Mixed Moment Limit} by using \cite[Theorem 2.24]{FonsecaLeoni}.
Similarly to Section \ref{Subsubsection: Convergence of Integral},
for this we need upper bounds of the form
\begin{align}\label{Equation: Spiked Vitali for Small x}
\mbf E\left[|F^w_{n,t_i}(T^{i;x^n})|^k\right]^{1/k}
\leq c_1\left((1+x)^{-c_2\theta}+\mr e^{-c_3n^{2\mf d}}\right),\qquad x\in[0,n^{1-\eps}/m_n)
\end{align}
and
\begin{align}\label{Equation: Spiked Vitali for Large x}
\mbf E\left[|F^w_{n,t_i}(T^{i;x^n})|^k\right]^{1/k}
\leq \mr e^{-\bar c_4 n^{\al(1-\mf d)-\al\eps}}
+\mr e^{-\bar c_5 n^{2\mf d}},\qquad x\in[n^{1-\eps}/m_n,(n+1)/m_n),
\end{align}
where $\eps,c_1,c_2,c_3,\bar c_4,\bar c_5>0$ are independent of $n$
and $\theta>0$ is taken large enough so that $(1+x)^{-c_2\theta}$ is integrable.

We begin with $x\in[0,n^{1-\eps}/m_n)$.
Replicating the analysis leading up to \eqref{Equation: Vitali Convergence Noise 1}
and \eqref{Equation: Vitali Convergence Potential 1} leads to bounding
$\mbf E\left[|F^w_{n,t_i}(T^{i;x^n})|^k\right]^{1/k}$ by the product of
the following five terms:
\begingroup
\allowdisplaybreaks
\begin{align}
\label{Equation: Spiked Vitali Convergence Boundary}
&\mbf E\left[\left|1-\frac{(1-w_n)}2-\frac{D_n(0)}{2m_n^2}\right|^{7k\La_{\th_i}^{(0,0)}(T^{i;x^n_i})}\right]^{1/7k},\\
\label{Equation: Spiked Vitali Convergence Noise 1.1}
&\prod_{E\in\{U,L\}}\mbf E\left[\prod_{a\in\mbb N_0}\left|1-\frac{\xi^E_n(a)}{m_n^2-V^E_n(a)}\right|
^{7k\La^{(a_E,\bar a_E)}_{\th_i}(T^{i;x^n})}
\right]^{1/7k},\\
\label{Equation: Spiked Vitali Convergence Noise 1.2}
&\mbf E\left[\prod_{a\in\mbb N}\left|1-\frac{\xi^D_n(a)}{m_n^2-V^D_n(a)}\right|
^{7k\La^{(a,a)}_{\th_i}(T^{i;x^n})}
\right]^{1/7k},\\
\label{Equation: Spiked Vitali Convergence Potential 1.1}
&\prod_{E\in\{U,L\}}\mbf E\left[\prod_{a\in\mbb N_0}\left|1-\frac{V_n^E(a)}{m_n^2}\right|^{7k\La^{(a_E,\bar a_E)}_{\th_i}(T^{i;x^n})}\right]^{1/7k},\\
\label{Equation: Spiked Vitali Convergence Potential 1.2}
&\mbf E\left[\prod_{a\in\mbb N}\left|1-\frac{V_n^D(a)}{m_n^2}\right|^{7k\La^{(a,a)}_{\th_i}(T^{i;x^n})}\right]^{1/7k}.
\end{align}
\endgroup
Suppose without loss of generality that $V^D_n$ satisfies \eqref{Equation: Log Potential Lower Bound}.
\eqref{Equation: Spiked Vitali Convergence Boundary} can be controlled with \eqref{Equation: Spiked Moments};
\eqref{Equation: Spiked Vitali Convergence Noise 1.1} and \eqref{Equation: Spiked Vitali Convergence Noise 1.2}
can be controlled with Proposition \ref{Proposition: Spiked Local Time Uniform Exponential Moments};
and \eqref{Equation: Spiked Vitali Convergence Potential 1.1} can be controlled with
\eqref{Equation: Potential Absolute Bounds}.
For \eqref{Equation: Spiked Vitali Convergence Potential 1.2}, up to a constant independent of $n$,
we get from \eqref{Equation: Log Potential Lower Bound} the upper bound
\begin{align}\label{Equation: Spiked Vitali Convergence Potential 1.2-2}
\mbf E\left[\exp\left(-\frac{7k\theta}{m_n^2}\sum_{a\in\mbb N}\log(1+|a|/m_n)\La^{(a,a)}_{\th_i}(T^{i;x^n})\right)\right]^{1/7k}.
\end{align}
Let us couple $T^{i;x^n}$ and $S^{x^n}=x^n+S^0$ as in Definition \ref{Definition: Time Change Coupling}.
The same argument used to control
\eqref{Equation: Vitali Convergence Potential 1-2} implies that
\eqref{Equation: Spiked Vitali Convergence Potential 1.2-2} is bounded above
by the product of
\begin{align}
\label{Equation: Spiked Vitali Convergence Potential 1.2-3}
&\mbf E\left[\exp\left(-\frac{14k\theta \log(1+x)}{m_n^2}
\sum_{a\in\mbb Z\setminus\{0\}}\La^{(a,a)}_{\rho^{x^n}_{\th_i}}(S^0)\right)\right]^{1/14k},\\
\label{Equation: Spiked Vitali Convergence Potential 1.2-4}
&\mbf E\left[\exp\left(\frac{14k\theta}{m_n^2}\sum_{0\leq u\leq\rho^{x^n}_{\th_i}}\frac{|S^0(u)|}{m_n}\right)\right]^{1/14k}.
\end{align}
Since $\rho^{x^n}_{\th_i}\leq\th_i$, we can prove that \eqref{Equation: Spiked Vitali Convergence Potential 1.2-4}
is bounded by a constant independent of $n$ by using \eqref{Equation: Range Exponential Moments} directly.
As for \eqref{Equation: Spiked Vitali Convergence Potential 1.2-3}, we have the following proposition:

\begin{proposition}\label{Proposition: Reflected Binomial Concentration}
Let $\th=\th(n,t):=\lfloor m_n^2 t\rfloor$ for some $t>0$.
For every $x\geq0$, let us couple $T^{x^n}$ and $S^{x^n}:=x^n+S^0$
as in Definition \ref{Definition: Time Change Coupling}.
For small enough $\nu>0$,
there exists $C,c>0$ independent of $x$ and $n$ such that
\[\sup_{x\geq0}\mbf P\left[\sum_{a\in\mbb Z\setminus\{0\}}
\La^{(a,a+b)}_{\rho^{x^n}_\th}(S^0)<\nu m_n^2\right]\leq C\mr e^{-cm_n^2}
,\qquad b\in\{-1,0,1\}.\]
\end{proposition}
\begin{proof}
By Proposition \ref{Proposition: Tail of Horizontal at Zero},
for any $0<\de<1$, we can find $\bar C,\bar c>0$ such that
\[\sup_{x\geq0}\mbf P\left[\La_\th^{(0,0)}(T^{x^n})\geq \de\th\right]\leq \bar C\mr e^{-\bar cm_n^2}.\]
Given that $\th-\rho^{x^n}_{\th}\leq\La_\th^{(0,0)}(T^{x^n})$,
it suffices to prove that
\[\sup_{x\geq0}\mbf P\left[\sum_{a\in\mbb Z\setminus\{0\}}
\La^{(a,a+b)}_{(1-\de)\th}(S^0)<\nu m_n^2\right]\leq C\mr e^{-cm_n^2}
,\qquad b\in\{-1,0,1\}\]
for large enough $N$. This follows by Hoeffding's inequality.
\end{proof}

By arguing as in the passage following \eqref{Equation: Uniform Integrability Hoeffding},
Proposition \ref{Proposition: Reflected Binomial Concentration}
implies that \eqref{Equation: Spiked Vitali Convergence Potential 1.2-3}
is bounded above by $c_1\Big((1+x)^{-c_2\theta}+\mr e^{-c_3n^{2\mf d}}\Big)$
for $c_1,c_2,c_3>0$ independent of $n$ (and $c_2$ independent of $\theta$),
hence \eqref{Equation: Spiked Vitali for Small x} holds.

We now prove \eqref{Equation: Spiked Vitali for Large x}.
Let $x\in[n^{1-\eps}/m_n,(n+1)/m_n)$.
Assuming without loss of generality that
$V^D_n$ satisfies \eqref{Equation: Polynomial Potential Lower Bound},
by arguing as in Section \ref{Subsubsection: Convergence of Integral},
we get that $\mbf E\left[|F^w_{n,t_i}(T^{i;x^n})|^k\right]^{1/k}$ is bounded
by the product of the four terms
\begingroup
\allowdisplaybreaks
\begin{align*}
&\mbf E\left[\left|1-\frac{(1-w_n)}2-\frac{D_n(0)}{2m_n^2}\right|^{5k\La_{\th_i}^{(0,0)}(T^{i;x^n_i})}\right]^{1/5k}
\cdot\mbf E\left[\prod_{a\in\mbb N}\left(1-\frac{\kappa(Cn^{1-\eps}/m_n)^\al}{m_n^2}\right)^{5k\La^{(a,a)}_{\th_i}(T^{i;x^n})}\right]^{1/5k}\\
&\cdot\mbf E\left[\prod_{a\in\mbb N}\left(1+\frac{2|\xi^D_n(a)|}{m_n^2}\right)^{5k\La^{(a,a)}_{\th_i}(T^{i;x^n})}\right]^{1/5k}
\cdot\prod_{E\in\{U,L\}}\mbf E\left[\prod_{a\in\mbb N_0}\left(1+\frac{|\xi^E_n(a)|}{m_n^2}\right)^{5k\La^{(a_E,\bar a_E)}_{\th_i}(T^{i;x^n})}\right]^{1/5k}.
\end{align*}
\endgroup
By combining Propositions \ref{Proposition: Spiked Local Time Uniform Exponential Moments} and
\ref{Proposition: Reflected Binomial Concentration} with
\eqref{Equation: Spiked Moments}, the same arguments used in Section \ref{Subsubsection: Convergence of Integral}
yields \eqref{Equation: Spiked Vitali for Large x}, concluding the proof
of the convergence of moments.

\subsection{Step 2: Convergence in Distribution}

The convergence in joint distribution follows from the convergence of moments by using the same
truncation/stochastic dominance argument Section \ref{Subsection: Distribution},
thus concluding the proof of Theorem \ref{Theorem: Main Robin}.

\bibliographystyle{plain}
\bibliography{Bibliography}
\end{document}